\documentclass[11pt]{article}
\usepackage{mathrsfs}
\usepackage{}

\usepackage{bbding}
\usepackage{pifont}
\usepackage{amsfonts}
\usepackage{amssymb,amsmath }
\usepackage{graphicx}
\usepackage{amsmath}
\usepackage{amsthm}
\usepackage{amssymb}
\usepackage{concrete}
\usepackage{amscd}

\def\beq{\begin{equation}}
\def\eeq{\end{equation}}

\def\nd{\noindent}

\def\<{\leq}
\def\>{\geq}

\newtheorem{thm}{Theorem}[section]
\newtheorem{lem}{Lemma}[section]
\newtheorem{prop}{Proposition}[section]
\newtheorem{cor}{Corollary}[section]
\newtheorem{defi}{Definition}[section]
\newtheorem{rem}{Remark}[section]

\begin{document}

\title{Quiver data of representations of Hecke algebras and Coxeter groups }
\author{Zhi Chen}
\date{}
\maketitle

\begin{abstract}

  Let $\Gamma$ be a Coxeter matrix and $q$ be a generic parameter, let $H_q (\Gamma)$ be the Iwahori-Hecke algebra.  We propose a  concept $H-$representation of type $[\Gamma;q]$, which are representations of certain quiver satisfying several sets of relations. We show type $[\Gamma;q]$ $H-$representation   and  $H_q (\Gamma)-$representation almost determine each other. By using this we construct all "nonintersecting type" $H_q (\Gamma)-$representations . $H-$representation is compatible with the Kazhdan-Lusztig theory . Some relations between $H-$representation and $W-$graph are proposed, which produce an equation system for a $W-$graph to giving a $W-$representation of $H_q (\Gamma)$ . With the intention of constructing general type $H_q (\Gamma)-$representations we introduce the concept $H^2 -$representation , which is another quiver representation hidden in a $H-$representation. A program to construct general type $H-$representations from a chosen $[\Gamma;q]-$realizable $H^2 -$representation is developed ,by using it we prove that if $\Gamma$ satisfies  $m_{i,j}\geq 5 $ for any $i\neq j$, then any $H^2 -$representation is $[\Gamma;q]-$realizable.

\end{abstract}

\section*{Introduction}

The Iwahori-Hecke algebra $H_q (\Gamma)$ ( $\Gamma=[m_{i,j}]_{i,j\in I}$ is a Coxeter matrix  ) first appeared in the work of Iwahori and Matsumoto on the representation theory of finite groups of Lie type \cite{IM}, where they arise as endomorphism algebras of induced representations from Borel subgroups. These algebras can be understood abstractly as deformations of the group algebra of Coxeter groups when the parameter $q$ are set to $1$.  A transformative moment in the subject arrived with Kazhdan and Lusztig's seminal 1979 paper \cite{KL2}, which introduced a canonical basis and the Kazndan-Lusztig polynomials for Hecke algebras and defined the notion of cells in Coxeter groups. This work revealed deep connections with the representation theory of Lie algebras, the geometry of flag varieties and intersection homology, made Iwahori-Hecke algebra one of the central objects of study in Lie theory. From then on, Hecke algebras, particularly those of finite type and affine type have been studied intensively by many scholars \cite{Ar} \cite{Ch}\cite{EW}\cite{KL1}\cite{Lu1}\cite{Lu2}\cite{OS}\cite{Ra}\cite{So}\cite{Xi2}, revealing many profound and fascinating features of these algebras. Yet currently, very little is known about representations of Hecke algebras of other infinite Coxeter types compared to finite and affine types.

 The constructions of this paper applies to $H_q (\Gamma)$ of all coxeter types, it starts with one of our main observations that a representation $\rho $ of $H_q (\Gamma)$ and certain quiver representation $\Psi_{\rho}$ almost determine each other. The relevant quiver $\mathbb{Q}_I$ is not the Dynkin diagram, but the quiver of the complete directed graph with $I$ (the index set in the definition of $\Gamma$ ) as the set of vertexes. The relevant quiver representation is a $\mathbb{Q}_I -$representations satisfying three sets of relations $(T_{i,j})$ ,$(T^{i,j}_k )$ and $(T^i _{j,k})$, which is called as a \textbf{$H-$representation of type $[\Gamma;q]$} in this paper. Explicitly, suppose $\Gamma = [m_{i,j}]_{i,j\in I}$ be a Coxeter matrix. Suppose the Hecke algebra $H_q (\Gamma)$ and the quiver  $\mathbb{Q}_I$ is defined as in the beginning of section 2.  For a $H_q (\Gamma)-$representation $(V,\rho)$ , the associated $\mathbb{Q}_I -$representation
 $\Psi_{\rho}$ is defined as follows. $\Psi_{\rho}( o_i  )= V_i $, which is the $q-$eigenspace of the morphism $\rho(\sigma_i)$. The
 "jumping operator" $ \Psi_{\rho}( J^i _j) : V_i \rightarrow V_j $ is defined as in Definition 2.1. Then $\Psi_{\rho}$ is proved to be a $H-$representation of type $[\Gamma; q]$ (Definition 2.3) in Theorem 2.2 and Theorem 2.4. Inversely with a $H-$representation $\Psi$ of type $[\Gamma ;q]$, we construct a $H_{q}(\Gamma)-$representation $\rho_{\Psi}$ as in Lemma 2.3$\sim$2.5 and Theorem 2.3. The following facts are established.

 \noindent $\bullet$$\Psi_{\rho_{\Psi}} = \Psi  $ and there is a natural
 morphism $f: \rho_{\Psi_{\rho}} \rightarrow \rho$ whose kernel $\ker(f)$ and cokernel $coker(f)$ are both trivial representations (Lemma 2.7).

 So we could say $\Psi_{\rho}$ and $\rho$ almost determine each other, and call $\Psi_{\rho}$ as the quiver data of $\rho$. In many aspects $\Psi_{\rho}$ can be looked as a condensation of $\rho$, containing almost the same information of $\rho$, but being easier to be constructed and studied. Section 3 displays an immediate outcome of this construction, we construct all $H_q (\Gamma)-$representations of "nonintersecting type" ( $V_i \cap V_j =0 $ for any $i\neq j \in I$   ). The construction of these cases are very easy because in these cases the triangular relations $(T^{i}_{j,k})$ and $(T^{i,j}_k )$ disappear. The nonintersecting representations form a particular family of $H_q (\Gamma)-$representaitons including the famous Tits representation, having relatively simple structure and a clear definition of deformation.

  In section 4, we  show this construction is naturally compatible with Kazhdan-Lusztig theory in several aspects. Firstly if one consider the associated quiver representation $\Psi$ of the regular representation of $H_q (\Gamma)$ , the canonical basis naturally equip every space $\Psi (i)$ with a basis (Proposition 4.2), and by using these bases the operators $\Psi ( J^i _j )$ acquire possibly the simplest matrix representations,  with those constants $\mu ^i _{\alpha ,\beta}$ appearing in its entries ( Theorem 4.2 ). Secondly with the data of a $W-$graph, on one hand one can define a sets morphisms $\tau_i $ as in many literatures, on the other hand  one can define a $\mathbb{Q}_I -$representation $\Psi$ and Theorem 4.3 states $\{ \tau_i \} _{i\in I}$ constitute a $H_q (\Gamma)-$representation if and only if $\Psi$ constitute a $H-$representation of type $[\Gamma ;q]$. For the datum of a $W-$graph indeed form a $W-$graph, the constants
   $\{ \mu^i _{\alpha ,\beta} \}$ should satisfy some complicate algebraic equations coming from the Artin relations of $H_q (\Gamma)$. By using Theorem 4.3, we write down these algebraic relations explicitly in a quite compact form.

   Half of this paper is devoted to construction of $H-$representations of general type. Unlike the nonintersecting type, it seems very hard to choose jumping operators $\{ A^i _j \}_{i\neq j\in I}$ satisfying the eigenvalue conditions $(T_{i,j})$ and the triangular conditions $(T^i _{j,k}),  (T^{i,j}_k )$ simultaneously,even for rank 3 Coxeter matrix $\Gamma$.  The way we found is to introduce another quiver representation $\Theta_{\Psi}$ "inside" a $H-$representation $\Psi$, which is called as the basic data of $\Psi= (V_i , A^i _j )$ (together with two sets of  indicated subspaces ) as follows. The relevant quiver $\mathbb{Q}^2 _I $ is defined in Definition 6.1. The set of vertexes of $\mathbb{Q}^2 _I$ is $I^2 =\{ {i,j}\subset I | i\neq j \}$, that is, the set of unordered pairs of elements of $I$. For any ${i,j}, {k,l} \in I^2$, there
   is a unique edge $J^{i,j} _{k,l}$ from $\{ i,j\}$ to $\{ k,l\}$.  Then $\Theta_{\Psi}( \{ i,j\}  )$ is set as $V_{i,j}\subset V_i $, that is, the $1-$eigenspace of $A^j _i A^i _j$ ( This space is identical to $V_i \cap V_j $ in the $H_q (\Gamma)-$representation $\rho_{\Psi}$  ). The morphism $\Theta_{\Psi}( J^{i,j}_{k,l}  ): V_{i,j} \rightarrow V_{k,l}$ is certain natural morphism. This quiver representation has the following features.

   \noindent \textit{$\bullet$ Morphisms from $V_{i,j}$ to $V_{k,l}$ defined in several different ways coincide (Lemma 6.1)}.

    \noindent \textit{$\bullet$ $\Theta_{\Psi}$ has no special relevance to the Coxeter type $\Gamma$}.

     \noindent \textit{$\bullet$ $\Theta_{\Psi}$ contains the information about the arrangement of subspaces $\{ V_{i,j} \}_{j:j\neq i}$ in $V_i$ and partial information about the jumping operators $A^i _j$}.

     Except for $\Theta_{\Psi}$, we can also derive two sets of subspaces $\{ U_i \subset \oplus _{j: j\neq i} V_{i,j} \} $ and $\{ W_i \subset \oplus _{j: j\neq i} V_{i,j} \} $ for $i\in I$. The combination $( \Theta_{\Psi} ; \{ U_i \} ; \{ W_i \}    )$ is defined ( Definition 6.4) as a "decorated $\mathbb{Q}^2 _{I}-$representaion", and is called as the basic data of the $H-$representation $\Psi$.
     Not every such combination $( \Theta ; \{ U_i \} ; \{ W_i \}    )$ come from a $H-$representation. In Theorem 6.2, several necessary conditions for this combination to be coming from some $H-$representations are found.  In Definition 6.4 we introduce the term \textbf{$H^2 -$repsentation} as such a combination satisfying these necessary relations. Furthermore,  a $H^2 -$representation is called as " $[\Gamma ;q]-$realizable " if it is induced from a $H-$representaion of type $[\Gamma ;q ]$.  We can now decompose the construction of $H-$representations of a prescribed type $[\Gamma ;q]$ into two steps.

     \noindent (1) \textit{Determine all $[\Gamma ;q]-$ realizable $H^2 -$representations}.

     \noindent (2) \textit{For each $[\Gamma ;q]-$ realizable $H^2 -$representation $( \Theta ; \{ U_i \} ; \{ W_i \}    )$, construct all  $[\Gamma ;q]$ type $H-$representations having it as the basic data}.

     In section 7 we propose a procedure to construct a $H-$representation of prescribed type $[\Gamma ;q]$ from a chosen $H^2 -$representation $( \Theta ; \{ U_i \} ; \{ W_i \}    )$ by three steps as follows.

     \noindent \textbf{Step 1. Space realizations.} In section 7.1, series of linear spaces $\{ V_i , V^{\#}_i \} $ (together with some accessory constructions as described below ) are constructed. For any $i\in I$, there is a perfect pairing
     $[-,-]_i : V_i \times V^{\#}_i \rightarrow \kappa$, and for every pair $i\neq j \in I$, there are decompositions (these decompositions are important for $H-$representations ) $V_i \cong V_{i,j} \oplus V_{i/j}$ and  $V^{\#} _i \cong V^{\#} _{i,j} \oplus V^{\#}_{i/j}$ such that $V^{\#}_{i,j} = ( V_{i/j}   )^{\pm}$ and $V^{\#}_{i/j} = ( V_{i,j}   )^{\pm}  $. This construction is  consistent with $\Theta$ and is unique "up to adding a free factor".

      \noindent \textbf{Step 2. Partial realization of jumping operators. } Section 7.2$\sim$7,4 are devoted to construction of the jumping operators $A^i _j : V_i \rightarrow V_j$ and $A^{\# i} _j : V^{\#}_i \rightarrow V^{\#}_j $.  The main point of this construction is if we construct $V_i$ and $A^i _j$ together with $V^{\#}_i$ and $A^{\# i}_j $ in the dual side simultaneously and keep them being consistent with each other in some sense, then the triangular relations $(T^{i,j}_k)$ and $(T^{\# i,j}_k )$ for $\{ A^i _j , A^{\# i} _j \}$will be satisfied automatically. Then the rest of the main tasks is to keeping $\{ A^i _j , A^{\# i} _j \}$ satisfying the triangular relations $(T_{i,j})$. For simple reason one only need to construct morphisms  $ B^i _j : V_{i/j} \rightarrow V_{j/i}$ and $B^{\# i}_j : V^{\#}_{i/j}\rightarrow V^{\#}_{j/i}$. To keeping these morphisms be consistent with $\Theta$, we first choose morphisms $\psi^{ i}_j : V^0 _{i/j} \rightarrow V_{j/i}  $ and  $\psi^{\# i}_j : V^{\# 0} _{i/j} \rightarrow V^{\#}_{j/i}  $ being consistent with $\Theta$ then try to extend them, where $V^0 _{i/j} \subset V_{i/j} $ and  $V^{\# 0} _{i/j}\subset V^{\#}_{i/j}$ are certain subspaces.

       \noindent \textbf{Step 3. Extension of matched pair of cross morphisms. } We introduced the conception "\textbf{cross morphism}" to name the 6-tuple $[V_{i/j} , V^{0}_{i/j}; V_{j/i}, V^0 _{j/i}; \psi^i _j , \psi^j _i   ]$. It form a "\textbf{matched pair of cross morphisms}" together with $[V^{\#}_{i/j} , V^{\# 0}_{i/j}; V^{\#}_{j/i}, V^{\# 0} _{j/i}; \psi^{\# i} _j , \psi^{\# j} _i   ]$ in the dual side. Now the remaining task is to find extension $B^i _j : V_{i/j}\rightarrow V_{j/i}$ of $\psi^i _{j}$ such that

\noindent (1) \textit{$(  B^i _j )^{*}$ is a extension of $\psi^{\# j}_i $}.

\noindent (2) \textit{Eigenvalues of $ B^j _i B^i _j  $ belong to  $\{ 4 cos(\frac{a\pi}{m_{i,j}})^2 \frac{1}{(v+v^{-1})^2} \} _{a=1,2,...,[\frac{m_{i,j}-1}{2} ]} $}.

This is purely  a  linear algebra problem. A thorough study of cross morphism in section 7.3 enable us to turn   "existence of such extension" into "simultaneous solvability of several line equation systems" (Theorem 7.4). One can construct such extensions from solutions of these equation systems. It enable us to construct many general type (not nonintersecting type) $H-$representations. 
  By using Theorem 7.4 we obtain a solution of problem (1) for those $\Gamma =[m_{i,j}]_{i,j\in I}$ where $m_{i,j}\geq 5$ for any $i\neq j\in I$.
\begin{thm} (Theorem 7.5) If $\Gamma=[ m_{i,j}]_{i,j\in I} $ such that $m_{i,j}\geq 5$ for any $i\neq j\in I$, then for any generic $q\in\kappa$,  and for any $H^2 -$representation $\Theta$, $\Theta$ is $[\Gamma;q]-$realizable.
\end{thm}

\section{The Quiver data construction}

Let $\Gamma=[m_{i,j}]_{i,j\in I}$ be any Coxeter matrix. Let $H_q(\Gamma)$ and $W_{\Gamma}$ be the corresponding Hecke algebra and Coxeter group respectively as follows.
\begin{defi} (\cite{Hu})
The Hecke algebra $H_q (\Gamma)$ is generated by $\{ \sigma_i \}_{i\in I}  $ with relations:

\noindent(1) $[\sigma_i \sigma_j ... ]_{m_{i,j}} = [\sigma_j \sigma_i ...  ]_{m_{i,j}}$ for any $i\neq j \in I$.

\noindent(2) $( \sigma_i -q  )(\sigma_i +1)=0 $ for any $i\in I$.
\end{defi}
The notation $[xy...]_{m}$ means a word of length $m$, in which the letter $x$ and $y$ appear alternatively.For example
$[xy...]_{5}=xyxyx$.
\begin{defi} (\cite{Hu})
 The Coxeter group $W_{\Gamma} $ is the group generated by $\{ s_i \}_{i\in I}$ with relations:

\noindent(1) $[\sigma_i \sigma_j ... ]_{m_{i,j}} = [\sigma_j \sigma_i ...  ]_{m_{i,j}}$ for any $i\neq j \in I$.

\noindent(2) $( \sigma_i )^2 =1 $ for any $i\in I$.
\end{defi}

 Denote the Coxeter type of dihedral group $D_n$  as $I_n$. Through out this paper we suppose the base field is some field $\kappa$, and suppose $q\in \kappa$ is  a generic parameter in the sense that $H_{q}(I_{m_{i,j}})$ is semisimple for any $i,j\in I$.

 With any finite set $I$, a quiver $\mathbb{Q}_I$ is defined as follows. The vertices of $\mathbb{Q}_I$ is just $I$. For each ordered pair of vertices $\{ i ,j \}$, there is a unique arrow $J^i _j$ from $i $ to $j $. So the set of arrows $\mathbb{Q}_I$ is $\{ J^i _j  \}_{ i\neq j \in I }$.

 Let $(V,\rho)$ be a finite dimensional linear representation of $H_q (\Gamma)$. Then the restriction of $\rho$ to any dihedral parabolic subalgebra of $H_q (\Gamma)$ is completely reducible.For any $i\in I$, let $V_i$ and $V^0 _{i}$ be the eigenspace of $\rho(\sigma_i)$ of eigenvalue $q$ and $-1$ respectively. So $V= V_i \oplus V^0 _{i}$.  Now let $i\neq j\in I$. For any $v\in V_i$, there is $v=v_1 +v_2$ in a unique way such that $v_1 \in V_j $ and $v_2 \in V^0 _{j}$. We set $A^i _j (v)= v_1$. In this way  a linear operator $A^i _j : V_i \rightarrow V_j$ is defined.

\begin{defi}

The \textbf{quiver data of the representation $(V, \rho)$ } is the following representation $\Psi_{\rho}$ of  the quiver $\mathbb{Q}_n$. Where

$\Psi_{\rho}(i) =V_i $ for any $i\in I$, and $\Psi_{\rho}(J^i _j) = A^i _j$ for any $i\neq j\in I$.

\end{defi}

The operators $A^i _j$ and the quiver representations $\Psi_{\rho}$ are main objects in this paper.  We call these $A^i _j$ as \textbf{jumping operators }. Representations of $H_q (I_m )$ are known well.

\begin{thm} (\cite{GP}) Let $q\in \kappa$ be generic such that $H_q (I_m )$ is irreducible. Let

$\tau_j = 4\cos (\frac{j \pi}{m}  )^2 q (1+q )^{-2}$. Up to equivalence, the irreducible representations of $H_{q} (I_m)$ are as follows:

\noindent(1) Let $m$ be even.   There are four one-dimensional representations given by

$ind: \sigma_1 \mapsto q , \sigma_2 \mapsto q $;        $\epsilon_1 : \sigma_1 \mapsto q , \sigma_2 \mapsto -1 $;

$\epsilon: \sigma_1 \mapsto -1 , \sigma_2 \mapsto -1 $;  $\epsilon_2 : \sigma_1 \mapsto -1 ,\sigma_2 \mapsto q $.

and $\frac{m}{2}-1$ irreducible representations  of dimension 2 given by

$\mu_j $: $\sigma_1 \mapsto \left(
                    \begin{array}{cc}
                      -1 &  1+q \\
                      0 & q \\
                    \end{array}
                  \right)
,  \sigma_2 \mapsto \left(
                      \begin{array}{cc}
                        q & 0 \\
                        \tau_j (1+q ) & -1 \\
                      \end{array}
                    \right)
 $, for $1\leq j\leq \frac{m}{2}-1$.

\noindent(2) Let $m$ be odd.   There are two one-dimensional representations given by

$ind: \sigma_1 \mapsto q , \sigma_2 \mapsto q $;        $\epsilon: \sigma_1 \mapsto -1 , \sigma_2 \mapsto -1 $;

and $\frac{m-1}{2}$ irreducible representations  of dimension 2 given by

$\mu_j $: $\sigma_1 \mapsto \left(
                    \begin{array}{cc}
                      -1 &  1+q \\
                      0 & q \\
                    \end{array}
                  \right)
,  \sigma_2 \mapsto \left(
                      \begin{array}{cc}
                        q & 0 \\
                        \tau_j (1+q ) & -1 \\
                      \end{array}
                    \right)
 $, for $1\leq j\leq \frac{m-1}{2}$.

\end{thm}

The following theorem  describe the eigenvalues of the operators  $A^j _i A^i _j$.

\begin{thm} Let $i\neq j \in I$.

\noindent(1) If $m_{i,j}$ is odd, then the operator $A^j _i A^i _j$ is semisimple, with eigenvalues in
$\{1\}\cup \{ 4 cos^2(\frac{a\pi}{m_{i,j}}) \frac{1}{q+q^{-1}+2} \} _{a=1,2,...,[\frac{m_{i,j}-1}{2} ]}$.Notice
$[\frac{m_{i,j}-1}{2} ]=\frac{m_{i,j}-1}{2}$ in these cases.

\noindent(2) If $m_{i,j}$ is even, then the operator $A^j _i A^i _j $ is semisimple, with eigenvalues in
$\{ 1,0\} \cup \{ 4 cos^2(\frac{a\pi} {m_{i,j} }) \frac{1}{q+q^{-1}+2} )  \}_{a=1,2,...,[\frac{m_{i,j}-1}{2} ]}$.
Notice $[\frac{m_{i,j}-1}{2} ]=\frac{m_{i,j}}{2} -1 $ in these cases.
In addition, if $u \in V_i $ is a eigenvector of $A^j _i A^i _j $ of eigenvalue $0$, then $A^i _j (u)=0$.

\noindent(3) $v\in V_i $ is an eigenvector of $A^{j} _i A^{i} _j $ of eigenvalue 1 if and only if $v\in V_i \cap V_j $.

For later convenience, the number $4 cos^2(\frac{a\pi}{m}) \frac{1}{q+q^{-1}+2}$  will be denoted  as $\lambda_{m;q}(a)    $.

\end{thm}

\begin{pf}
For any $v\in V^{0} _i $ and $j\in I$, the following identity

$ v=  \frac{1}{1+q} (\sigma_j +1)(v)-\frac{1}{1+q}(\sigma_j -q)(v)  $

implies $A^{i}_{j} (v) = \frac{1}{1+q} (\sigma_j +1)(v)$, since $ \frac{1}{1+q} (\sigma_j +1)(v)\in V^{0}_j    $ and
$-\frac{1}{1+q}(\sigma_j -q)(v) \in V^{-1} _j $. So
 $A^{j} _i A^{i} _j (v) =\frac{1}{(1+q)^2 } ( \sigma_i +1  )(\sigma_j +1)(v)  $.

Now let $H_{i,j} $ be the subalgebra of $H_{\Gamma}(q)$ generated by $ \{  \sigma_i , \sigma_j  \}  $, which is isomorphic to the
Hecke algebra $H_q (I_{m_{i,j}})$. Let $\psi: H_q (I_{m_{i,j}})\rightarrow H_{i,j} $ be the isomorphism that
$\psi(\sigma_1 )= \sigma_i ,\psi(\sigma_2 )= \sigma_j $. Now restrict the representation
$(V,\rho )$ to be a $H_{i,j}-$representation   $(V, \rho_{i,j}   )$ . Through $\psi$, $V$ become a $H_q (I_{m_{i,j}})$-module $( V,\rho_{i,j}\circ \psi  )$   , on which
$\sigma_1 ,\sigma_2$ act as $\rho_{i,j} (\sigma_i ) ,\rho_{i,j} (\sigma_j)  $ respectively. Here we transform important items in the theorem as items of the $H_q (I_{m_{i,j}})$-module $( V,\rho_{i,j}\circ \psi  )$:   $V_i$, $V_j$, $A^i _j$ and  $A^j _i$ can be explained as
 $\ker(\sigma_1 -q)$, $\ker(\sigma_2 -q)$, $\frac{1}{1+q}(\sigma_2 +1)$ and $ \frac{1}{1+q}(\sigma_1 +1) $ respectively.
 In this way, (1),(2), and (3) of this theorem are translated to statements $(1)^{'} $ , $(2)^{'} $ and $(3)^{'}$, which are statements about  the $H_q (I_{m_{i,j}})$-module $( V,\rho_{i,j}\circ \psi  )$.

 Since $ H_q (I_{m_{i,j}}) $ is semisimple, the representation $\rho_{i,j}\circ \psi$ decompose into a direct sum of irreducible representations $\oplus ^{N} _{l=1} \rho ^{l} _{i,j}$. Every $\rho ^{l} _{i,j}$ is some member in the list in Theorem 1.1. We arrange
 these irreducible components such that for $l=1,2,..., n$,  $\rho ^{l} _{i,j}$ is isomorphic to one dimensional representation $ind$ in this list. Since the operators  $A^{j} _i A^{i} _j (= (\frac{1}{1+q}(\sigma_1 +1) )(\frac{1}{1+q}(\sigma_2 +1)) )$,  $\sigma_1 -q$ and $\sigma_2 -q $  commute with any homomorphism of the module $ \rho_{i,j}\circ \psi $, we have $V_i = \oplus^{N} _{l=1} \ker( \sigma_1 |_{ \rho^{l} _{i,j}} -q )   )  $ and  $V_j = \oplus^{N} _{l=1} \ker( \sigma_2 |_{ \rho^{l} _{i,j}} -q )   )  $. So to prove the statements $(1)^{'}$ and $(2)^{'}$ we only need to prove  $(1)^{'}$ and $(2)^{'}$ in cases that the $H_q (I_{m_{i,j}})$-module is a irreducible module in the list of theorem 1.1, which can be checked case by case.

 For (3) , if $v\in V_i \cap V_j $, then $ A^i _j (v)=v   $ and $A^j _i (v) =v $, so $A^{j} _i A^{i} _j (v)=v$.  Inversely if $A^{j} _i A^{i} _j (v)=v$,  it is easy to see $v\in \oplus^{n} _{l=1} \rho ^{l} _{i,j}  $,   which is just  $V_i \cap V_j $. \\

\end{pf}

The following definition is inspired by this theorem.

\begin{defi}
Let $V,W$ be two $\kappa-$linear space.Let $f\in Hom(V,W), g\in Hom(W,V)$, then the pair $(f,g)$ is called as a couple between $V$ and $W$.

Let $m\in 2\mathbb{Z}^{\geq 0} +1$, then a couple $(f,g)$ between $V$ and $W$ is called as a \textbf{$(m;q)-$couple}, if both $g\circ f$ and $f\circ g$ are semisimple, with eigenvalues in $\{1\}\cup \{ 4 cos(\frac{a\pi}{m})^2 \frac{1}{q+q^{-1}+2} \} _{a=1,2,...,[\frac{m-1}{2} ]}$.

In in addition the eigenvalues of $g\circ f$ and $f\circ g$ are in  $\{ 4 cos(\frac{a\pi}{m})^2 \frac{1}{q+q^{-1}+2} \} _{a=1,2,...,[\frac{m-1}{2} ]}$, then  this $(m;q)-$couple is called  as a \textbf{specialized $(m;q)-$couple}.

Let $m\in 2\mathbb{Z}^{\geq 1} $, a couple $(f,g)$ is called as a $(m; q)-$couple, if
both $g\circ f$ and $f\circ g$ are semisimple, with eigenvalues in $\{1,0 \}\cup \{ 4 cos(\frac{a\pi}{m})^2 \frac{1}{q+q^{-1}+2} \} _{a=1,2,...,[\frac{m-1}{2} ]}$, and for any $0-$eigenvector $v\in V$ of $g\circ f$  $(w\in W$ of $f\circ g )$ , there is $f(v)=0$  $(g(w)=0)$.

In in addition the eigenvalues of $g\circ f$ and $f\circ g$ are in  $\{ 0\} \cup \{ 4 cos(\frac{a\pi}{m})^2 \frac{1}{q+q^{-1}+2} \} _{a=1,2,...,[\frac{m-1}{2} ]}$, then this $(m;q)-$couple is called  as a specialized $(m;q)-$couple.

\end{defi}



\begin{rem}
 In this convention, the order of $f,g$ can't be altered. "$(f,g)$ is a couple between $V$ and $W$ " is equivalent to "$(g,f)$ is a couple between $W$ and $V$   ", but not to "$(g,f)$ is a couple between $V$ and $W$  ".
\end{rem}

\noindent\textbf{Algebraic properties of $(m;q)-$couples}. Denote a $(m;q)$-couple $(f, g)  $ as $\mathscr{C}$, where $f\in Hom(V,W), g\in Hom(W,V)$. Suppose $m\in \mathbb{Z}^{\geq 2}$.

Let $V_{\mathscr{C}, \cap }$  be the $1$-eigenspace of $g\circ f$, $V_{\mathscr{C},0}$ be the $0$-eigenspace of $g\circ f$,
$V_{\mathscr{C},a}$ be the $\lambda_{m;q}(a) $-eigenspace for $1\leq a\leq [\frac{m-1}{2}]$. Notice if $m$ is odd then $V_{\mathscr{C},0}=0$
 since it does not exist. Let $W_{\mathscr{C}, \cap }$, $W_{\mathscr{C},0}$,and $W_{\mathscr{C},a}$ be subspaces of $W$ defined similarly. Let
$V_{\mathscr{C}, \setminus }= ( \oplus^{[\frac{m-1}{2}]} _{a=1}  V_{\mathscr{C}, a} )\oplus V_{\mathscr{C} ,0}  $, so
$V= V_{\mathscr{C},\cap } \oplus V_{\mathscr{C}, \setminus }$. Similarly we define $W_{\mathscr{C}, \setminus }$ which fits in a decomposition
$W= W_{\mathscr{C},\cap } \oplus W_{\mathscr{C}, \setminus }$.  The following lemma can be proved by elementary linear algebra.

\begin{lem}
(1) The restriction of $f$ on $V_{\mathscr{C},\cap }$ and $V_{\mathscr{C}, a}$  are isomorphism to the subspace $W_{\mathscr{C},\cap }$ and $W_{\mathscr{C}, a}$ for any $1\leq a\leq [\frac{m-1}{2}]$.

\noindent(2) The restriction of $f$ on $V_{\mathscr{C},0 }$ is a zero map.

\noindent(3) $V_{\mathscr{C},\setminus }= Im ( g\circ f - id_{V}  )$.

\noindent(4) $f ( V_{\mathscr{C}, \setminus }  )\subset W_{\mathscr{C}, \setminus } $.

\end{lem}

Let $f^{*} \in Hom(W^* , V^* )$ and $g^* \in Hom (V^* , W^* )$ be the conjugates of $f$ and $g$ respectively.

Recall the decomposition $V\cong V_{\mathscr{C},\cap } \oplus  ( \oplus^{[\frac{m-1}{2}]} _{a=1}  V_{\mathscr{C}, a} )\oplus V_{\mathscr{C} ,0}$. Let $V^* _{\mathscr{C},\cap } $  ( $V^* _{\mathscr{C},a } $ and $V^* _{\mathscr{C},0 }  $ ) be the subspace of $V^* $ consisting of those linear functions on $V$ disappearing over direct sum components except  $V _{\mathscr{C},\cap } $ (  $V _{\mathscr{C},a } $ and $V _{\mathscr{C},0 }  $ ) . Subspaces $W^* _{\mathscr{C},\cap } $  , $W^* _{\mathscr{C},a } $, $W^* _{\mathscr{C},0 }  $  and $W^* _{\mathscr{C}, \  }$ are defined similarly. The following lemma is also consequence of elementary linear algebra.

\begin{lem}

$\mathscr{C}^{*} $ is also a $(m;q)$-couple, and there are the following identities.

 $ V^* _{ \mathscr{C}^{*} , \cap    }  = V^* _{\mathscr{C},\cap }  $, $ V^* _{ \mathscr{C}^{*} , a   }  = V^* _{\mathscr{C},a }  $ ,
$ V^* _{ \mathscr{C}^{*} , 0    }  = V^* _{\mathscr{C},0 }  $, and  $ V^* _{ \mathscr{C}^{*} , \setminus   }  = V^* _{\mathscr{C},\setminus }  $.

$\mathscr{C}^{*} $ is called as the dual of $\mathscr{C}$ (as a $(m;q)-$couple ).

\end{lem}

Now we ask if  a  $H_{\Gamma}(q)$ representation could be constructed from a representation of the quiver  $\mathbb{Q}_I$ satisfying
suitable conditions, at least those conditions implied in theorem 1.2.

 Let $\Psi$ be a representation of the quiver $\mathbb{Q}_I$. Denote $\Psi(i)$ as $\bar{V}_i$, and $\Psi(J^i _j )$ as $A^i _j$. Denote the $1$-eigenspace of $A^j _i A^i _j  $ in $\bar{V}_i $ as $\bar{V}_{i,j}$.   Inspired by Theorem 1.2, we consider such a $\Psi$ satisfying the following two set of conditions.

\noindent $\bullet$ $(T^{(m_{i,j};q)}_{i,j}) $ (  for any $i\neq j\in I$  ):  $ ( A^i _j , A^j _i )$ is a $(m_{i,j} ;q)-$couple between $\bar{V}_i $ and $\bar{V}_j$.

\noindent $\bullet$ $(T^{i.j}_k)$( for any three different elements  $i,j,k \in I$  ):  for any $v\in \bar{V}_{i,j} $ , $A^i _k (v)= A^j _k A^i _j (v) $.
Here is one of the main definition of this paper.

\begin{defi} ($H-$representation)  Let $\Gamma = ( m_{i,j} ) _{i,j \in I}$ be any Coxeter system. Let $\kappa$ be a field and $q \in \kappa$. Then
a representation  $\Psi$ of $\mathbb{Q}_{I}$ is called a \textbf{$H-$representation of type $(\Gamma;q )$} if it satisfy the conditions $(T^{(m_{i,j};q)}_{i,j})$ for any $i\neq j\in I$, the conditions $(T^{i,j} _k ) $ for any $i,j,k \in I$ different with each other, and the following conditions:

$(T^k _{i,j}):$ $Im(A^k _j -A^i _j A^k _i )\subset Im(id_{V_j} -A^i _j A^j _i ),$ for any $i,j,k \in I$ different with each other.

\end{defi}

\noindent\textbf{Construction of a $H_q (\Gamma)-$representation from a $H-$representation.}

 A $H_{\Gamma}(q)-$representation is constructed  from a  $H-$representation  $\Psi$  of type $(\Gamma;q)$ as follows. The following settings and lemmas are preparatory works for Theorem 1.3.

 Assume notations related to $\Psi$ as before.  Let $\bar{V}= \oplus^{n} _{i=1} \bar{V}_i $.  For each pair $1\leq i\neq j \leq n$, let $V^0 _i [j]=\{ v-A^j _i (v) \} _{v\in \bar{V}_j} ,$
and set $\bar{V}^{0}_i = \oplus_{j\neq i} V^0 _i [j]$.  Then  $\bar{V} \cong \bar{V}_i \oplus \bar{V}^0 _{i}$ for any $i$.

For each pair $1\leq i\neq j\leq n$, and for any $a\in [1, [\frac{m_{i,j}-1}{2}]]$, let $\bar{V}_{i/j}(a)$ be  the $\lambda_{m_{i,j};q} (a)$-eigenspace of $A^j _i A^i _j$.  let $\bar{V}_{i/j}(0)$ be the $0$-eigenspace of $A^j _i A^i _j$. Notice if
$m_{i,j}$ is odd, then $\bar{V}_{i/j}(0)=0$.  Denote the direct sum
 $ (  \oplus^{ [\frac{m_{i,j}-1}{2}   ]   } _{a=1}  \bar{V}_{i/j}(a) )\oplus \bar{V}_{i/j} (0)   $ as  $\bar{V}_{i/j}$. Let $\bar{V}_{i,j}$ be the $1$-eigenspace of $A^j _i A^i _j$. Under these assumptions there is a decomposition $\bar{V}_i = \bar{V}_{i,j}\oplus \bar{V}_{i/j} $, which is important for later discussing and are called as \textbf{canonical decompositions}. Denote the projection operator from $\bar{V}_i$ to $\bar{V}_{i,j}$ with respect to this direct sum as \textbf{$p^{i} _{ij}$}, and denote the natural injection from $\bar{V}_{i,j}$ to $\bar{V}_i$ as \textbf{$J^{i,j} _i $}.

\begin{rem} Here is an explanation of the condition $(T^k _{i,j})$.  Since $Im (id_{V_j } - A^i _j A^j _i )$ is the sum of the eigenspaces of $A^i _j A^j _i$
 of eigenvalues except $1$, so  $Im (id_{V_j } - A^i _j A^j _i )= \bar{V}_{j/i}$. Since $A^i _j ( \bar{V}_{i/j}  )\subset \bar{V}_{j/i}$ and $A^i _j (V_{i,j})= V_{j,i}$, the meaning of $(T^k _{i,j})$ can be understood as: for any $v\in \bar{V}_k$,
  $A^i _j ( p^i _{i,j} (  A^k _i (v)    )    )= p^j _{j,i} ( A^k _j (v)   )   $. That is, the component of $A^k _i (v)$ in $V_{i,j }$ with respect to the canonical decomposition $\bar{V}_i = \bar{V}_{i,j} \oplus \bar{V}_{i/j}$ is identified with the component of $A^k _j (v)$ in $\bar{V}_{j,i}$ through $A^i _j$. So $(T^k _{i,j})$ is equivalent to $(T^k _{j,i})$.
  \end{rem}

  Denote the subspace of $\bar{V}$ spanned by those vectors $ v-A^j _i (v) $ for all pairs $i\neq j$ and all vectors $v\in \bar{V}_{j,i} $ as $W$.

 The notations $\bar{V}_{i,j}, \bar{V}_{i/j} , \bar{V}_{i/j}(a)$ and $ \bar{V}_{i/j}(0) $ will be used through out this paper. If the
 $(m_{i,j} ;q)$-couple $( A^i _j , A^j _i  )$ is denoted as $\mathscr{C}$, then these spaces are just subspaces
 $ (\bar{V}_i )_{\mathscr{C}, \cap } $ , $ (\bar{V}_i )_{\mathscr{C}, \setminus} $ , $(\bar{V}_i )_{\mathscr{C}, a }$  and
 $(\bar{V}_i )_{\mathscr{C}, 0 }$  of $\bar{V}_{i}$ by using notations in Lemma 1.1.

Let $\iota _i \in GL( \bar{V})$ be the linear automorphism of $\bar{V}$ acting on $\bar{V}_i$ as $q id_{\bar{V}_i }$ and acting on $\bar{V}^0 _i $ as $-id_{\bar{V}^0 _i }$. So $\iota_i $ satisfies the Hecke algebra relation $(\iota_i +1)(\iota_i -q)=0$.


\begin{lem} Assume notations as above, then  $W\subset \bar{V}^0 _i $ for any $i$.

\end{lem}

\begin{pf}
Consider the vector $v-A^j _i (v)$ for some pair $i\neq j$ and some $v\in \bar{V}_{ji}$. By definition it is in $\bar{V}^0 _i $. Now since $v\in \bar{V}_{ji}$, there is $A^j _i (v) \in \bar{V}_{ij}$, and $ v= A^i _j A^j _i (v) $, so $v-A^j _i (v)=  A^i _j A^j _i (v)-A^j _i (v) \in \bar{V}^0 _j$.   For $k\neq i,j$,
 by $(T^{i,j}_k )$, there is $A^j _k (v) = A^i _k A^j _i (v)$. So we have $v-A^j _i (v)= v-A^j _k (v) + A^i _k A^j _i (v) -A^j _i (v) \in \bar{V}^0 _k $ and the proof is complete.

\end{pf}

\begin{lem}
For any $i\neq j$, $( \bar{V}_i \oplus \bar{V}_j  ) \cap W = \kappa < v-A^i _j (v)   >_{v\in \bar{V}_{i,j}}$.

\end{lem}

\begin{pf} First we have $\kappa <v-A^i _j (v) >_{v\in \bar{V}_{i,j} } \subset ( \bar{V}_i \oplus \bar{V}_j ) \cap W$ by definition of $W$.  Now
 suppose $v+w\in ( \bar{V}_i \oplus \bar{V}_j  ) \cap W$, for $v\in \bar{V}_i$ and $w\in \bar{V}_j$. Since $v+w$ is in $W$, by definition of $W$ there is
 $v+w = \Sigma _{k\neq l} ( u_{k,l} - A^k _l (u_{k,l} ) )$, where $u_{k,l}$ is certain element in $ \bar{V}_{k,l}$.

 Now for any $i\in I$ we define a projection $p_i : \bar{V} \rightarrow \bar{V}_i $ by:

 \[
 p_i (v) = \begin{cases}
 v, &\text{ when $v\in \bar{V}_i ,$ }\\
 A^k _i (v), &\text{when $v\in \bar{V}_k $ and $k\neq i$.   }\\

\end{cases}
\]

 Then come back to the element $v+w$. By  $(T^{i,j} _k )$, there is
  $p_j ( u_{k,l} -A^k _l (u_{k,l})   )=0 $ for any pair $k \neq l$.  So  $p_j (v+w)=0 $.  On the other hand, there is
  $p_j (v) = \bar{A}^i _j (v) $ and $ p_j (w) =w $, so  $p_j (v+w) = A^i _j (v) +w $. So there is $A^i _j (v) = -w$. Similarly by considering $p_i (v+w)$ we can prove $\bar{A}^j _i (w)= -v $. By summarizing these two identities we have
   $A^j _i A^i _j (v) =v $ which means $v$ is a eigenvector of the operator $A^j _i A^i _j $ with eigenvalue $1$. So  $v\in \bar{V}_{i,j} $  and the element $v+w=v-A^i _j (v)$ is in the space $\kappa < v-A^i _j (v)   >_{v\in \bar{V}_{i,j}} $. So  $( \bar{V}_i \oplus \bar{V}_j  ) \cap W \subset \kappa < v-A^i _j (v)   >_{v\in \bar{V}_{i,j}}$ and the proof is complete.\\

  \end{pf}
Now let $V= \bar{V}/W$, and denote the quotient map from $\bar{V}$ to $V$ as $p$. Denote $p( \bar{V}_i   )\subset V$ as $ V_i $, and
 $  p( \bar{V}^0 _i  ) $ as $V^0 _i$. Lemma 1.3 and Lemma 1.4 imply the following corollary.

 \begin{cor} Assuming notations as above. Then

\noindent(1) The restriction $p| _{\bar{V}_i } : \bar{V}_i \rightarrow V_i$ is an isomorphism.

\noindent(2) For any two different $i,j\in I$, the restriction $p|_{\bar{V}_{i,j}} : \bar{V}_{i,j}\rightarrow V  $ is an injection whose image is  the subspace $V_i \cap V_j \subset V$.
 \end{cor}

 \begin{pf} Clearly (1) is the consequence of Lemma 1.3.  By definition of $W$, there is a natural map
 $ \chi  :  ( \bar{V}_i \oplus \bar{V}_j  )/ \kappa < v-A^i _j (v)   >_{v\in \bar{V}_{i,j}} \rightarrow V $ induced by the inclusion

 $  \bar{V}_i \oplus \bar{V}_j  \rightarrow \bar{V}  $.  Lemma 2.4 shows $\chi$ is an injection, whose image is just the subspace
 $V_i + V_j$. It shows $\bar{V}_{i,j}$ is isomorphic to $V_i \cap V_j$ through $p$.

 \end{pf}

\begin{lem}
For any $H-$representation $\Psi$, there is
$\bar{V}\cong (\bar{V}_i \oplus \bar{V}_j ) + ( \bar{V}^0 _{i} \cap \bar{V}^0 _{j}   )$ and $    V\cong (V_i + V_j ) + ( V^0 _{i} \cap V^0 _{j} ) $  for any pair of different indices $i,j\in I$ .

\end{lem}

\begin{pf}
We only need to prove the first identity, since the second one can be induced from the first one by using Lemma 1.3.

\noindent$\bullet$ Since $\bar{V}\cong \bar{V}_i \oplus \bar{V}^0 _i $, the identity is true if and only if $\bar{V}^0 _i $ is in the left side space of this identity, which is equivalent to $\bar{V}^0 _i \cap [( \bar{V}_i + \bar{V}_j ) + \bar{V}^0 _i \cap \bar{V}^0 _j ]=\bar{V}^0 _i . $ 

\noindent$\bullet$ Recall $\bar{V}^0 _i \cong \oplus _{k\neq i} \bar{V}^0 _i [k]$. We need to show the space $\bar{V}^0 _i [k]$ is in the left side space for any $k\neq i$.  First for $k=j$,  $\bar{V}^0 _i [j] =\{  v-A^j _i (v) \}_{v\in \bar{V}_j } \subset    \bar{V}^0 _i \cap (\bar{V}_i + \bar{V}_j )$, so this case is done.

\noindent$\bullet$ Suppose $k\neq i,j$. Choose any $v- A^k _i (v) \in \bar{V}^0 _i [k]$ where $v\in \bar{V}_k $.  Now by the relation $(T^k _{i,j})$, there exist some vector $w\in \bar{V}_j $ such that $A^k _j (v) -A^i _j A^k _i (v) + A^i _j A^j _i (w) -w =0  $. Now we have
\[
v-A^k _i (v) = (w- A^j _i (w)) + ( v-A^k _i (v) -w + A^j _i (w)  ).
\]
The vector in the first bracket is evidently in $\bar{V}^0 _i \cap (\bar{V}_i + \bar{V}_j )$. The vector in the second bracket is in $\bar{V}^0 _i $ obviously. Any by above identity this vector equals

$(v-A^k _j (v)) + [ ( A^j _i (w) -A^k _i (v) ) - A^i _j ( A^j _i (w) - A^k _i (v)  )    ],$
which is in $\bar{V}^0 _j $. So it is in $\bar{V}^0 _i \cap \bar{V}^0 _j $.\\

\end{pf}

For any $1\leq i\leq n$, $V= V_i \oplus V^0 _i$.  Let  $t_i \in GL( V )$ be the automorphism that acting on $V_i$ as $q id_{V_i }$ and acting on $V^0 _i$ as $-id_{V^0 _i }$.  The following is the main theorem of this section.

\begin{thm} Suppose $\Psi$ is a $H-$representation of $\mathbb{Q}_I$ of type $(\Gamma ;q)$ , then the correspondence $\sigma _i \mapsto t_i ,i\in I$ extends to a representation $\rho_{\Psi} $ of the Hecke algebra $H_{q} (\Gamma) $ on $V$.

\end{thm}

\begin{pf} First there is $( \rho_{\Psi }(\sigma _i ) -q  ) ( \rho_{\Psi } (\sigma _i ) +1  )=(t_i -q)(t_i +1) =0   $ for any $i\in I$.  We need to prove for any pair $i\neq j $, the Artin relation $ [\sigma_i \sigma_j ...]_{m_{i,j}} = [\sigma_j \sigma_i ... ]_{m_{i,j}} $ holds on $V$.

  First suppose  $m_{i,j}$ is odd.  Recall the quotient map $p: \bar{V} \rightarrow V$. By lemma 1.1 we see $V_i \cong \bar{V}_i $. And by lemma 1.2 there is
 $V_i + V_j \cong \bar{V}_i \oplus \bar{V}_j / \kappa < v - \bar{A}^i _j (v)  >_{v\in \bar{V}_{i,j}}$.
Denote the subspaces $p( \bar{V}_{i/j}(a)   )$ and $p( \bar{V}_{j/i}(a) )$ as $V_{i/j}(a)$ and $V_{j/i}(a)$ respectively for any $1\leq a\leq [\frac{m_{i,j}-1}{2}]$. Since
$V_i = ( \oplus^{[\frac{m_{i,j}-1}{2}]}_{a=1} \bar{V}_{i/j}(a) )\oplus \bar{V}_{i,j}$, $V_j = ( \oplus^{[\frac{m_{i,j}-1}{2}]}_{a=1} \bar{V}_{j/i}(a) ) \oplus \bar{V}_{j,i}   $, and since both
$\bar{V}_{i,j}$ and $\bar{V}_{j,i}$ are identified with $V_i \cap V_j $ through $p$, there is
\[V_{i} + V_{j} =  (  \oplus^{[\frac{m_{i,j}-1}{2}]}_{a=1} V_{i/j}(a) ) \oplus (V_i \cap V_j ) \oplus (  \oplus^{[\frac{m_{i,j}-1}{2}]}_{a=1} V_{j/i}(a)  ) .\]
 For any $1 \leq a\leq [\frac{m_{i,j}-1}{2}] $,  let $\{  v^a _{s}   \}_{s=1,..., N_a }$ be a basis of $V_{i/j}(a)$. Denote $A^i _j (v^a _{s})$ as $v^{' a} _{s}$. So the set $\{ v^{' a} _{s} \}_{s=1,...,N_a } $ is a basis of $V_{j/i}(a)$. Notice $A^j _i (v^{' a} _{s}) = \lambda _{m_{i,j};v} (a) v^a _{s}$.

 Since $V_i $ is the characteristic space of the morphism $t_i $ in $V$ of eigenvalue $v^2 $, there is  $V_i + V_j = V_i \oplus (V_j \cap V^0 _i )$. So $t_i $ is stable on $V_i + V_j$. Similarly  $t_j$ is also stable on $V_i + V_j $.

 Now we show the Artin relations for $t_i , t_j$ holds on $V_i + V_j$.   First on $V_i \cap V_j$ , both $t_i$ and $t_j$ act as multiplication by $q$, so the Artin relation $ [t_i t_j ...]_{m_{i,j}} = [t_j t_i ... ]_{m_{i,j}} $ holds on $V_i \cap V_j$.

 Then for each $1\leq a\leq [\frac{m_{i,j}-1}{2}]$ and $1\leq s\leq N_a $, there is

\noindent $t_i ( v^a _{s}  ) =q v^a _{s}$,
 $ t_i ( v^{' a}_{s} - A^j _i (v^{' a} _{s})   ) = t_i (v^{' a}_{s} -  \lambda _{m_{i,j};v} (a) v^a _{s} ) = -( v^{' a}_{s} -  \lambda _{m_{i,j};v} (a) v^a _{s} )    $,

\noindent $ t_j ( v^{' a} _{s}  )= q v^{' a}_{s} $ and
 $  t_j ( v^a _{s} - v^{' a} _{s}   )=- (   v^a _{s} - v^{' a} _{s}   )   $.

 Where the second identity and the fourth identity are derived  from $ v^{'}_{a;s} -  \lambda _{m_{i,j};v} (a) v^a _{s} \in V^0 _{i} $ and $  v^a _{s} - v^{' a} _{s} \in V^0 _j  $ respectively.  These identities, on  one hand imply that $t_i$ and $t_j$ are stable on the two dimensional space $\kappa < v^a _{s} , v^{' a}_{s} >$, on the other hand, by writing  down the matrix form of $t_i $ and $ t_j$ according to this basis, imply $t_i $ and $t_j $ just combine to give a two dimensional representation of the Rank 2 Hecke algebra $H_{I_{m_{i,j}}}(v)$. So the Artin relation for $t_i ,t_j$ is satisfied on $\kappa < v^a _{s} , v^{' a}_{s} > $.
  Since $V_i +V_j =  \oplus^{[\frac{m_{i,j}-1}{2}  ] , N_a} _{a=1,s=1}  \kappa < v^a _s , v^{' a}_s   >  \oplus (V_i \cap V_j) $,  so the Artin relation for $t_i , t_j$ is satisfied on $ V_i +V_j $.

   Then on the subspace $ V^0 _i \cap V^0 _j $, both $t_i $ and $t_j$ acts as multiplication by constant $-1$, so the Artin relation for $t_i ,t_j $ is satisfied on $ V^0 _i \cap V^0 _j $.  As the last step,  by lemma 1.5  there is  $V= (V_i + V_j) + ( V^0 _{i} \cap V^{0} _{j}  )   $, so we have proved the Artin relation for $t_i ,t_j$ is satisfied on the whole space $ V $.

   The cases when $m_{i,j}$ is even can be proved similarly so we omit.

 \begin{rem}
Let $\bar{t}_i \in GL( V )$ be the automorphism that acting on $V_i$ as $- id_{V_i }$ and acting on $V^0 _i$ as $q id_{V^0 _i }$, then the correspondence $\sigma_i \mapsto \bar{t}_i $ also extends to a $H_{\Gamma}(q)-$representations  assuming conditions in the theorem.  The reason is as follows. Assuming conditions in the theorem, the $H_{\Gamma}(q)-$representation $\bar{\Psi}$ is already constructed. For $i\in I$, from
$ ( t_i -q   )(t_i +1)=0  $ , there is $( -q (t_i )^{-1} -q  )( -q (t_i )^{-1} +1     )=0$. So the correspondence $\sigma_i \mapsto -q (t_i)^{-1} $ for $i\in I$ also extends to a $H_{\Gamma}(q)-$representation. Now $-q (t_i)^{-1}= \bar{t}_i$. Denote this representation as
$\bar{\rho}_{\Psi}$.

 \end{rem}

\end{pf}

Inversely, there is  the following theorem.

\begin{thm} If $(V, \rho)$ is a representation of the Hecke algebra $H_{\Gamma}(q)$, then its Quiver data $\Psi _{\rho}$ is a
$H-$representation of type $(\Gamma ;q)$. Further more, $V = (  V_i + V_j ) \oplus V^0 _i \cap V^0 _j  $.

\end{thm}

\begin{pf} By Theorem 1.2, $\Psi _{\rho}$ satisfies the relations $(T^{(m_{i,j};q)}_{i,j})$ and the relations $(T^i_{j,k})$. We need to prove
it also satisfies the relations $(T^k _{i,j})$. As before for any $k\in I$ let $V_k , V^0 _{k}$ be the
$q-$eigenspace and $-1 -$eigenspace of $\rho(\sigma_k )$respectively. Let $i\neq j \in I$,  and denote the parabolic subalgebras of
$H_{\Gamma}(q)$ generated by $\sigma_i $ and $\sigma_j $ as $H_{i,j}$, which is isomorphic to the dihedral Hecke algebra
$H_{I_{m_{i,j}}}$. Let $\phi: H_{I_{m_{i,j}}}\rightarrow H_{i,j} $ be the isomorphism extending
$ \phi(\sigma_1)=\sigma_i , \phi(\sigma_2)=\sigma_j  $. So $\rho\circ\phi$ is a $H_{I_{m_{i,j}}}-$representation on $V$. We suppose
$m_{i,j}$ is even. The cases when $m_{i,j}$ is odd can be proved in the same way and be easier.

Since $q$ is generic the representation $\rho\circ\phi$ decompose into a direct sum of irreducible representations. So there is
 $V = (  V_i + V_j ) \oplus V^0 _i \cap V^0 _j $. Where $V_i +V_j $ is the sum of irreducible subrepresentations of type
     $ind, \epsilon_1 ,\epsilon_2 $ and $\mu_j$ for $1\leq j\leq \frac{m}{2}-1$ , and $V^0 _i \cap V^0 _j$ is the sum of irreducible
     subrepresentations of type $\epsilon$.  Denote the projection from $V$ to $V_i + V_j$ with respect to this decomposition as $p$.

      By Theorem 1.2, there are $V_i + V_j = V_{i/j} \oplus ( V_{i} \cap V_j ) \oplus V_{j/i}$, $V_i = V_{i/j} \oplus ( V_{i} \cap V_j ) $
      and $V_j = ( V_{i} \cap V_j ) \oplus V_{j/i}$.   Now let $v\in V_k$. Suppose $p(v) = v_1 + v_2 + v_3 \in V_i + V_j$, where
      $v_1 \in V_{i/j} , v_2 \in V_i \cap V_j $ and $v_3 \in V_{j/i}$. So $v= v_1 + A^k _i (v_3 )+ v_2 + v_3 - A^k _i (v_3)$. Since
      $  v_1 + A^k _i (v_3 )+ v_2 \in V_i  $ and $   v_3 - A^k _i (v_3) \in V^0 _i $ and $ v_1 + A^j _i (v_3) \in V_{i/j} $, so $A^k _i (v)= v_1 + A^j _i (v_3) + v_2$ and  $ p^i _{i,j} ( A^k _i (v)  )= v_2 $.  Similarly we can prove $ p^j _{j,i} ( A^k _j (v)  )= v_2  $,  thus $ p^i _{i,j} ( A^k _i (v)  )= p^j _{j,i} ( A^k _j (v)  )  $, which is equivalent to $(T^k _{i,j})$ according to Remark 1.2.

   \end{pf}

\begin{rem}
Let $ (V,\rho )$ be a $H_{\Gamma}(q)-$representation. For $i\in I$,let $V_i ,V^0 _i$ be the $q-$eigenspace and $-1-$eigenspace of $\rho(\sigma_i )$ respectively. For $i\neq j\in I$, An operator $B^i _j : V^0 _i \rightarrow V^0 _j$ similar with $A^i _j$ is defined
as : for $v\in V^0 _i$, $B^i _j (v)$ is defined as the $V^0 _j -$component of $v$ with respect to the direct sum decomposition
$V= V_j \oplus V^0 _j$. Then we have another representation $\Psi^0 _{\rho}$ of $\mathbb{Q}_I$ as: $\Psi^0 _{\rho}(o_i)=V^0 _i$,
$\Psi^0 _i (J^i _j)= B^i _j$. The following argument shows $\Psi^0 _{\rho}$ is also a $H-$representation of type $(\Gamma;q)$. Since
$ ( \rho(\sigma_i) -q  )( \rho(\sigma_i) +1 )=0$, then $ ( -q \rho(\sigma_i)^{-1} -q   )( -q \rho (\sigma_i )^{-1}+1   )=0 $, so the map
 $\bar{\rho}( \sigma_i )=-q \rho(\sigma_i)^{-1} $ extends to another $H_{\Gamma}(q)-$representation $\bar{\rho}$ on $V$. Now $V^0 _i$ become the $q-$eigenspace of $\bar{\rho}(\sigma_i)$ and $V_i$ become the $-1-$eigenspace of $\bar{\rho}(\sigma_i)$. So by Theorem 1.4,
 $\Psi^0 _{\rho}$ is a $H-$representation of type $(\Gamma ;q)$.
 \end{rem}

   Let $\Psi$ be a representation of $\mathbb{Q}_I $. Denote $\Psi (o_i )$ as $V_i $, and $\Psi ( J^i _j )$ as $A^i _j$.

   Then $\Psi$ has a natural dual representation $\Psi^{\#  }$ of $\mathbb{Q}_I$,
   which is defined as follows.  Let $\Psi^{\# } (o_i ) = V^{*} _i  $,which is denoted as $V^{\#}_i$,  and let $\Psi^{\#}(J^i _j ) = ( A^j _i  )^{*} $, which is denoted as $A^{\# i}_{j} $. Recall the canonical decomposition $V_i \cong V_{i,j}\oplus V_{i/j}$. By Lemma 1.2, the pair
   $( A^{\# i}_j , A^{\# j}_i  )$ is also a $( m_{i,j};v )-$couple between $V^{\#} _i$ and $ V^{\# }_j $, so there is a canonical decomposition
    $V^{\#}_i \cong V^{\#}_{i,j}\oplus V^{\#}_{i/j}  $. Recall the subspaces of $V^* _i$ ( or $V^{\#} _i $ ): $V^* _{i,j} =\{ f \in V^* _i | f(  V_{i/j} )=0 ) \} $ and
     $V^* _{i/j}= \{ f \in V^* _i | f(  V_{i,j} )=0 ) \}$ defined behind Lemma 1.1.  By Lemma 1.2, we have the following lemma.

      \begin{lem} Assuming above notations, then:  $V^{\#}_{i,j}= V^* _{i,j}$ and  $V^{\#}_{i/j} = V^*_{i/j} $.

  \end{lem}

   \begin{prop} Assume notations as above. Then

   (1) $\Psi$ satisfies the relations $(T^{(m_{i,j};q)}_{i,j})$ if and only if $\Psi^{\#}$ satisfy the relations $(T^{(m_{i,j};q)}_{i,j})$.

   (2) Suppose $\Psi$ satisfy the relations $(T^{(m_{i,j};q)}_{i,j})$. Then $\Psi$ satisfy the relations $(T^k _{i,j})$ if and only if
   $\Psi^{\#}$ satisfy the relations $(T^{i,j} _k) $.

   (3) If $\Psi$ is a $H-$representation of type $(\Gamma ; q)$, then $\Psi^{\#} $ is also a $H-$representation of type
   $(\Gamma ; q)$.

   \end{prop}

   \begin{pf}
  First, (1) follows from Lemma 1.2, and (3) follows from (1) and (2).  Then we prove (2). What we need to prove is that

   $(*)$  for any $f\in V^{\#}_{i,j} $, there is $ A^{\# i} _k (f) = A^{\# j} _k A^{\# i} _j (f)  $ for any $k$.

  Recall the canonical decomposition $V_i = V_{i,j} \oplus V_{i/j}  $.  Let $V^{*}_{i,j} $ and $ V^{*} _{i/j}$ be subspace of $V^{*}_i $  defined as above. Then  $V^{\#}_i = V^{*} _i = V^{*}_{i,j} \oplus V^{*} _{i/j}=  V^{\#}_{i,j} \oplus V^{\#} _{i/j} $,
   $V^{\#} _{i,j} =  V^{*}_{i,j} $ and $V^{\#} _{i/j} =  V^{*}_{i/j} $ by Lemma 1.2 and Lemma 1.6.

   The statement $(*)$ is equivalent to:

   for any $f\in V^{\#}_{i,j} $, and for any $v\in V_k$, $A^{\# i} _k (f) (v) = A^{\# j} _k A^{\# i} _j (f)  (v)  $, which is equivalent to

    for any $f\in V^{\#}_{i,j} $, and for any $v\in V_k$, $ f (  (  A^k _i - A^j _i A^k _j      )(v)     )=0    $, which is  equivalent to

    for any $v\in V_k$,  $ (  A^k _i - A^j _i A^k _j      )(v) \in V_{i/j} $. Since $V^{\#}_{i/j} = V^* _{i/j}$ by Lemma 1.6.  The last condition is just the condition $ (T^k _{i,j})$ by Remark 1.2.\\

    \end{pf}

    This gives another definition of $H-$representations as follows.

    \begin{defi}
    Let $\Gamma = ( m_{i,j} )_{i,j\in I}$ be a Coxeter matrix, $v\in \kappa$.  Let $\Psi$ be a representation of the quiver $\mathbb{Q}_{I}$. Then $\Psi$ is a $H-$representation of type $(\Gamma ; q )$ if and only if :

    (1) $(T^{(m_{i,j};q)}_{i,j})$ : For any $i\neq j\in I$, $\Psi$ satisfies the conditions $(T^{(m_{i,j};q)}_{i,j})$;

    (2) $(T^{i,j}_k)$  For any three different $i,j,k\in I$,  $\Psi$ satisfy the conditions $(  T^{i,j} _k )$.

    (3) $( T^{\# i,j} _k) $   For any three different $i,j,k\in I$,  $\Psi^{\#}$ satisfy the conditions $(  T^{i,j} _k )$.

    \end{defi}

    Suppose $(\bar{V},\rho)$ is a $H_{\Gamma}(q)$-representation. Denote the $q-$eigenspace and the $(-1)-$eigen
    space of $\rho(\sigma_i)$ as $\bar{V}_i$ and $\bar{V}^0 _i$ respectively for any $i\in I$. Denote the imbedding map from $\bar{V}_i$ to $\bar{V}$ as $ J_i$.
     For any $i\neq j\in I$, denote the jumping operator from $\bar{V}_i$ to
    $\bar{V}_j$ as $A^i _j$. Then $(\bar{V}_i , A^i _j )$ is a $H-$representation of type $(\Gamma ;v)$ by Theorem 1.4, which is denoted as $\Psi_{\rho}$.  Then recall the $H_{\Gamma}(q)-$representation  $(V,  \rho^{'} )$ derived from $\Psi_{\rho}$. Where
    $V= \oplus_{i\in I} \bar{V}_i /W $ and $W$ is the subspace spanned by vectors $\{ v-A^i _j (v)  \}_{i\neq j, v\in \bar{V}_{i,j}  }$.

    A natural map $f : V\rightarrow \bar{V}$ is defined as follows. First let $ \bar{f} : \oplus_{i\in I} \bar{V}_i \rightarrow \bar{V}$ be the map defined as $ \bar{f}( \sum_{i\in I} v_i  ) = \sum_{i\in I} J_i ( v_i )$ for any $\sum_{i\in I} v_i \in  \oplus_{i\in I} \bar{V}_i$.  Since $W\subset \ker(\bar{f})$, then a map $f: V\rightarrow \bar{V}$ is induced.

    \begin{lem} Assume notations as above, then
     (1) $f$ is a $H_{\Gamma}(q)-$morphism.

     \noindent(2) $\ker(f) \subset \cap_{i\in I} V^0 _i $ so $\ker(f)$ is a direct sum of $\epsilon$ (as in Theorem 1.1 ) as a

     $H_{q}(\Gamma)-$representation.

    \noindent (3) $coker(f)$ is a direct sum of $ind$ ( as in Theorem 1.1)as a $H_{q}(\Gamma)-$representation.

    \end{lem}

    \begin{pf}
    First, for $i\in I$ recall $V= V_i \oplus V^0 _i$. By definition of $V^0 _i$ and the jumping operator $A^i _j$, there hold
\[f( V_i ) = \bar{V}_i \quad \text{ and }\quad f( V^0 _i ) \subset \bar{V}^0 _i \quad  (*)\]
 So for any $i\in I$ and any $v=v_1 +v_2 \in V$ where $v_1 \in V_i $ and $v_2 \in V^0 _i$,
     $f(  \sigma_i (v) )= f( q v_i - v_2  )= q f(v_1) -f(v_2)= \sigma_i f(v) $. So $f$ is a $H_{\Gamma}(q)-$morphism.

     (2) is also an evident fact derived from $(*)$. For (3), since $\sum_i \bar{V}^0 _i \subset im (f)$, so $\sigma_i$ acts as
     $(-1)-$multiplication on $coker(f)$ for any $i$.

    \end{pf}

    \begin{rem} If $M,N$ are two $H_q (\Gamma)-$modules, the group $EXT(M,N)$ is easily computed without considering projective resolutions. So construction of $H-$representation  is equivalent to construction of $H_q (\Gamma)-$representation in some sense by above lemma.

    \end{rem}

\section{ Nonintersecting $H-$representations  }

  First application of  $H-$representation is to constructing representations of Hecke algebra, through constructing $H-$representations. To do that we need to construct linear representations of the quiver $\mathbb{Q}_{I}$ satisfying the eigenvalue relations $(T^{(m_{i,j};q)}_{i,j})$, the  triangular relations $(T^{i,j} _{k} )$ and $( T^k _{i,j}  )$. Usually it is still  hard to construct quiver representations satisfying all these three classes of relations simultaneously. But if $V_{i,j}=0$ for all $i\neq j\in I$ in a $H-$representation, then the triangular relations $(T^{i,j} _{k} )$ and $( T^k _{i,j}  )$ are satisfied automatically, so the operators $\Psi ( J^i _j  )$ only need to satisfy  relations $( T^{(m_{i,j};q)}_{i,j} )$.

 Let $\Gamma =(m_{i,j})_{i,j\in I}$ be any Coxeter matrix. Let $q\in \kappa$ be generic as in the beginning of Section 2.

 \begin{thm}
 Let $\Psi $ be a $H-$representation of type $(\Gamma;q )$. Denote $\Psi (i  ) $ as $V_i $, $ \Psi(J^i _j ) $ as $A^i _j$, and let
 $V_{i,j} , V_{i/j} $ be subspaces of $V_i $ as in comments following Definition 2.3. Then if $V_{i,j}=0$ for any $i\neq j\in I$,  $\Psi$
  is a $H-$representation of type $(\Gamma; q)$ if and only if  $( A^i _j , A^j _i  )$ is a specialized $(m_{i,j};q)$-couple between $V_i$ and $V_j$ for any $i\neq j\in I$.

    This kind of $H-$representations  $\Psi$ are called of \textbf{nonintersecting type.}

 \end{thm}

 By this theorem, to construct a nonintersecting type  $H-$representation is very easy.  We only need to take a set of $\kappa$-linear spaces $\{ V_i \} _{i\in I}$ such that $ \dim V_i =\dim V_j $ if $m_{i,j}$ is odd, then for any  binary subset $\{ i,j\} \in I$, take a specialized  $(m_{i,j};q)$-couple $(f,g)$ between $V_i $ and $V_j $ and let $A^i _j =f , A^j _i =g$. Proof of the following theorem is direct. (3) of this theorem shows nonintersecting $h-$representations are naturally deformable.

 \begin{thm} Let $\Psi$ be a nonintersecting $H-$representation of type $(\Gamma;q)$, and assume the attached notations as in Theorem 2.1.

\noindent (1) If $\Psi$ is of nonintersecting type then its dual $\Psi^{\#}$ is also.

\noindent (2) Let $\Psi ^{'} $ be the $\mathbb{Q}_{I}$-representation such that $ \Psi^{'} (o_i ) =\Psi (o_i) =V_i  $ for any $i\in I$, and
 $ \Psi^{'} (J^i _j) = - \Psi (  J^i _j )$ for any $i\neq j \in I$, then $\Psi^{'}$ is also a $H-$representation of type
 $(  \Gamma;q ) $.

\noindent (3) Let $0\neq q^{'} \in \kappa$,   let $\Psi^{"} $ be the $\mathbb{Q}_{I}$-representation such that $ \Psi^{"} (i ) =\Psi (i) =V_i  $ for any $i\in I$, and
 $ \Psi^{"} (J^i _j) = q^{'} \Psi (  J^i _j )= q^{'} A^i _j $ for any $i\neq j \in I$, then $\Psi^{"}$ is a $H-$representation of type
 $(  \Gamma;qq^{'} ) $.
  \end{thm}

The construction of nonintersecting type $H-$rerpesentations is relatively simple , but there are many not simple questions can be asked for them, especially when $q=1$, that is the $H-$representations corresponding to representation of Coxeter groups.

\noindent(1) To classify the irreducible nonintersecting  $H-$representations for every $(\Gamma;q)$.

\noindent(2) When $\Gamma$ is of affine type, to clarify which elements in the classical classification of irreducible representations are nonintersecting.

\noindent(3) When $q=1$, one can produce $H-$representations of intersecting type by taking components in $\Psi_1 \otimes ...\otimes \Psi_n  $, where $\Psi_i$ is nonintersecting for any $1\leq i\leq n$. If any irreducible $H-$representation can be realized as such a component, we could say "all irreducible $H-$representations of type $(\Gamma;1)$ are reachable". So the following question is natural,

Whether for all irreducible $H-$representation $\Psi$ of type $(\Gamma;1)$, there exist a finite sequence of $H-$representations of
type $(\Gamma;1)$ $\Psi_1 ,..., \Psi_n$, such that $\Psi$ is a component of $ \Psi_1 \otimes ...\otimes \Psi_n $?

\noindent(4) When $q=1$, the geometric representation ( Tits representation  ) seems in some sense the simplest member in the family of nonintersecting $H-$representations. It invoke the following questions.

Which nonintersecting $H-$representations of type $(\Gamma ;1)$ are faithful?
Which nonintersecting $H-$representations of type $(\Gamma ;1)$ are discrete?

 \section{Relation with the Kazhdan-Lusztig theory}
  In this section we combine $H-$representation with the famous  Kazhdan-Lusztig theory for Hecke algebra representations.

  First let's recall the Kazhdan-Lusztig theory. In section 3 and section 4 we need to take a square root of $q$, and view $H_q (\Gamma)$ as a module over certain Laurent polynomial ring. We set it as follows. $\Lambda= \kappa[ v^{\pm} ]$ be a laurent polynomial ring and let $q=v^2$. Let $H_q [\Gamma]$ be a $\Lambda-$algebra generated by the same set of generators and relations as in Definition 1.1. It is well known the Hecke algebra $H_{q}(\Gamma) $ is a free $\Lambda$-module with a natural basis $\{  T_{w} \}_{w\in W_{\Gamma}}$, in which $T_{s_i }$ can be identified with $ \sigma_i $ in the definition of $H_q (\Gamma)$ for any $i\in I$.

   Denote the Bruhat order in $W_{\Gamma}$ as $ "<" $. Denote length of $w$ as $l(w)$ for any $w\in W_{\Gamma}$, and denote $ v^{l(w)} $, $(-1)^{l(w)} $ as
   $v_w , \epsilon_w $ respectively. There is a canonical involution $\overline{( \cdot  ) }$ on $H_{q}(\Gamma)$, which is defined by $ \overline{f(v)  } = f ( v^{-1}  ) $ for any $f(v)\in \kappa[v^{\pm}]$ and $ \overline{  T_{w} } = ( T_{w^{-1}}   )^{-1}  $ for any $w\in W_{\Gamma}$.

   \begin{thm}(\cite{KL2})
   For any $w\in W_{\Gamma}$, there is a unique element $ C_w \in H_q (\Gamma) $ satisfying:

  \noindent (1) $\overline{C_w } = C_w  $,

  \noindent (2) $C_w = \sum_{y\leq w} \epsilon_y \epsilon_w v_w q^{-1}
   _y \overline{  P_{y,w} } T_y  $, where $P_{y,w}$ is a polynomial in $\kappa[q ]$ of degree $\leq \frac{1}{2}( l(w)-l(y)-1    )$ for $y<w$, and $P_{w,w}=1$.

   \end{thm}

   It is known the set $ \{ C_w \} _{w\in W_{\Gamma }} $ constitute a basis of $H_q (\Gamma)$ and is called the canonical basis.  The polynomials $P_{y,w}$ in the theorem are the celebrated Kazhdan-Lusztig polynomials. For $y,w\in W_{\Gamma}$, the notation $y \prec w $ means:
    $y<w$, $\epsilon_y = -\epsilon_w$, and the degree of $P_{y,w}$ is exactly $\frac{1}{2}( l(w)- l(y) -1  )$. When $ y\prec w $, a important integer $ \mu (y,w )$ is introduced as the coefficient of the highest power in $P_{y,w}$.
 The  following "Multiplication formulas" is important for our intent.

   \begin{prop}(\cite{KL2})
   (1) If $s_j w < w$, then $T_{s_j} C_w = -C_w $.

   \noindent(2) If $w< s_j w $,then $T_{s_j} C_w = v^2 C_w + v C_{s_j w} + v \Sigma _{ z\prec w; s_j z<z } \mu(z,w) C_z $.

   \noindent(3) $C_{s_j w} = C_{s_j} C_{w}- \Sigma _{s_j z <z ; z\prec w   } \mu(z,w) C_z $, if $s_j w>w$.
   \end{prop}

   Denote the characteristic space of $T_{s_i}$ of the regular representation with eigenvalue $-1$ , $q$ as $V^{-1}_i $, $V^{q} _i$ respectively. Set $\bar{T}_{s}= v^{-1} T_s $. Here is another automorphism $(\delta )$ of $H_{q}(\Gamma)$.

   \begin{lem}  The correspondence $ \bar{T}_{s_j }\mapsto \bar{T}_{s_j }  $ for any $j\in I$ and $v\mapsto -v^{-1}$  extends to a degree 2 automorphism of
   $H_{q}(\Gamma)$. Denote this automorphism as $\delta$.

   \end{lem}

   \begin{pf}
   It is easy to check the identity $(  \bar{T}_{s_j } -v   )(\bar{T}_{s_j} + v^{-1}  )=0  $, so the relations (2) (of Definition 1.1) are preserved. Evidently the relations (1) are also preserved. So we have a morphism $\delta$. Since $\delta ^2 = id _{H_{\Gamma}(q)}$, so $\delta$ is a degree 2 automorphism.  \\

   \end{pf}

    For convenience  denote $\delta ( C_w)  $ as $C^{'} _w $ for any $w\in W_{\Gamma}$. Now consider the regular representation of $H_q (\Gamma)$ and its associated $H-$representation. Denote the $q-$eigenspace and $(-1)-$eigenspace of $\sigma_j$ simply as $V_j ^q $ and $V_j ^{-1}$ respectively for $j\in I$.

    \begin{prop}
     $V_j ^{-1} = \Lambda < C_w   >_{s_j w<w  }  $, and $V_j ^{q} = \Lambda <   C^{'} _w    >_{s_j w <w }$.

       \end{prop}

    \begin{pf}
    First consider the bijection $W_{\Gamma} \rightarrow W_{\Gamma}$, $w\mapsto s_j w  $, as a permutation it is of degree 2, and has no fixed element. So  $| \{ w \} _{s_j w<w }  | =\frac{1}{2} | W_{\Gamma} | $. By the multiplication rule of Proposition 3.1, every nonzero element in
    $\Lambda < C_w >_{s_j w<w} $ is an eigenvector of $T_{s_j}$ with eigenvalue $-1$. Now for some $w$ satisfying $s_j w<w$,  apply the automorphism $\delta$ to the identity
     $ T_{s_j} C_w = -C_w $,  we have $ -q^{-1}T_{s_j } C^{'} _w  = - C^{'}_w  $, which is equivalent to
     $T_{s_j} C^{'} _w = q C^{'} _w  $.  So every element in $\Lambda <  C^{'} _w  >_{s_j w<w }$ is an $q-$eigenvector of
     $T_{s_j}$. Since these two subspaces contain eigenvectors of $T_{s_j}$ with different eigenvalues, the subspace spanned by them is a direct sum, $ \Lambda <  C_w > _{s_j w<w} \oplus \Lambda < C^{'}_w  >_{s_j w<w} $, whose dimension (rank as a free $\Lambda-$module ) is just
     $|W_{\Gamma}|$, so it must be the whole space $\Lambda W_{\Gamma}$ and the proposition is proved.

    \end{pf}

    Based on this proposition, we have the following description of the jumping operators. Suppose $i,j\in I$. By this proposition,
    $ V^{q} _i = \Lambda < C^{'}_w    >_{s_i w<w}$ and $V^{q} _j = \Lambda < C^{'}_w >_{s_j w<w}$.

    \begin{thm}
     \[
      A^i _j ( C^{'}_w   )=\begin{cases}
   C^{'}_w, &\text{ for $w: s_i w<w , s_j w<w $,    } \\
 -\frac{1}{v+v^{-1} } (C^{'} _{s_j w} - \Sigma_{z:s_j z<z; z\prec w} \mu(z,w) C^{'} _z ),
    &\text{for  $w: s_i w<w $ and $ s_j w > w$}.\\

\end{cases}
\]

    \end{thm}

   \begin{pf}
   For (1), since $ C^{'}_w \in V^{q} _i \cap V^{q} _j  $ when $s_i w<w , s_j w<w$, we have simply  $A^i _j (C^{'}_w )= C^{'}_w$. For (2), in cases $s_i w<w $ and $ s_j w >w $,  there is the following explicit decomposition of $C^{'}_w$.
   First there are $C_{s_j} = v^{-1} T_{s_j} -v  $ and $ C^{'} _{s_j } = v^{-1} T_{s_j} + v^{-1} $, which induce a "decomposition of 1 ":
   $1= \frac{1}{v+v^{-1}} ( C^{'} _{s_j} -C_{s_j}  ) $. This gives a decomposition

   $C^{'} _w = 1 \cdot C^{'} _w = \frac{1}{v+v^{-1}}C^{'} _{s_j} \cdot C^{'} _{w} -\frac{1}{v+v^{-1}} C_{s_j} \cdot C^{'} _w $.

   Now by the identity $( v^{-1} T_{s_j} -v  )( v^{-1} T_{s_j} +v^{-1}  )=0 $, there are $T_{s_j} C^{'} _{s_j} =q C^{'} _{s_j}  $
   and $T_{s_j } C_{s_j } = - C_{s_j } $.  So  $ C^{'}_{s_j} C^{'} _{w} \in V^{q} _{j} $ and
   $ C_{s_j} C^{'} _{w} \in V^{-1} _{j}  $, which imply $ A^i _j ( C^{'} _w  )=  \frac{1}{v+v^{-1}} C^{'}_{s_j} C^{'} _{w}  $. At last by the multiplication rule there is
    $ C^{'} _{s_j } C^{'} _{w} = C^{'} _{s_j w} + \Sigma_{s_j z<z; z\prec w   } \mu(z,w) C^{'} _z  $.\\
   \end{pf}

  Now recall the important conception $W-$graph and $W-$representations of $H_q (\Gamma)$, which are introduced by Kazhdan and Lusztig in \cite{KL2}. Here we use a slightly different version in \cite{GP}.   A $W$-graph $( X, I , X_i , \{ \mu^j _{y,x}  \}   )$ consists of a set of vertices  $X$ and a set of edges $Y$; for any $x\in X$, a subset  $I_x \subset I$ is associated; for any $x,y,i : \{y,x\}\in Y,   i\in I_y \setminus I_x$,  a number $0\neq \mu^{i} _{y,x}\in \kappa $ is associated. For these data to become a $W-$graph they need to satisfy further representational conditions ( defining a $H_q (\Gamma)-$ representation as follows ). But for our purpose we would like to introduce the term "\text{pre $W-$graph  }   "  just for these data without the representational conditions.

    With these data, one can define a set of linear operators as follows. Let $V$ be the
   free $\Lambda$-module with basis $X$. For any $i\in I$, let
\[
\tau_i (x)= \begin{cases}
-x, &\text{ when $i\in I_x,$   }\\
v^2x+v\sum_{y\in X, \{y,x\} \in Y, i\in I_y }\mu^{i} _{y,x}y, &\text{when $ i\notin I_x$}.\\

\end{cases}
\]

 Since by definition $(\tau_i -q)(\tau_i +1)=0$ for any $i\in I$, if the Artin  relations $[\tau_i \tau_j ...   ]_{m_{i,j}}= [\tau_j \tau_i ...]_{m_{i,j}}  $ hold  for any $i\neq j\in I$, then a $H_{q}(\Gamma)$-representation is defined on $V$. These representations are called $W$-graph representations, and the original set of data is called a $W-$graph. It is known that for a finite type $\Gamma$, any irreducible representation of $H_{\Gamma}(q)$ is a $W$-graph representation. But for certain $\Gamma$ of affine type, there exist irreducible $H_{q}(\Gamma)- $representations that can't be $W$-graph representations.  In the following we show very close relationship between $W$- graph representations and $H-$representations.

 Given a pre $W$-graph, we define a $\mathbb{Q}_{I}-$representation ( by imitating the identities in Theorem 3.2 ) as follows. Here we consider $H_q (\Gamma)$ as the $\kappa-$algebra in Definition 1.1, and $H_q (\Gamma)-$representations are $\kappa-$linear spaces.  For any $i\in I$ let $X_i = \{ x | i\in I_x \} $, and $\bar{V}_i = \kappa < x   >_{x\in X_i }$. Here $q\in \kappa$ is generic, and $v\in \kappa $ satisfying $q=v^2$. For $j\neq i\in I$, define linear morphism $A^j _{i;v}: \bar{V}_j \rightarrow \bar{V}_i$ as follows.

\[
A^j _{i;v} (x )=\begin{cases}
x,   &\text{when $  x\in X_j \cap X_i $},\\
 \frac{-1}{v+ v^{-1}} \sum_{ y\in X_i  } \mu^{i}_{y,x} y, &\text{when $x\in X_j \setminus X_i $}.    \\

 \end{cases}
 \]

The associated $\mathbb{Q}_{I}-$representation $(\bar{V}_i , A^i _{j;v})$ assigns $\bar{V}_i$ to $i$ and $ A^i _{j;v} $ to $J^i _j$.


 \begin{thm} With a pre $W$-graph as above, let $q$ be generic, then the $W$-graph representation is a $H_{q}(\Gamma )$-representation if and only if the following conditions hold.

 \noindent(1)The associated $\mathbb{Q}_I$-representation $(\bar{V}_i , A^i _{j;v} )$ is a $H-$representation of type $(\Gamma ; v)$.

  \noindent (2) For any $i\neq j \in I$, the $1$-eigenspace of $A^j _{i;v} A^i _{j;v }   $ is the subspace $\kappa < \{ x \}_{x\in X_i \cap X_j }   > $.
\end{thm}

\begin{pf} Denote the associated $\mathbb{Q}_{I}-$representation $(\bar{V}_i , A^i _{j;v} )$ as $\Phi$.  If the $W-$graph representation is a $H_{\Gamma}(q)-$representation, denote it as $\rho$. We show $\Phi$ is just the $\mathbb{Q}_{I}$ representation $\Psi^0 _{\rho}$ as in Remark 1.4.   By definition of $\tau_i$, the $(-1)-$eigenspace of $\tau_i$ is
$\bar{V}^0 _i =\kappa < x >_{x\in X_i }$ and the the $q-$eigenspace of $\tau_i$ is
$\bar{V} _i = \kappa < \{ x+ \frac{1}{v+v^{-1}} \sum_{ \{ y,x\} \in Y , y\in X_i }  \mu^i _{y,x} y  \}_{x\in X\setminus X_i} >$.

 Now for
$x\in X_j \setminus X_i $,
$x= ( x+ \frac{1}{v+v^{-1}} \sum_{ \{ y,x\} \in Y , y\in X_i }  \mu^i _{y,x} y ) +   ( -\frac{1}{v+v^{-1}} \sum_{ \{ y,x\} \in Y , y\in X_i }  \mu^i _{y,x} y )  $ where the first term is in $\bar{V}_i $ and the second term is in $\bar{V}^0 _i$, so $ B^j _i (x) =-\frac{1}{v+v^{-1}} \sum_{ \{ y,x\} \in Y , y\in X_i }  \mu^i _{y,x} y =A^j _{i;v} (x) $. So by Remark 1.3,  $(\bar{V}_i , A^i _{j;v})$ is a $H-$representation of type $(\Gamma ,v)$. Since in
$V$, for any $i\neq j\in I $, there is $\bar{V}_i \cap \bar{V}_j =\kappa < x >_{x\in X_i \cap X_j }  $, we see the $1$-eigenspace of $A^j _{i;v} A^i _{j;v}$ is $\kappa < x >_{x\in X_i \cap X_j }$.

On the other if the conditions on $\Phi$ in the theorem are satisfied, recall the procession deriving a $H_{\Gamma}(q)-$representation   $\bar{\rho}_{\Phi}$ from $(\bar{V}_i , A^i _{j;v})$ as in Remark 1.3.  For $x\in X_i $, denote the corresponding element in $\bar{V}_i$ as $[x]_i$, so $\bar{V}_i =\kappa < \{   [x]_i  \}_{x\in X_i}   >$. Since the $1$-eigenspace of $A^j _{i;v} A^i _{j;v }   $ is the subspace $\kappa < \{ [x]_i \}_{x\in X_i \cap X_j }   > $, the representation space of $\rho$ is
$V=  \oplus_{i\in I} \bar{V}_i  / W $, where $W $ is spanned by vectors $\{  [x]_i - A^i _{j;v} (  [x]_i  )    \}_{ i\neq j, x\in X_i \cap X_j}$. Let $\phi: \oplus_{i\in I} \bar{V}_i \rightarrow V  $ be the morphism sending $ [x]_i $ to $x$ for any $i\in I$ and any $x\in X_i$.
Since $ [x]_i - A^i _{j;v} (  [x]_i  )=[x]_i - [x]_j  $,so $\phi$ induces an isomorphism $\bar{\phi}$ from $V$ to $\kappa <  X >$. It is readily to check the operator $ \bar{\rho}_{\Phi}(\sigma_i )  $ coincides with $\tau_i$ for any $i\in I$ through the isomorphism $\bar{\phi}$.
\end{pf}

\begin{defi} Let $\Gamma$ be a Coxeter matrix and $q=v^2 \in \kappa $ be generic for $\Gamma$.
A pre $W$-graph $( X, I , X_i , \{ \mu^j _{y,x}  \}   )$ is called as $(\Gamma ;v)$-admissible, if one of the following equivalent conditions holds.

\noindent (1) The associated operators $\{ \tau_i  \} $ satisfy the relations of the Hecke algebra $H_q (\Gamma)$.

\noindent (2) The associated quiver representation $(\bar{V}_i ,A^i _{j;v} )$ is a $H-$representation of type $(\Gamma;q)$.

\end{defi}

 We introduce \textbf{flat $H-$representations} as follows, inspired by $W$-representation and Theorem 3.3.
Let $V$ be a linear space over some field $\kappa$, $\mathbb{A}=\{ V_i \}_{i\in I}$ be a set of subspaces. $\mathbb{A}$ is called as \textbf{an arrangement of subspaces}.
For any $J\subset I$, denote $V_{J} =\cap_{j\in J} V_j$. Let $W_J = V_J / \Sigma_{K: J\subsetneq K } V_K$ for $K\subsetneq I$, and
 let $W_I =V_I$.

\begin{defi}

An arrangement $\mathbb{A}$ is called as flat, if for any $J\subset I$, $\dim V_J = \sum _{K:K\subset J} \dim W_K $.

\end{defi}

The following lemma explains that for a flat arrangement $\{ V_i \} _{i\in I}$, the operation $\cup$ and $\cap$ of these subspaces $V_i$ can be reduced to $\cup$ and $\cap$ on the level of sets by certain basis.

\begin{lem}
 Suppose $ \mathbb{A}=\{  V_i   \}_{i\in I}  $ is an arrangement in $V$.

 \noindent (1) $\mathbb{A}$ is flat if and only if there is a basis $B= \{ v_{\alpha } \}_{\alpha \in S } $ of $V$, such that for any $J\in I$, $B \cap V_J $ is a basis of $V_J $.

\noindent (2) $\mathbb{A}$  is flat if and only if there is a basis $B= \{ v_{\alpha } \}_{\alpha \in S } $ of $V$, such that for any $i\in I$, $B \cap V_i$ is a basis of $V_i$.

 \noindent (3) The arrangement $ \mathbb{A}=\{  V_i   \}_{i\in I}  $ is always flat if  $|I|=2$.
 \end{lem}

\begin{pf}
First suppose there exist such a basis $B= \{ v_{\alpha } \}_{\alpha \in S } $. For $i\in I$, let $S_i \subset S $ be determined by $B\cap V_i = \{ v_{\alpha } \}_{\alpha \in S_i }$. For any $J\subset I$, let $S_J = \cap _{j \in J} S_j $. Then
$\{ v_{\alpha}  \}_{\alpha \in S_J  }$ is a basis of $V_J$, and $\{ v_{\alpha } \}_{ \alpha \in  \cup_{K: J\subsetneq K} S_K }  $ is a basis of $\sum_{ K: J\subsetneq K } V_K $. So $\{ v_{\alpha} \}_{ \alpha \in S_J \setminus  \cup_{K: J\subsetneq K} S_K }$ projects to a
basis of $W_J$. Denote $S^{'} _{J} = S_{J} \setminus \cup_{K: J\subsetneq K} S_K $.  So $S = \sqcup _{J\subset I} S^{'} _{J}$ and
$\dim V = \sum _{J\subset I} \dim W_J $.

 Then suppose an arrangement $\{ V_i \}_{i\in I}$ is flat. For any $J
 \subset I$, we choose a subset $\{  v_{\alpha } \}_{\alpha \in S^{'} _J }  \subset V_J  $ which project to a basis of the quotient space $W_J$.
 Since $| S^{'} _J  | = \dim W_J $, and the union $B= \cup _{J\subset I } \{ v_{\alpha }  \}_{\alpha \in S^{'} _J }   $ spans $V$, by the
 identity $\dim V = \sum _{J\subset I} \dim W_J$, this union is a basis of $V$. For any $J\subset I$, denote
 $S_J = \cup_{K\subset J } S^{'} _K $, and $S= \cup _{J\subset I} S_J = \cup _{J\subset I} S^{'} _J $. It isn't hard to see
 $B= \{  v_{\alpha} \}_{\alpha \in S } $ is a requested basis, so (1) is proved.

 Now let $B=\{ v_{\alpha} \}_{\alpha \in S}$ be a basis of $V$, such that for any $i\in I$, $B\cap V_i$ is a basis of $V_i$.  Then for any
 $J\subset I$, and for any $v\in V_J$, represent $v= \sum_{\alpha} k_{\alpha} v_{\alpha} $ and  let $Supp(v)=\{ \alpha | k_{\alpha}\neq 0 \}$. Then for any $i\in J$, $Supp(v)\subset S_i $, so $B\cap V_J$ is a basis of $V_J$. This fact along with above arguments conclude (2). (3) is evident so we omit the proof.

\end{pf}

\begin{rem}
we call a basis of $V$ in above theorem as a atomic basis with respect to the arrangement $\mathbb{A}$.
If $\mathbb{A}=\{ V_i \}_{i\in I}$ is an arrangement in $V$,  then for any $J\subset I$, $\{ V_J \cap V_i \}_{i\notin J }  $ is an arrangement in $V_J $, which is denoted  as $\mathbb{A}_J $.
\end{rem}

\begin{cor} Assuming notations in Remark 3.1, If $\mathbb{A}$ is a flat arrangement, then $\mathbb{A}_J $ is also flat.
\end{cor}
\begin{pf}
By Lemma 3.2, suppose $B= \{ v_{\alpha } \}_{\alpha \in S }$ is a atomic basis with respect to $\mathbb{A}$. Then $B\cap V_J $ is a basis of $V_J$, and for any $K: J\subset K$, $(B\cap V_J ) \cap V_K = B\cap V_K $ is a basis of $V_K$. So by using Lemma 3.2 again, $\mathbb{A}_J$ is flat, with $B\cap V_J $ as an atomic basis.
\end{pf}

\begin{defi}
A $H_q (\Gamma )-$representation $(V, \rho )$ is called as flat if $\{  V_i  \}_{i\in I }$ is a flat arrangement in $V$.
A $H-$representation $(  \bar{V}_i  , A^i _j  )$ is called as flat if for any $i\in I$, $\{  V_{i,j }  \}_{j\in I , j\neq i}$
 is a flat arrangement in $V_i$.
\end{defi}

The next theorem says two kinds of flatness in this definition are closely connected. Here is a preparatory lemma. If $(\bar{V}_i , A^i _j )$ is a $H-$epresentation, recall $\bar{V}_{i,j}\subset \bar{V}_i$ is defined as the 1-eigenspace of $A^j _i A^i _j$. For convenience set $\bar{V}_{i,i}=\bar{V}_i$.
 For $J\subset I$ and $i\in J$, let $\bar{V}_{i|J}=\cap _{j\in J} \bar{V}_{i,j}$.

\begin{lem}
Let $J\subset I$ and $i,j \in J$ such that $i\neq j$. Then the operator $A^i _j $ restricts to a isomorphism  $A^i _{j;J}$  from
$\bar{V}_{i|J }  $ to $\bar{V}_{j| J }$,  whose inverse  is  $A^j _{i;J}$.

Let $i,j,k \in J\subset I$ be different to each other.  Then
$ A^i _{k;  J } = A^j _{k; J } \circ   A^i _{j;  J } $.

\end{lem}

\begin{pf}
Recall the $H_q (\Gamma)-$representation $\rho_{\Psi}$ associated with $\Psi$ constructed in Theorem 1.3, whose representation space is $V$. As is corollary 1.1 there is isomorphism $p|_{\bar{V}_i } : \bar{V}_i \rightarrow V_i \subset V $ for any $i\in I$, and
$p|_{\bar{V}_i } (  \bar{V}_{i,j} )= V_i \cap V_j$ for any $i\neq j \in I$.  For $J\subset I$, let $V_J = \cap_{j\in J} V_j $. Then
 $ p|_{\bar{V}_i } ( \bar{V}_{i |J}  )=
 p|_{\bar{V}_i } (\cap _{j\in J , j\neq i} \bar{V}_{i,j})= \cap _{j\in J , j\neq i} p|_{\bar{V}_i } (\bar{V}_{i,j}) =
 \cap _{j\in J , j\neq i} ( V_i \cap V_j ) = V_J $. Denote this isomorphism from $\bar{V}_{i|J} $ to $V_J$ as $\phi_{i|J}$.

Let $J\subset I$ and $i,j \in J$ such that $i\neq j$. Let $v\in \bar{V}_{i|J}$. As in the construction of the space $V$,  both
$v$ and $A^i _j (v)$ are identified with the vector $ p|_{\bar{V}_i } (v) =p|_{\bar{V}_j }(  A^i _j (v) ) \in V $, so
$ A^i _{j;J} (v) = ( \phi_{j|J} )^{-1} \phi_{i|J} (v)  $. All statements of this lemma follow from this identity.

\end{pf}

\begin{thm}
If a $H_q (\Gamma )-$representation $(V,\rho )$ is flat,  then its associated quiver representation $(V_i , A^i _j   )$ is  flat.
Inversely if a $H-$representation $( V_i , A^i _j   )$ of type $(\Gamma;q)$ is flat, then the associated $H_q (\Gamma)-$representation  is flat also.
\end{thm}

\begin{pf}
First suppose $(V,\rho )$ is a flat $H_q (\Gamma )-$representation, then $\mathbb{A}=\{ V_i \}_{i\in I}$ is a flat arrangement in $V$, where $V_i $ is the $q-$eigenspace of $\sigma_i$. By definition of the jumping operators $A^i _j$, the 1-eigenspace  $V_{i,j}$ of $A^j _i A^i _j$ can be identified with $V_i \cap V_j$. By using Corollary 3.1, $\{  V_i \cap V_j \} _{j\neq i} $ is a flat arrangement in $V_i$, so the associated quiver data $(V_i , A^i _j   )$ is a flat.

Then suppose $(\bar{V}_i , A^i _j   )$ is a flat $H-$representation of type $(\Gamma ;q)$, suppose $(V, \rho)$ is the associated $H_q (\Gamma)-$representation constructed in Theorem 1.3. Recall in that construction, the quotient map $p : \bar{V}=\oplus_{i\in I} \bar{V}_i \rightarrow V$ maps $\bar{V}_i $ isomorphically to a subspace $V_i$ and $\bar{V}_{i,j}$ isomorphically to $V_i \cap V_j$. For any $J\subset I$, let $V_J = \cap_{i\in J} V_i$.
Choose any total order $\prec $ in $2^I$(the power set of $I$), such that if $ J_1 \subset J_2$ then $J_1 \prec J_2$. By using this total order we define an index set $C_J $ for $J\subset I$ inductively as follows.

First for $J=I$, let $ \{ v_{\alpha} \}_{\alpha \in C_{I}} $ be a basis of $V_I$. It means $C_I $ is the index set for this basis. Assuming notations as in Lemma 3.3. For any $i\in I$ and $\alpha\in C_I$, denote $( \phi_{i|I}  )^{-1} (v_{\alpha})  $ as $ v^i _{\alpha} $, so
$ \{ v^i _{\alpha} \}_{\alpha \in C_{I}}   $ is a basis of $\bar{V}_{i|I}$ by Lemma 3.3.

Now Let $J\in 2^I$, and suppose for any $K: J\prec K$, $C_{K}$ and a linear independent subset $\{  v_{\alpha} \}_{\alpha \in C_k} \subset V$ is already constructed. Choose any $i\in J$, choose any linear independent subset $\{ v^i _{\alpha} \}_{ \alpha \in C_J  }\subset  \bar{V}_{i|J}$ such that
$ \bar{V}_{i|J} = \kappa < \{ v^i _{\alpha} \}_{ \alpha \in C_J  }  > \oplus ( \sum_{K: J\precnsim K }   V_{i|K} ) $. For  $\alpha \in C_J$ and $j\in J$ let $v^j _{\alpha}= A^i _j ( v^i _{\alpha} )$ and $v_{\alpha}= \phi_{i|J}( v^i _{\alpha} )$,  so by Lemma 3.3, there are
\[
\bar{V}_{j|J}= \kappa <  \{ v^j _{\alpha} \}_{ \alpha \in C_J  }  >  \oplus ( \sum_{K: J\precnsim K }  ) \bar{V}_{j|K} \quad \text{and} \quad V_{J} = \kappa <  \{ v_{\alpha} \}_{ \alpha \in C_J  }  >  \oplus ( \sum_{K: J\precnsim K }  ) V_{K}   . \]

Since  $(\bar{V}_i , A^i _j   )$ is flat then for any $i\in I$, the set
$\{  v^i _{\alpha}  \}_{\alpha \in \sqcup_{J: i\in J}  C_J }$ is a basis of $ \bar{V}_i $. Recall the construction
$ V= \bar{V} / W $, the process of taking a quotient of $\bar{V}=\oplus_{i\in I} \bar{V}_i $ is nothing but identifying all
$\{ v^i _{\alpha } \}_{ i\in J  }$  with one vector $v_{\alpha} \in V$ for $\alpha \in C_J$. So $\{ v _{\alpha}  \}_{\alpha \in \sqcup_{J\subset I} } C_J  $ is a basis of $V$ and $\{ v _{\alpha}  \}_{\alpha \in \sqcup_{J: i\in J } } C_J  $ is a basis of $V_i$ for any $i\in I$. So the representation $\rho$ is flat.\\

\end{pf}

Suppose $\Psi=( \bar{V}_i , A^i _j   )$ is a flat $H-$representation. Denote the associated $H_q (\Gamma)-$representa
tion as $(\rho_{\Psi}, V)$ and assume notations in that proof, for example the indices sets $C_J$ for any $J\subset I$, and vectors $v_{\alpha}, v^i _{\alpha} $.
For any $i\in I$, let $X_i = \bigcup_{J:i\in J} C_J$, then $\bar{V}_i$ has a basis $\{ v^i _{\alpha} \}_{\alpha \in X_i  }$. Using these
bases the jumping operator $A^i _j$ is presented as follows.

\[
A^i _{j} (v^i _{\alpha } )=\begin{cases}
 v^j _{\alpha },   &\text{when $  \alpha\in X_i \cap X_j $},\\
 \frac{-1}{v+ v^{-1}} \sum_{ \beta \in X_j  } \mu^{j}_{i: \beta ,\alpha } y, &\text{when $\alpha \in X_i \setminus X_j $}.\\

 \end{cases}
 \]

Since $\Psi$ satisfies the triangular relation $( T^{i,j} _k )$, for different $i,k\in I$ there must be
$  \mu^{j}_{i; \beta ,\alpha }= \mu^{j}_{k; \beta ,\alpha}   $, so we can conveniently rename all $\mu^j _{i; \beta ,\alpha}$'s as
$\mu^j _{\beta , \alpha}$. A comparison of this identity with the identity about $A^j _{i;v}$ before Theorem 3.3 indicates the close
relationship between flat $H-$representations and pre $W-$graphs, that is, every flat $H-$representation can be presented by a pre $W-$graph through the identity above Theorem 3.3.

\begin{prop} If $\gamma$ is a rank 3 Coxeter system, then any type $(\Gamma,v)$ $H-$representation is flat.

\end{prop}

\begin{pf} Suppose the index set $I=\{ 1,2,3 \}$ for $\Gamma$. Choose any $i\in I$, denote the left two elements in $I$ as $j,k$.
Then the arrangement $\{ V_{i,j} ,V_{i,k} \} $ consists of two subspace only. This arrangement must be flat by Lemma 3.2,  so the
proposition follows by Theorem 3.4.
\end{pf}

\section{Equations for $W-$graphs }

Last section connects pre $W$-graphs with $H-$representations. Combining with conditions in definition of  $H-$representation, we derive equations ( about constants $\{ \mu^j _{y,x} \}$ ) that making a pre $W-$graph to be a $W-$graph.
 Assume  $\mathcal{G}=(X, I, X_i , \mu^j _{\beta ,\alpha} ) $ is a pre $W$-graph, $0\neq v\in \kappa$. Suppose $(\bar{V}_i , A^i _{j;v})$ is the associated quiver representation as in discussions above Theorem 3.3.

\begin{thm} Let $\bar{V}^0 _{i/j} =\kappa < v^i _{\alpha} >_{\alpha \in X_i \setminus X_j } $ for any $i\neq j\in I$.  Let $P^i _j : \bar{V}^0 _{i/j} \rightarrow \bar{V}^0 _{j/i}$ be the linear map
 $P^i _j ( v^i _{\alpha})=\sum_{ \beta \in X_j \setminus X_i } \mu^j _{\beta,\alpha} v^j _{\beta}  $ for $\alpha \in C_i \setminus C_j $ for any $i\neq j \in I$. Then

\noindent(1) $( \bar{V}_i , A^i _{j;v} )$ satisfies the condition $(T^{(m_{i,j};q)}_{i,j})$ and the space $\bar{V}_{i,j} = \kappa < v^i _{\alpha} >_{\alpha \in X_i \cap X_j }$

if and only if $P^i _j  P^j _i $ is diagonalizable  with eigenvalues in $\{ 4\cos ( \frac{a\pi}{m_{i,j} } ) \} _{a=1,2,...,[\frac{m_{i,j}-1}{2}]   } $.

\noindent(2) If the conditions in (1) are satisfied, then $(\bar{V}_i , A^i _j )$ satisfies the conditions $( T^{i,j} _k  ) $ automatically.

\end{thm}

\begin{pf} First if $( V_i , A^i _{j;v} )$ satisfies the eigenvalue condition $(T_{i,j})$ and the space $\bar{V}_{i,j} = \kappa < x >_{x\in C_i \cap C_j }$, since $A^i _{j;v} ( \bar{V}_{i,j} ) \subset \bar{V}_{j,i}$ and $A^j _{i;v} ( \bar{V}_{j,i} ) \subset \bar{V}_{i,j}$, so the couple $( A^i _{j;v} , A^j _{i;v}  )$ induces
 a couple $ ( \bar{A}^i _j , \bar{A}^j _i   ) $ between $\bar{V}_i / \bar{V}_{i,j}   $ and $ \bar{V}_j / \bar{V}_{j,i}  $.  Clearly this couple is equivalent to the couple $( \omega P^i _j , \omega P^j _i  ) $ between $\bar{V}^0 _{i/j} $ and $\bar{V}^0 _{i/j} $ .  Since $A^j _{i;v} A^i _{j;v}$ is semisimple and $\bar{V}_{i,j}$ is its $1-$eigenspace,
 so $\bar{A}^j _i \bar{A}^i _j $ is semisimple with eigenvalues in $\{ 4\omega^2 \cos ( \frac{a\pi}{m_{i,j} } ) \} _{a=1,2,...,[\frac{m_{i,j}-1}{2}]   } $, which implies (1).

Inversely if the conditions in (1) for  $( P^i _j , P^j _i )$ are true,  so $A^j _{i;v} A^i _{j;v}   $ is semisimple with eigenvalues in $\{ 4\omega^2 \cos ( \frac{a\pi}{m_{i,j} } ) \} _{a=1,2,...,[\frac{m_{i,j}-1}{2}]   } \cup \{ 1\} $. Since $\bar{V}_{i,j}$ is in the $1-$eigenspace of
  $A^j _{i;v} A^i _{j;v}$ and $1\neq 4\omega^2 \cos ( \frac{a\pi}{m_{i,j} } )$   for any $1\leq a\leq [\frac{m_{i,j}-1}{2}]$, so $ A^j _{i;v} A^i _{j;v}$ is semisimple and its $1-$eigenspace is  $\bar{V}_{i,j}$.

For $(2)$, by above argument if the conditions in (2) hold then $\bar{V}_{i,j}=\kappa <x>_{x\in C_i \cap C_j }$. Now let $k\in I\setminus \{i,j \} $, and let $x\in C_i \cap C_j $, we see both $A^i _{k;v} (x)$ and $A^j _{k;v} (x)$ are $ \sum_{y\in C_k} \mu^k _{y,x} y $, which implies (2).\\

\end{pf}

Suppose  $\Psi=( \bar{V}_i , A^i _{j;v}  )$ satisfy the conditions (1) in Theorem 4.1.  Now we derive the sufficient and necessary conditions for $\Psi $ to satisfying the triangular conditions $(T^i _{j,k})$ in addition. Let $i,j,k\in I$ be different with each other. By definition of $A^j _{k;v} $, the $1-$eigenspace of $A^k _{j;v} A^j _{k;v}$, that is, the space $\bar{V}_{j,k}$ is simply
$\kappa <  \{ v^j _{\alpha} \}_{\alpha\in X_j \cap X_k }   >  $.  Let $\bar{V}^0 _{j/k}=\kappa < \{ v^{j} _{\alpha}  \}_{\alpha \in X_j \setminus X_k }   >$.

Unfortunately the space $\bar{V}_{j/k}$ is not $ \bar{V}^0 _{j/k}$. Instead, there is a morphism
$\phi_1 : \bar{V}^0 _{j/k} \rightarrow  \bar{V}_{j,k} $ such that $\bar{V}_{j/k}= \{  v+\phi_1 (v)       \}_{v \in \bar{V}^0 _{j/k}}  . $ Similarly there is a morphism
$\phi_2 : \bar{V}^0 _{k/j} \rightarrow \bar{V}_{k,j} $ such that $\bar{V}_{k/j}= \{  v+\phi_2 (v)       \}_{v \in \bar{V}^0 _{k/j}}  . $ First we determine $\phi_1$ and $\phi_2$.

Take a total order $\prec_{1}$ in $X_j$ and a total order $\prec_2$ in $X_k$ such that,
(1)for any $\alpha \in X_j \cap X_k $ and any
$\beta \in X_j \setminus X_k$, there is $\beta \prec_1 \alpha$;
(2)for any $\alpha \in X_k \cap X_j $ and any
$\beta \in X_k \setminus X_j$, there is $\beta \prec_2 \alpha$;
 (3)restriction of $\prec_1$ on $X_j \cap X_k \subset X_j$ is the same total order as the restriction of $\prec_2$ on $X_k \cap X_j \subset X_k$.

 By using $\prec_1$, $\{ v^j _{\alpha} \}_{\alpha \in X_j }$ become a ordered basis of $\bar{V}_j$. Similarly by using  $\prec_2$, $\{ v^j _{\alpha} \}_{\alpha \in X_k }$ become a ordered basis of $\bar{V}_k$. With respect to these bases,
the operators $A^j _{k;v}$ and $A^k _{j;v}$ are represented as the following block matrices.
\[
\begin{pmatrix}
 M (f_{11}) &  M(f_{12})\\
0 &  I
\end{pmatrix}     ,      \begin{pmatrix}
M( g_{11}) &  M(g_{12})\\
0 &  I
\end{pmatrix}
\]

where $f_{11}$ ( $f_{12}$ ) is certain morphism from $\bar{V}^0 _{j/k} $ to
$\bar{V}^0 _{k/j} $ ( $ \bar{V}_{k,j} $     ). Similarly $g_{11}$ ( $g_{12}$ ) is certain morphism from $\bar{V}^0 _{k/j} $ to
$\bar{V}^0 _{j/k}$ ( $\bar{V}_{j,k}$     ).
 $"I"$ means an identity matrix.    Then we take a new ordered basis of $\bar{V}_i$ as follows. In the original one, for
 $\alpha \in X_j \setminus X_k$, replace $v^j _{\alpha} $ with the vector $v^j _{\alpha} + \phi_1 ( v^j _{\alpha} ) $, and if
 $\alpha \in X_j \cap X_k$, still take the same vector $v^j _{\alpha}$. The transition matrix from the old basis to the new basis
 is represented by the block matrix
 $\begin{pmatrix}
 I &  M( \phi_1 )\\
0 &  I
\end{pmatrix} $ . Similarly take a new basis of $\bar{V}_k$, the transition matrix from the old basis to the new basis is
$\begin{pmatrix}
 I &  M( \phi_2 )\\
0 &  I
\end{pmatrix} $.  Now write down the matrices of $A^j _{k;v}$ and $A^k _{j;v}$ with respect to these new bases. Since $\{ v^j _{\alpha} + \phi_1 ( v^j _{\alpha} ) \}_{\alpha \in X_j \setminus X_k } $ is a basis of $\bar{V}_{j/k}$ and  $\{ v^k _{\alpha} + \phi_2 ( v^k _{\alpha} ) \}_{\alpha \in X_k \setminus X_j } $ is a basis of $\bar{V}_{k/j}$, so these matrices should be diagonal block matrices. It gives the following identities.

(1) $ M( f_{12}) + M(\phi_1) - M(f_{11})M( \phi_2)  =0  $;  (2)  $ M( g_{12}) + M(\phi_2) - M(g_{11})M( \phi_1)  =0  $.

Now $M( \phi_1 )$ and $ M(\phi_2 )  $ can  be solved from these identities as

$ M( \phi_1 ) = -( I - M(f_{11}) M( g_{11} )   )^{-1} ( M(  f_{11} )M( g_{12} ) +M( f_{12})  )   $ and

$M( \phi_2 ) = -( I - M(g_{11}) M( f_{11} )   )^{-1} ( M(  g_{11} )M( f_{12} ) +M( g_{12})  )   $.

With respect to the decomposition $\bar{V}_j =  \bar{V}^0 _{j/k} \oplus \bar{V}_{j,k}   $, denote the projection from
 $\bar{V}_j $ to the first direct sum component and the second component as $p_{j,1}  $ and $ p_{j,2}  $ respectively.
  Similarly  define two projection $p_{k,1}$ and $p_{k,2} $ from $\bar{V}_{k}$ to $\bar{V}^0 _{k/j}$ and $\bar{V}_{k,j}$ respectively.
  Denote the morphism $p_{j,1} \circ A^i _{j;v} $ ($p_{j,2} \circ A^i _{k;v} $    ) from $\bar{V}_i$ to $ \bar{V}^0 _{j/k}  $  ($ \bar{V}^0 _{k/j}  $)  as $B^i _j $ ($B^i _k $).

\begin{thm} Let  $\mathcal{G}=(X, I, X_i , \mu^j _{\beta ,\alpha} ) $ be a $W$-graph, $v\in \kappa$ be generic, and  $\Psi=(\bar{V}_i , A^i _{j;v})$ be the associated quiver representation. Assume notations as above. Then $\Psi$ is a $H-$representation if and only if the following conditions are satisfied.

\noindent (1)  For any $i\neq j \in I$, $P^i _j  P^j _i $ is diagonalizable  with eigenvalues in $\{ 4\cos ( \frac{a\pi}{m_{i,j} } ) \} _{a=1,2,...,[\frac{m_{i,j}-1}{2}]   } $;

\noindent (2) For any three different $i,j,k\in I$,
$ M( B^i _j )( I - M(f_{11}) M( g_{11} )   )^{-1} ( M(  f_{11} )M( g_{12} ) +M( f_{12})  )$

$ = M(B^i _k ) ( I - M(g_{11}) M( f_{11} )   )^{-1} ( M(  g_{11} )M( f_{12} ) +M( g_{12})  ) $.

\end{thm}

\begin{pf}
First, as in (1) of Theorem 4.1,  $\Psi$ satisfies the condition $(T_{i,j})$ if and only if condition (1) hold. $\Psi$ satisfies
 the triangular conditions $(T^{i,j}_k )$ automatically.  It is enough to show when condition (1) hold, $\Psi$ satisfies the
 conditions $(T^{i}_{jk})$ if and only if the condition $(2)$ hold. Let $i,j,k\in I$ be different from each other. The condition $(T^i _{j,k})$ is equivalent to the following statement by Remark 1.2.

 \noindent $\bullet$ For any $v\in \bar{V}_i$, the component of $A^i _{j;v} (v)$ in $\bar{V}_{j,k}$ (with respect to the canonical decomposition
 $\bar{V}_j = \bar{V}_{j/k} \oplus \bar{V}_{j,k}$) equals the component of $A^i _{k;v} (v)$ in $\bar{V}_{k,j}$ (with respect to the canonical decomposition $\bar{V}_j = \bar{V}_{j/k} \oplus \bar{V}_{j,k}$).

 By above construction, the component of $A^i _{j;v} (v)$ in $\bar{V}_{j,k}$ is
 $- \phi_1 \circ p_{j,1} \circ A^i _{j;v} (v) + p_{j,2}\circ A^i _{j;v}(v) $. The the component of $A^i _{k;v} (v)$ in $\bar{V}_{k,j}$ is
 $- \phi_2 \circ p_{k,1} \circ A^i _{k;v} (v) + p_{k,2}\circ A^i _{k;v} (v) $. Since $p_{j,2}\circ A^i _{j;v} (v) =  p_{k,2}\circ A^i _{k;v} (v)$, this identity is equivalent to  $- \phi_1 \circ p_{j,1} \circ A^i _{j;v} (v)= - \phi_2 \circ p_{k,1} \circ A^i _{k;v} (v)$, whose matrix form is just the condition (2).\\

\end{pf}

\section{The Basic Data of $H-$Representation }

For every $H-$representation $\Psi=( \bar{V}_i , A^i _j  )$, there is a quiver representation $\Theta_{\Psi}$ called \textbf{the basic data of $\Psi$}  "hidden" in $\Psi$. The relevant quiver $\mathbb{Q}^2 _I$ is defined in Definition 5.1. $\Theta_{\Psi}$ contains partial information of the operators $A^i _j$, and it determine the subspace arrangement $\cup_{j\neq i}\{ \bar{V}_{i,j}, \bar{V}_{i/j} \}$ in $\bar{V}_i$ for any $i\in I$.

 \begin{defi}
 For a set $I$, the quiver $\mathbb{Q}^2 _I $ is defined as follows. Let $I^2 = \{ \{ i,j\} \subset I  |  i\neq j \} $ be the set of two-element sets of $I$. Notice $\{ i,j \} =\{ j,i\} $. The set of vertices of $\mathbb{Q}^2 _I $ is
 $ I^2  $. For any two different
 $\{ i,j  \} ,  \{ k,l\} \in I^2 $, there is an edge $J^{i,j}_{k,l}$ from ${i,j}$ to ${k,l}$, and they constitute all edges of $\mathbb{Q}^2 _I $.  Notice in above notations  $ J^{i,j} _{k,l} = J^{j,i} _{k,l} = J^{i,j}_{l,k}  $.

 \end{defi}

 For convenience of discussions, we choose any total order $ \prec $ on $I$. In the following whenever we talk about an element
 $\{ i,j\}  $ in $I^2$, it implies $i\prec j$.

 Suppose $\Psi=(\bar{V}_i , A^i _j )$ is a $H-$representation of type $(\Gamma ;q)$. Recall $\bar{V}_{i,j}$ is the $1-$eigenspace of
 $A^j _i A^i _j$. Recall the $H_q (\Gamma)-$representation $(  V,  \rho_{\Psi})$ associated with $\Psi$ as in Theorem 1.3. In the proof of that theorem there is a quotient map $p: \bar{V}= \oplus_{i\in I} \bar{V}_i \rightarrow V$. The restriction of $p$ on $\bar{V}_i$ is an isomorphism to certain subspace $V_i \subset V$ by Lemma 1.3, which will be denoted as $p_i$. By lemma 1.4, there are isomorphisms
 $p_i |_{ \bar{V}_{i,j}   } : \bar{V}_{i,j} \rightarrow V_i \cap V_j$ and $p_j |_{ \bar{V}_{j,i}   } : \bar{V}_{j,i} \rightarrow V_i \cap V_j$. There is also an canonical isomorphism $A^i _j |_{ \bar{V}_{i,j}  } : \bar{V}_{i,j} \rightarrow \bar{V}_{j,i} $.  These three isomorphisms satisfy $p_j |_{ \bar{V}_{j,i}  } \circ A^i _j |_{V_{i,j}}= p_i |_{ V_{i,j} }   $. It is convenient to identify these
 spaces $ \bar{V}_{i,j} ,  \bar{V}_{j,i} $ and $V_i \cap V_j$ through these isomorphisms in this way without making any confusion.

\noindent \textbf{A natural morphism from $\bar{V}_{i,j}$ to $\bar{V}_{k,l}$ .}
 Let $\{ i,j\} ,\{ k,l\} \in I^2$ be two different element,  the following argument shows several morphisms from  $V_{i,j}$ to $V_{k,l}$ defined in different ways coincide, thus there exists a single natural morphism from  $V_{i,j}$ to $V_{k,l}$.

\noindent Case (1).  $| \{ i,j \} \cap \{ k,l \} |=1$.
These cases are further divided into four subcases $i=k$, $i=l$, $j=k $, $j=l$. It is enough to deal with the cases $i=k$ since other subcases can be dealt with little modifications of the following argument. In this subcase let the two elements of $I^2$  be $\{ i,j\} $ and $\{ i,k\} $ where $i,j,k\in I$ are different to each other.

 There are three natural ways to chose a morphism  from $V_i \cap V_j$ to $V_i \cap V_k$ as follows.

 Recall the canonical decomposition $\bar{V}_i \cong \bar{V}_{i,j} \oplus \bar{V}_{i/j}$. Denote the natural projector from $\bar{V}_i$ to $\bar{V}_{i,j}$ with respect to this decomposition as $p^i _{i,j}$.

 $A^{i,j} _{i,k} [1] :  \bar{V}_{i,j} \stackrel{ p^i _{i,k} }{\longrightarrow } \bar{V}_{i,k} $.
 $A^{i,j} _{i,k}[2] :  \bar{V}_{i,j} \stackrel{A^i _k }{\longrightarrow  } \bar{V}_k
  \stackrel{p^k _{k,i}}{\longrightarrow} \bar{V}_{k,i}\stackrel{A^k _i }{\longrightarrow} \bar{V}_{i,k} $.

  $A^{i,j} _{i,k} [3]: \bar{V}_{i,j}\stackrel{A^i _j }{\longrightarrow } \bar{V}_{j,i} \stackrel{ A^j _k }{\longrightarrow} \bar{V}_k
  \stackrel{ p^k _{k,i}}{\longrightarrow} \bar{V}_{k,i} \stackrel{A^k _i }{\longrightarrow } \bar{V}_{i,k} $.

\begin{lem}
Above three morphisms are the same. We denote this morphism as $A^{i,j}_{i,k} $.

\end{lem}

\begin{pf} Let $v\in \bar{V}_{i,j}$.
 Proof of $A^{i,j} _{i,k}[1](v)= A^{i,j} _{i,k}[2](v) $ is as follows.  Suppose $v=v_1 +v_2$, where $v_1 \in \bar{V}_{i,k}, v_2 \in \bar{V}_{i/k}$. Then
$A^{i,j} _{i,k}[1](v)= p^i _{i,k}(v)= v_1 $. The  mapping process from $v$ to $A^{i,j} _{i,k}[2](v)$  is as follows:
$ v= v_1 +v_2 \stackrel{ A^i _k }{\longrightarrow} A^i _k (v_1 ) + A^i _k (v_2 )\stackrel{p^k _{k,i}}{\longrightarrow}
  A^i _k (v_1 ) \stackrel{A^k _i }{\longrightarrow} A^i _k A^i _k (v_1 )=v_1  .$

  The second arrow is because $A^i _k (v_1 )\in \bar{V}_{k,i}  $ and $  A^i _k (v_2 )\in \bar{V}_{k/i} $. The third arrow is because
  $v_i \in \bar{V}_{i,k} $ so there is $ A^i _k A^i _k (v_1 )=v_1    $.

  Proof of $A^{i,j} _{i,k}[2](v)= A^{i,j} _{i,k}[3](v) $ is as follows. Since $v\in \bar{V}_{i,j}$,  the triangular condition $(T^{i,j}_k )$ shows $ A^j _k A^i _j (v) = A^i _k (v) $, and their subsequent  mapping process are the same.\\

\end{pf}

\noindent Case(2). $\{ i,j\} \cap \{ k,l\} =\emptyset$.
In these cases there are four natural ways to define a morphism from $\bar{V}_{i,j}$ to $\bar{V}_{k,l}$ as follows.

$A^{i,j}_{k.j}[1] : \bar{V}_{i,j} \stackrel{A^i _k }{\longrightarrow} \bar{V}_k  \stackrel{p^k _{k,l}}{\longrightarrow} \bar{V}_{k,l} $.
$A^{i,j}_{k.j}[2]: \bar{V}_{i,j} \stackrel{A^i _j }{\longrightarrow} \bar{V}_{j,i} \stackrel{A^j _k }{\longrightarrow} \bar{V}_k
 \stackrel{p^k _{k,l}}{\longrightarrow} \bar{V}_{k,l} $.

$A^{i,j}_{k.j}[3]: \bar{V}_{i,j} \stackrel{A^i _l }{\longrightarrow} \bar{V}_{l} \stackrel{p^l _{l,k}}{\longrightarrow} \bar{V}_{l,k}
 \stackrel{A^l _k }{\longrightarrow} \bar{V}_{k,l} $.
$A^{i,j}_{k.j}[4]: \bar{V}_{i,j} \stackrel{A^i _j }{\longrightarrow} \bar{V}_{j,i} \stackrel{A^j _l }{\longrightarrow} \bar{V}_l \stackrel{p^l _{l,k}}{\longrightarrow} \bar{V}_{l,k} \stackrel{A^l _k }{\longrightarrow} \bar{V}_{k,l} $.

\begin{lem}
Above four morphisms are the same. We denote this morphism as $A^{i,j}_{k,l}$.
\end{lem}

\begin{pf} Let $v\in \bar{V}_{i,j}$.

\noindent  $\bullet$ The proof of $ A^{i,j}_{k.j}[1](v) =A^{i,j}_{k.j}[2](v)$. Because of the  condition $(T^{i,j}_k )$ and $v\in \bar{V}_{i,j}$, there is
 $ A^i _k (v) = A^j _k A^i _j (v) $ and their subsequent mapping processes are the same.

\noindent$\bullet$ The proof of $ A^{i,j}_{k.j}[1](v) =A^{i,j}_{k.j}[3](v)$. By using the condition $( T^i _{k,l})$, $p^k _{k,l} (  A^i _k (v) )\in \bar{V}_{k,l} $
should be  the element identified with $ p^l _{l,k} ( A^i _l (v) ) \in \bar{V}_{l,k} $. So
 $A^l _k ( p^l _{l,k} ( A^i _l (v) ) )= p^k _{k,l} (  A^i _k (v) )\in \bar{V}_{k,l} $.

\noindent $\bullet$ The proof of $ A^{i,j}_{k.j}[3](v) =A^{i,j}_{k.j}[4](v)$. By using the condition $(T_{i,j})$ and because $v\in \bar{V}_{i,j}$, so
  $A^j _l A^i _j (v) = A^i _l (v) $ and their subsequent mapping processes are the same.\\

\end{pf}

Now we can define the claimed quiver representation.

\begin{defi}
 Let $\Psi = (\bar{V}_i , A^i _j) $ be a $H-$representation of type $\Gamma$. Assuming notations as above. Define a representation $\Xi_{\Psi} $ of the quiver
 $\mathbb{Q}^2 _{I}$ as follows.

\noindent $\bullet$ For a vertex ${i,j}$ of $\mathbb{Q}^2 _{I}$ where $\{ i,j\} \in I^2$, let $\Xi_{\Psi} ( {i,j} ) = \bar{V}_{i,j}$.

 \noindent $\bullet$ For an edge  $J^{i,j} _{ k,l }$ of $\mathbb{Q}^2 _{I}$ where $\{ i,j\}\neq \{ k,l\}  \in I^2$, let $\Xi_{\Psi} ( J^{i,j}_{k,l} ) = A^{i,j} _{k,l}$  defined in Lemma 5.1 and Lemma 5.2.

\end{defi}

If $\Xi $ is a linear representation of the quiver $\mathbb{Q}^2 _{I}$, and for each $i\in I$, a pair of subspaces
$U_i ,W_i \subset \oplus_{j:j\neq i} \Xi({i,j}) $ are chosen. Then the combination $[ \Xi , \cup_{i\in I} \{ U_i , W_i \}   ]$ will be called as a \textbf{decorated $\mathbb{Q}^2 _{I}-$representation}.

 Assume above notations. For $i\in I$, for any $j\in I$ different with $i$, denote the natural injection from $\bar{V}_{i,j}$ to $\bar{V}_i $ as $J^{i,j} _i $. Suming up these injections gives a mophism
 $J_i  : \oplus_{j\neq i} \bar{V}_{i,j} \longrightarrow \bar{V}_i $ as :
 for any $\sum_{j:j\neq i} v_j \in  \oplus_{j\neq i} \bar{V}_{i,j} $,
 $J_i ( \sum_{j:j\neq i} v_j   ) = \sum_{J:j\neq i} J^{i,j} _i (v_j ) \in \bar{V}_i $.

 $\ker (J_i   )  $ is an important space in this construction, it will be denoted as $U_i$. Furthermore the subspace
 $Im ( J_i )$ is denoted as $ \bar{V}^0 _i $. On the other hand recall the projector $p^i _{i,j} : \bar{V}_i \longrightarrow \bar{V}_{i,j}$ with respect to the decomposition $\bar{V}_i = \bar{V}_{i,j} \oplus \bar{V}_{i/j}$.  Summing up them gives a map $P_i : \bar{V}_i \longrightarrow \oplus_{j\neq i} \bar{V}_{i,j} $ as:
for any $ v\in \bar{V}_i $, $P_i (v) = \sum _{j\neq i} p^i _{i,j}(v) \in \oplus_{j\neq i} \bar{V}_{i,j} $.
 $Im (  P_i )$ will be denoted as $W_i$.

\begin{defi}
 Let $\Psi=( \bar{V}_i , A^i _j  )$ be a $H-$representation.  Assuming above notations,  we call the decorated $\mathbb{Q}^2 _{I}-$representation $[ \Xi_{\Psi} , \cup_{i\in I} \{ U_i ,W_i \} ] $ as the basic data of  $\Psi$,  and denote it as $\Theta_{\Psi}$.  On the other hand if a decorated $\mathbb{Q}^2 _{I}-$representation $\Theta$  can become the basic data of some $H-$representation of type $(\Gamma;q)$, then we call $\Theta$ as $(\Gamma ;q)-$realizable.
\end{defi}

The basic data $\Theta_{\Psi}$ contains partial message of the $H-$representation $\Psi$ as follows.

\noindent$\bullet $ $\Theta_{\Psi} (\{ i,j\} )= \bar{V}_{i,j}$;

\noindent$\bullet $ The subspace $\bar{V}^0 _i = \sum_{j:j\neq i} \bar{V}_{i,j}\subset \bar{V}_i $ is determined as the quotient $\oplus_{j:j\neq i} \Theta_{\Psi}  (\{i,j\}) / U_i $. Also from definition of
$A^{i,j} _{k,l} $ we see $\Theta_{\Psi}$ contains partial information of $ A^i _j | _{V^0 _i} : V^0 _i   \rightarrow V_j  $.

\noindent$\bullet $ For each $i$ and $j: j\neq i$ , there is a canonical decomposition $\bar{V}_i \cong \bar{V}_{i,j} \oplus \bar{V}_{i/j}$. Information about these decompositions is contained in the operators $\{ A^{i,j} _{i,k} \}$, seen from the definition of these operators.

  \noindent \textbf{Dual of decorated $\mathbb{Q}^2 _{I}-$representations. }

  Let $\Psi = ( V_i , A^i _j  )$ be a $H-$representation. Recall its dual
 $H-$representation $\Psi^{\#}= ( V^{\#}_i , A^{\# i}_j   )$ ( Lemma 1.6 ) .  Now if $\Theta=
 ( (V_{i,j} , A^{i,k}_{j,l}), \cup_{i\in I} \{ U_i ,W_i \}  )  $ is a decorated $\mathbb{Q}^2 _{I}-$represent
 -ation, we define its dual
 $\Theta^{\#}=( (V^{\#}_{i,j} , A^{\# i,k}_{j,l}), \cup_{i\in I} \{ U^{\#} _i ,W^{\#} _i \}  ) $ as follows.

 \noindent $\bullet$ $V^{\#}_{i,j} = ( V_{i,j} )^{*}$ and $ A^{\# i,k} _{j,l} = ( A^{j,l}_{i,k} )^{*}   $.

 \noindent $\bullet$ $ U^{\#}_i = ( W_i )^{\perp}  $ and $W^{\#}_i = (U_i )^{\perp}  $.  Meaning of the notation " $ (.)^{\perp}  $  " is as follows. Denote the natural pairing $ V^{\#} _{i,j} \times V_{i,j}\rightarrow \kappa  $ as $( -,- )_{i,j}$. Summing up these pairings gives a perfect pairing: $(-,-): \oplus_{j:j\neq i } V^{\#}_{i,j} \times \oplus_{j: j\neq i } V_{i,j} \rightarrow \kappa$. $(X)^{\perp}$ means the perpendicular subspace of $X$ with respect to $(-,-)$.

  \begin{thm} Let $\Psi = (V_i , A^i _j )$ be a $H-$representaiton, then $\Theta^{\#} _{\Psi} = \Theta_{ \Psi^{\#} } .   $

  \end{thm}

  \begin{pf}
  \noindent \textbf{Proof of $ A^{\# i,j}_{k,l}= ( A^{k,l}_{i,j}  )^{*}  $.} in cases $i\neq k$, by definition the morphism $A^{\# i,j}_{k,l}$ is a composition of morphisms $V^{\#}_{i,j}  \stackrel{ J^{\# i,j} _i }{\longrightarrow} V^{\#}_i   \stackrel{ A^{\# i}_k  }{\longrightarrow}  V^{\#}_k
   \stackrel{ P^{\# k}_{k,l} }{\longrightarrow} V^{\# }_{k,l}  $.

   Since $ J^{\# i,j} _i =( P^i _{i,j} )^* $, $ P^{\# k}_{k,l} = ( J^{k,l} _k  )^* $ and $ A^{\# i}_k = ( A^k _i )^* $, above morphism is the dual of the morphism  $ V_{k,l}  \stackrel{J^{k,l}_k  }{\longrightarrow}   V_k  \stackrel{ A^k _i  }{\longrightarrow}    V_i
    \stackrel{ P^i _{i,j}  }{\longrightarrow}   V_{i,j}  $,  which is just $ A^{k,l} _{i,j} $.

    In cases $i=k$, by definition $ A^{\# i,j } _{i,l} $ is the composition
    $ V^{\#} _{i,j}  \stackrel{ J^{\# i,j}_i  }{\longrightarrow}   V^{\#}_i   \stackrel{P^{\# i}_{j,k}  }{\longrightarrow}  V^{\#}_{i,l}  $, so it is the dual of the morphism
    $ V_{i,k}  \stackrel{ J^{i,k}_i }{\longrightarrow}   V_i  \stackrel{ P^i _{i,j} }{\longrightarrow}   V_{i,j} $, which is just
    $A^{i,k}_{i,j} $.

    \noindent \textbf{ Proof of $ U^{\#}_i = ( W_i )^{\perp}  $ and $W^{\#}_i = (U_i )^{\perp}  $.} Recall the definitions of $U_i$ and $W_i$ below Definition 5.2. For the dual $H-$representation $\Psi^{\#}$ there are similar
    morphisms $J^{\#} _i  : \oplus_{j\neq i} V^{\#} _{i,j} \longrightarrow V^{\#}_i $ and
     $P^{\#}_i : V^{\#}_i \longrightarrow \oplus_{j\neq i} V^{\#}_{i,j} $. It is easy to see $ J^{\#}_i = ( P_i )^* $ and
     $ P^{\#}_i = ( J_i )^* $. So $U^{\#}_i = \ker (  J^{\#} _i    ) = ( Im (  P_i ) )^{\perp} = ( W_i )^{\perp} $ and
     $  W^{\# }_i = Im ( P^{\#} _i )= ( \ker ( J_i ) )^{\perp}  = (U_i )^{\perp}. $\\

  \end{pf}

\noindent \textbf{Necessary conditions for $(\Gamma;q)-$realizability of $\Theta$ .}
 Let $\Theta= [ ( \bar{V}_{i,j} , \bar{A}^{i,j}_{k,l} ) ,  \cup_{i\in I} \{ \bar{U}_i ,\bar{W}_i \}    ] $ be a decorated $\mathbb{Q}^2 _I -$representation. The following objects related to $\Theta$ are introduced for further discussion.

 \noindent $\bullet$ For any $i\in I$,  $\bar{V}^0 _i $ is defined as the quotient space $ \oplus_{j: j\neq i} \bar{V}_{i,j} /  \bar{U}_i  $.

\noindent $\bullet$ $\bar{V}^0 $ is defined as the quotient space $\oplus _{\{ i,j\} \in I^2 } \bar{V}_{i,j} / \sum_{i\in I } \bar{U}_i   $.  The natural imbedding from  $ \oplus_{j: j\neq i} \bar{V}_{i,j} $ to $\oplus _{\{ i,j\} \in I^2 } \bar{V}_{i,j}$ induces a morphism
$ \Upsilon_i : \bar{V}^0 _i \rightarrow \bar{V}^0$, since $ \bar{U}_i  \subset \sum_{j\in I } \bar{U}_j  $.

\noindent $\bullet$ Let $i,j\in I$. Then a morphism  $ \Phi^i _j : \oplus_{j: k\neq i} \bar{V}_{i,k} \rightarrow  \oplus_{l: l\neq j} \bar{V}_{j,l} $ is defined by

$\Phi^i _j (  \sum_{k:k\neq i} v_{i,k}  ) =   \sum_{ k: k\neq i  } \sum _{l: l\neq j } \bar{A}^{i,k} _{j,l} (v_{i,k})   $.

  \begin{thm}
  Let $\Theta = [ ( \bar{V}_{i,j} , \bar{A}^{i,j}_{k,l} ) ,  \cup_{i\in I} \{ \bar{U}_i ,\bar{W}_i \}    ] $ be a decorated $\mathbb{Q}^2 _I -$representation. Let $\Theta^{\#}$ be the dual of $\Theta$.  Assume above notations.  If $\Theta$ is $( \Gamma;q) -$realizable, then the following conditions hold.

 \noindent (1)  For any $i\in I$, $\Upsilon_i $ is injective.

 \noindent (2)  For any $i\in I$, and for any $j\in I$,  $\bar{U}_i  \subset \ker( \Phi^i _j  ) $.

 \noindent (3)  For any $i\in I$, and for any $j\in I$, $Im ( \Phi^j _i   )\subset \bar{W}_i   $.

  \noindent (4) For any $i\in I$, in $ \oplus_{j: j\neq i} \bar{V}_{i,j} $,   $  \bar{U}_i \cap  \bar{V}_{i,j} = \{ 0\} $, for any $j\neq i$.

 \noindent (5)  (1)$\sim$(4) are satisfied for the dual $\Theta^{\#}$.

  \end{thm}

  \begin{pf} Suppose  $\Psi =(\bar{V}_i , A^i _j )$ is any $H-$representation of type $(\Gamma;q)$, suppose
  $\Theta_{\Psi} = [ ( \bar{V}_{i,j} , A^{i,j}_{k,l} ) ,  \cup_{i\in I} \{ \bar{U}_i ,\bar{W}_i \}    ]$ is the basic data of $\Psi$  as in Definition 5.3.
  It is enough to prove $\Theta_{\Psi}$ satisfies the conditions in this theorem.

   Recall the $H_{q}(\Gamma)-$representation $(V, \rho_{\Psi})$ associated to $\Psi$ constructed in Theorem 1.3.
    For (1), consider the following morphisms. Recall the quotient map $p: \bar{V}=\oplus_{i\in I} \bar{V}_i \rightarrow V$, the subspace
    $\bar{V}_{i,j}\subset \bar{V}_i $. By (2) of Corollary 1.1, restriction of $p$ gives a isomorphism $\bar{f}_{i,j}: \bar{V}_{i,j} \rightarrow V_i \cap V_j  $. Now consider the following morphisms.

     $ (\oplus_{j:j\neq i} \bar{V}_{i,j})/\bar{U}_i  \stackrel{\Upsilon_i }{\longrightarrow}   (\oplus_{k\neq j} \bar{V}_{k,j}) /\sum_{k} \bar{U}_k \stackrel{f }{\longrightarrow}  \sum_{k\neq j} (V_{k}\cap V_{j} ) $.

     The morphism $f$ is defined as follows. Summing up all $\bar{f}_{i,j}$'s gives a morphism

    \noindent $\bar{f}: \oplus_{k\neq j} \bar{V}_{k,j}\rightarrow  \sum_{k\neq j} (V_{k}\cap V_{j} ) $. For any $i\in I$, restriction of $\bar{f}$ on $\oplus_{j:j \neq i} \bar{V}_{i,j}$ have $ \sum_{j: j\neq i} (V_{i}\cap V_{j} )  $ as the image and $\bar{U}_i$ as the kernel. So $\bar{U}_i \subset \ker (\bar{f})$ and $ \sum_{i} \bar{U}_i  \subset \ker(\bar{f}) $, and $f$ is the morphism  induced by  $\bar{f}$.  Now by definition of $\bar{U}_i$, $f\circ \Upsilon_i $ gives an isomorphism from $(\oplus_{j:j\neq i} \bar{V}_{i,j})/\bar{U}_i $ to $\oplus_{j:j\neq i} (V_i \cap V_j)  $. So $\Upsilon_i $ must be injective.

     Now consider the morphisms:
$ \oplus_{k:k\neq i} \bar{V}_{i,k}  \stackrel{J_i }{\longrightarrow}   V_i \stackrel{A^i _j }{\longrightarrow}  V_j
     \stackrel{ P_j  }{\longrightarrow}   \oplus_{l:l\neq j} \bar{V}_{j,l} ,  $
     where morphisms $J_i $ and $P_j$ are  introduced behind Definition 5.2 . By definition of
     $ A^{i,j} _{k,l} $, we have $ \Phi^i _j =  P_j \circ A^i _j \circ J_i   $. So $U_i = \ker( J_i  ) \subset \ker(  \Phi^i _j  )   $
      and $Im( \Phi^i _j   )\subset Im (  P_j )= W_j   $, so (2) and (3) follow.

      Since $\bar{V}_{i,k} $ imbeds in $ \oplus_{j: j\neq i} \bar{V}_{i,j} / U_i $, then $(4)$ must holds.

      If $\Theta$ can be realized as the basic data of some H-representation  $\Psi$ of type $(\Gamma;v)$, then $\Theta^{\#}$ can be realized as the basic data of the dual H-representation $\Psi^{\#}$ by Theorem 5.1, which is also of type $(\Gamma;q)$. So
      (1)$\sim$(3) should be satisfied for $\Theta^{\#}$.

  \end{pf}

\begin{defi} Let  $\Theta = [ ( \bar{V}_{i,j} , \bar{A}^{i,j}_{k,l} ) ,  \cup_{i\in I} \{ \bar{U}_i ,\bar{W}_i \}    ] $ be a decorated $\mathbb{Q}^2 _I -$representation.  $\Theta$ is called as a $H^2$-representation if it satisfy the conditions $(1)\sim (5)$ in Theorem 5.2.

\end{defi}

\section{Construction of general $H-$Representations}

Comparing to construction of nonintersecting $H-$representations, it seems hard to construct a $H-$representation $( \bar{V}_{i} , A^i _j  )$ for which $\bar{V}_{i,j}$ do not disappear, even in low rank cases. In this section a procedure to construct general $H-$representations is developed, roughly consisting the following steps.

\noindent \textbf{Step 1}. Chose a $(\Gamma;q)-$realizable  $H^2 -$representation $\Theta $ (Definition 5.1).

\noindent \textbf{Step 2}.  From $\Theta$ realize some pairs of spaces $V_i , V^{\#}_i , (i\in I) $,together with some auxiliary structures as follows.

\noindent $\bullet$ A nondegenerate pairing  $[-,-]_i : V_i \times V^{\#}_i \rightarrow \kappa $.

\noindent $\bullet$ Subspaces $V^0 _i \subset V_i$ and $ V^{\# 0}_i \subset V^{\#} _i $ ( refer to the setting above Theorem 5.2  ).

\noindent $\bullet$ Canonical decompositions $V_i \cong V_{i,j}\oplus V_{i/j} $ and $V^{\#}_{i}\cong V^{\#}_{i,j}\oplus V^{\#}_{i/j}$ compatible with $\Theta$ ( refer to the setting above Remark 1.2) such that $V^{\#}_{i,j} = ( V_{i/j} )^{\perp}$ and $ V^{\#}_{i/j} = ( V_{i,j} )^{\perp}  $.

\noindent \textbf{Step 3}.  Choose a set of morphisms $\phi^i _j : V^0 _i \rightarrow V_j $ and $\phi^{\# i }_j : V^{\# 0}_i \rightarrow V^{\#}_j $.

\noindent $\bullet$ They are candidates of $A^i _j |_{ V^0 _i }$ and $A^{\# i} _j |_{ V^{\# 0} _i }$. The choice should be compatible with $\Theta$ since $\{ A^{i,k}_{j,l} \}$ contain partial information of $A^i _j |_{ V^0 _i }$ and $A^{\# i} _j |_{ V^{\# 0} _i }$. The choice is not unique and depend on s set of linear morphisms $\{  \Delta_i : \bar{V}_0 \rightarrow V_{i/} \}_{i\in I} $. Here
$V_{i/}= \cap_{j\neq i} V_{i/j}$ and $\bar{V}^0$ is defined above Theorem 5.2.

\noindent $\bullet$ It naturally holds that for any $i,j,k$ and $v\in V_{j,k}  , \alpha \in V^{\#}_{j,k}$, $\phi^j _i (v)= \phi^k _i (v)$, $\phi^{\# j}_i (\alpha)=\phi^{\# k}_i (\alpha)  $. This guarantees conditions $(T^{i,j}_k )$ and $( T^{\# i,j}_k  )$ in Definition 1.6. So
 the left job is to extend $\phi^i _j $ and $\phi^{\# i} _j$ as in the next step, making sure that the conditions $( T^{( m_{i,j};q  )}_{i,j})$ hold.

\noindent \textbf{Step 4}.  Extend $\phi^i _j$ to $A^i _j : V_i \rightarrow V_j $ and $\phi^{\# i}_j $ to $A^{\# i} _j : V^{\#}_i \rightarrow V^{\#}_j $ such that $A^{\# j} _i $ is the dual of $A^i _j$, and $ ( A^i _j , A^j _i   )  $ ( $ ( A^{\# i} _j , A^{\# j} _i   )  $   ) is a $(m_{i,j} ;q)$-couple between $V_i $ and $V_j$  ( $V^{\#}_i $ and $V^{\#}_j$   ).

\begin{rem}
This procedure is still far from "constructing all $H-$representations of type $(\Gamma;q)$", because if $\phi^i _j $ and $\phi^{\# i}_j$ in (3) is a "bad choice" then expected extension in (4) would not exist.  Objects constructed in (1)$\sim$(3) constitute a six-tuple
$[V_{i} , V^{0}_{i}; V_{j}, V^0 _{j}; \phi^i _j , \phi^j _i   ]$ which is a new conception called as \textbf{cross morphism}. A thorough study of cross morphism enable us to transmit  "determining whether $\{  \phi^i _j , \phi^{\# i}_j  \}$ is a good choice " to "simultaneous solvability of a set of linear equation systems". If those equation systems are solvable, based on a solution we can construct the expected extension in (4) and construct a $H-$representation successfully.
\end{rem}

\subsection{Space realization of decorated $\mathbb{Q}^2 _{I}-$representations }

Here are some explains about notations.

(a) If $X$ is some notation related to $H$-representation $(V_i , A^{i}_{j})$, then denote the corresponding notation related to
$(V^{\#}  , A^{\# i}_j )$ as $X^{\#}$. The symbol $\#$ here means on the dual side.

(b) If there are some non degenerate bilinear forms defined earlier, $(-,-)_1 : V_1 \times W_1 \rightarrow \kappa $ and $ (-,-)_2 : V_2 \times W_2 \rightarrow \kappa$, if $f\in Hom( V_1 , V_2  ) $, then $f^* \in Hom (W_2 , W_1)$ means the morphism satisfying
$(f(v) ,w  )_2 = ( v, f^* (w)  )_1  $ for any $v\in V_1 , w\in W_2$.

(c) If there is a bilinear form $(-,-): V\times W \rightarrow \kappa$ defined earlier, then for subspace $W_1 \subset W$, $(W)^{\perp}_{(-,-)} \subset V $ mean the subspace $\{ v\in V | (v ,w)=0 $ for any $ w\in W \}  $.

(d) If there is a bilinear form $(-,-): V\times W\rightarrow \kappa$ defined earlier, then $\ker( -, W  )\subset V$ means the subspace
$\{  v\in V | (v,w)=0 $ for any $ w\in W  \}$. Similarly   $\ker( V, -  )\subset W$ means the subspace
$\{  w\in W | (v,w)=0 $ for any $ v\in V  \}$.

Let $\Theta = [ ( V_{i,j} , A^{i,j}_{k,l} ) ,  \cup_{i\in I} \{ U_i ,W_i \}    ] $ be a $H^2$-representation as in Definition 5.4.  Let $\Theta^{\#} = [ ( V^{\#} _{i,j} , A^{\# i,j}_{k,l} ) ,  \cup_{i\in I} \{ U^{\#} _i ,W^{\#}_i \}    ] $ be the dual of $\Theta$.
Let $V^0 _i = \oplus_{k} V_{ik} / U_i  $ and $ V^{\# 0}_i = \oplus_{k} V^{\#} _{ik} / U^{\#} _i   $. There are two related short exact sequences as follows.
 $0\rightarrow U_i \rightarrow \oplus_{j:j\neq i} V_{i,j} \rightarrow V^0 _i \rightarrow 0  $,
$0\rightarrow U^{\#} _i \rightarrow \oplus_{j:j\neq i} V^{\#} _{i,j} \rightarrow V^{\# 0} _i \rightarrow 0   $.  Several bilinear forms are defined as follows.

\noindent $\bullet$ $ [-,-]^{0} _i : \oplus_{j:j\neq i} V_{i,j} \times \oplus_{j:j\neq i}  V^{\#} _{i,j}   \rightarrow \kappa  $ is the natural pairing between a linear space and its dual. So $ U^{\#}_i = (W_i )^{\perp}_{ [-,-]^{0} _i} $ and $ W^{\#}_i = ( U_i )^{\perp}_{ [-,-]^{0} _i}  $, and  $[-,-]^{0} _i$ induces the following  perfect pairings.

\noindent $\bullet$ $(-,- )_i : V^0 _i \times W^{\#}_i \rightarrow \kappa   $;  $ (-,- )^{'}_i : W_i \times V^{\#0}_i \rightarrow \kappa. $

\noindent $\bullet$  $[-,-]_i : V^0 _i \times V^{\# 0}_i \rightarrow \kappa $ is defined as follows. Recall the morphism (above Theorem 5.2)

\noindent$\Phi^j _i : \oplus_{l:l\neq j} V_{j,l} \rightarrow  \oplus_{k:k\neq i} V_{i,k} $. Since $\Theta$ is a $H^2 -$rerpesentation, there hold   $\ker (\Phi^j _i  )\subset U_j  $ and $ Im (\Phi^j _i )\subset W_i $ by Theorem 5.2.  So $\Phi^j _i$  induces a morphism
 $\bar{\Phi}^j _i : V^0 _j \rightarrow W_i $. Now for  $v\in V^0 _i $ and $\alpha \in V^{\# 0}_i $,   set
 $ [ v ,\alpha ]_i =( \bar{\Phi}^i _i (v),  \alpha   )^{'} _i  $.

In the dual side a morphism $\bar{\Phi}^{\# i} _i : V^{\# 0}_i \rightarrow W^{\#} _i $ is defined similarly. Denote $Im( \bar{\Phi}^i _i )$  ($Im ( \bar{\Phi}^{\# i} _i  )   $   ) as  $W^0 _i$ ( $ W^{\# 0}_i  $   ). For later use in $W_i$ ( $W^{\# }_i $   ) choose a subspace
 $W^{'}_i$ ( $W^{\# '}_i $ ) complementary to $W^0 _i$(  $W^{\# 0}_i$  ) , so $W_i = W^0 _i \oplus W^{'}_i $ ( $W^{\#}_i = W^{\# 0}_i \oplus W^{\# '} _i  $  ).

\begin{lem} Assume notations  as above. Then

\noindent(1) For any $v\in V^0 _i $ and  $\alpha \in V^{\# 0}_i$, there is $( v,  \bar{\Phi}^{\# i}_i (\alpha )   )_i  = ( \bar{\Phi}^i _i (v) ,\alpha  )^{'} _i    $.

\noindent(2) $\ker [-, V^{\# 0}_i  ]_i = ( W^{\# 0}_i )^{\perp} _{(-,-)_i   } $ and $\ker [ V^0 _i  ,-]_i = ( W^0 _i  )^{\perp} _{(-,-)^{'}_i }  $.

\end{lem}

\begin{pf}
(1) can be proved essentially by the relation $ A^{\# i,k} _{j,l} = ( A^{j,l} _{i,k}   )^{*}   $. Since

 \noindent $ [ v ,\alpha ]_i =( \bar{\Phi}^i _i (v),  \alpha   )^{'} _i =( v,  \bar{\Phi}^{\# i}_i (\alpha )   )_i  $ and $W^{\# 0}_i = Im ( \bar{\Phi}^i _i )$, the first identity of (2) follows. The second identity can be proved similarly.\\

\end{pf}

 As the next step we extend $V^0 _i ( V^{\# 0}_i  )$ to larger space $V_i (V^{\#}_i )$, and extend the bilinear form $[-,-]_i$ to a non degenerate bilinear form between $V_i$ and $V^{\#}_i$.

 Since $W^{\#} _i = W^{\# 0}_i \oplus W^{\# '}_i $, by the perfect pairing $(-,-)_i$, $V^0 _i$ gets a decomposition
$V^0 _i = (W^{\# 0}_i )^{\perp}_{(-,-)_i } \oplus (W^{\# '}_i )^{\perp}_{(-,-)_i} $. In the dual side similarly there is a decomposition
$V^{\# 0}_i = (W^{ 0}_i )^{\perp}_{(-,-)^{'}_i } \oplus (W^{ '}_i )^{\perp} _{(-,-)^{'}_i }  $. By (2) of the lemma 6.1, the rank of $[-,-]_i $
equals $\dim V^0 _i - \dim \ker [-, V^{\# 0}_i  ]_i =\dim V^0 _i - \dim ( W^{\# 0} _i   )^{\perp}_{(-,-)_i } = \dim ( W^{\# '} _i   )^{\perp}_{(-,-)_i }$. Similarly, the rank of the bilinear form $[-,-]_i $ equals
$\dim (  W^{'} _i  )^{\perp}_{(-,-)^{'}_i }$. So $\dim (  W^{'} _i  )^{\perp}_{(-,-)^{'}_i }= \dim ( W^{\# '} _i   )^{\perp}_{(-,-)_i }  $.  Furthermore, the pairing by restriction of $[-,-]_i$:
$( W^{\# '} _i  )^{\perp}_{(-,-)_i } \times (  W^{'} _i  )^{\perp}_{(-,-)^{'}_i } \rightarrow \kappa   $ must be nondegenerate, otherwise $\ker [-, V^{\# 0}_i  ]_i$
would be larger than $ (W^{\# 0}_i )^{\perp}_{(-,-)_i }  $.
Let $l= \dim  (  W^{'} _i  )^{\perp}_{(-,-)^{'}_i } $,
$l_1 =\dim  (  W^{\# 0} _i  )^{\perp}_{(-,-)_i } =\dim W^{\# '}_i $ and
$l_2 =\dim  (  W^{0} _i  )^{\perp}_{(-,-)^{'}_i }=\dim W^{'}_i $.

The spaces $V_i , V^{\#}_i$ and the pairing between them are defined as follows.

\noindent $\bullet$ Take a linear space $V^{'}_i $ ($ V^{\# '}_i  $   )of dimension $l_2 $( $l_1$ ) , together with a linear isomorphism
$\iota_i :V^{'}_i \rightarrow W^{'}_i $(  $ \iota^{\#} _i : V^{\# '}_i \rightarrow  W^{\# '}_i $  ) .

\noindent $\bullet$ Take any bilinear form $<-,->^{'}_i : V^{'} _i \times V^{\# '} _i \rightarrow \kappa   $.

\noindent $\bullet$ For any integer $d_i \geq 0$, take two linear spaces $ V^{"} _i  $ and $V^{\# "}_i $ of the same dimension $d_i$. Let
 $<-,->^{"}_i : V^{"} _i \times V^{\# "} _i \rightarrow \kappa $ be any non degenerate bilinear form.

\noindent $\bullet$ let $V_i = V^0 _i \oplus V^{'} _i \oplus V^{"} _i $ and  $V^{\#} _i = V^{\# 0} _i \oplus V^{\# '} _i \oplus V^{\# "} _i $. A bilinear form $[-, -]_i : V_i \times V^{\#} _i \rightarrow \kappa  $ is defined as follows ( we use the same notation $[-,-]_i$ without making any confusion) .
 For any $v+v^{'}+v^{"}\in V_i$ where $v\in  V^0 _i,  v^{'}\in  V^{'} _i  , v^{"}\in  V^{"} _i  $, and any
 $\alpha  +\alpha^{'} +\alpha^{"} \in V^{\#}_i $ where $\alpha\in V^{\# 0} _i ,\alpha^{'}\in V^{\# '} _i , \alpha^{"}\in V^{\# "} _i  $,

$ [v+v^{'}+v^{"}, \alpha  +\alpha^{'} +\alpha^{"}  ]_i = [v,\alpha]_i + ( \iota_i (v^{'}) , \alpha )^{'}_i + (v, \iota^{\#}_i (\alpha^{'})_i )_i + < v^{'} , \alpha^{'}  >^{'}_i + < v^{"} ,\alpha^{"}   >^{"}_i  $.

  \begin{prop}
  Above bilinear form $[-,-]_i$ is non degenerate.

  \end{prop}

  \begin{pf}
  Chose a basis of $ V_i $   as $( \vec{v}^0 _{1}, \vec{v}^0 _{2}, \vec{v}^{'} ,\vec{v}^{"}    )$ where  $\vec{v}^0 _{1}   , \vec{v}^0 _{2} , \vec{v}^{'}  ,\vec{v}^{"} $  are some basis of subspace $  (W^{\# 0}_i )^{\perp}_{(-,-)_i }, (W^{\# '}_i )^{\perp}_{(-,-)_i} ,
  V^{'}_i , V^{"}_i $  respectively.  Also chose a basis of $ V^{\#}_i $   as
  $( \vec{v}^{\# 0} _{1}, \vec{v}^{\# 0} _{2}, \vec{v}^{\# '} ,\vec{v}^{\# "}    )$ where  $\vec{v}^{\# 0} _{1}   , \vec{v}^{\# 0} _{2} , \vec{v}^{\# '}  ,\vec{v}^{\# "} $  are some basis of subspace $  (W^{ 0}_i )^{\perp}_{(-,-)_i }, (W^{ '}_i )^{\perp}_{(-,-)_i} ,
  V^{\# '}_i , V^{\# "}_i $  respectively. By using these two basses, $[-,-]_i$ is presented by the following block matrix.
  \[
  \left(
    \begin{array}{cccc}
      0 & 0 & A_2 & 0 \\
      0 & A_1 & A_4 & 0 \\
      A_3 & A_5 & A_6 & 0 \\
      0 & 0 & 0 & A_7 \\
    \end{array}
  \right)
  \]

  Where each row presents the pairing of some element in the first basis  with elements of the second basis. By the previous construction, $A_1 , A_2 , A_3 ,A_7$ are nondegenerate square matrices so this matrix is nondegenerate.\\
  \end{pf}

  By (4) of Theorem 5.2 and Definition 5.4, the space $V^0 _i$, and those spaces $V_{i,j}$ imbed  in $V_i$ naturally. Denote the images of these embedding morphisms still as  $V^0 _i$, $V_{i,j}$ without making any confusion.  In the dual side  subspaces $V^{\# 0} _i  $, $V^{\#}_{i,j }$ of $V^{\#}_i$ are defined similarly.

  Let $p^i _{i,j} : V_i \rightarrow V_{i,j} $ be the morphism  defined as follows.

   \[
p^i _{i,j} (v)= \begin{cases}
   \varrho _{i,j} \circ \bar{\Phi}^i _i  (v), & \text{ when $v\in V^0 _i$;} \\
    \varrho_{i,j} \circ \iota_i (v), & \text{ when $v\in V^{'} _i $.}\\
    0,& \text{when $v\in V^{"}_i $. }
\end{cases}
\]

  Where $\varrho_{i,j} : \oplus_{j:j\neq i  } V_{i,j} \rightarrow V_{i,j}  $ is  the projection morphism  to $ V_{i,j} $.

  When $v\in V_{i,j} \subset V_i $, $p^i _{i,j} (v) =  \varrho _{i,j} \circ \bar{\Phi}^i _i  (v)= v$. So $p^i _{i,j}$ is a projector to the subspace $V_{i,j}$. Denote $\ker( p^i _{i,j} )$ as $ V_{i/j} $. In the dual side a morphism
  $p^{\# i}_{i,j}: V^{\# } _i \rightarrow V^{\#} _{i,j} $ is defined in similar way, and $\ker( p^{\# i}_{i,j} )$ is denoted as $V^{\#}_{i/j} $.

  Let  $p^i _i : V_i \rightarrow \oplus_{j:j\neq i} V_{ij} $ be the morphism defined by
   $p^i _i (v) = \sum_{j: j\neq i} p^i _{ij}(v) $ for any $v\in V_i $. It is easy to see $p^i _i (v)= \bar{\Phi}^i _i (v)$ if $v\in V^0 _i$.

  \begin{prop} Assume notations as above.

  \noindent (1) With respect to  $[-,- ]_i $, for any $j:j\neq i$, $V^{\#}_{i,j}= (  V_{i/j} )^{\perp} $  and $ V^{\# }_{i/j} = (  V_{i,j} )^{\perp}  $.

 \noindent (2) For any $j,k\neq i $,  $p^i _{i,j} |_{ V_{i,k} } : V_{i,k} \rightarrow V_{i,j} $ equals $A^{i,k} _{i,j} $, and
  $ p^{\# i} _{i,j} |_{ V^{\#}_{i,k}  }$ equals $ A^{\# i,k } _{i,j}   $.

  \noindent (3) For any $j: j\neq i$, $V^{"} _i \subset V_{i/j}$, $V^{\# "} _i \subset V^{\# } _{i/j}$.

  \noindent(4) $Im(p^i _i ) = W_i $. Denote $\ker( p^i _i ) $ as $V_{i/}$.

  \noindent (5) For $v\in V_i$, if $p^i _i (v)= w\in \oplus_{l:l\neq i} V_{il}$, then $p^i _{il}(v)= \rho_{il}(w)$.

  \end{prop}

\begin{pf}
 The proof of $V_{i/j} = ( V^{\#}_{ij}  )^{\perp}  $ is as follows. Suppose $v= \sum_{k: k\neq i} v_{ik} +v_2 +v_3 \in V_i$ where $v_{ik} \in V_{ik}, v_2 \in V^{'}_i , v_3 \in V^{"}_i  $. By definition of $ p^i _{ij} $, $v\in \ker(p^i _{ij})$ if and only if

$0= \varrho_{ij}\circ \bar{\Phi}^i _i (\sum_{k:k\neq i} v_{ik}   ) + \varrho_{ij}\circ \iota_i (v_2)
 = \sum_{k:k\neq i} A^{ik}_{ij} (v_{ik}) + \rho_{ij} \circ \iota_i (v_2 ) .$

On the other hand by definition of $[-,-]_i$, $v\in ( V^{\#}_{ij}  )^{\perp} $ if and only if for any $\alpha \in V^{\#} _{ij} $,

 $ 0=[v,\alpha  ]_i =  [\bar{\Phi}^i _i ( \sum_{k:k\neq i} v_{ik}  ),\alpha   ]^0 _i + (\iota_i (v_2 ), \alpha  )^{'}_i ,$
 which is equivalent to

 $\sum_{k:k\neq i} A^{ik}_{ij} (v_{ik}) + \rho_{ij} \circ \iota_i (v_2 )=0$.

 Apply the same argument to dual side gives $ V^{\#} _{i/j} = ( V_{ij}  )^{\perp}   $ so (1) is proved.

 (2) holds because, for any $v\in V_{ik}$, $p^i _{ij} (v) = \varrho_{ij} \circ \bar{\Phi}^i _i (v) = A^{ik}_{ij}(v)$.

 By definition of $p^i _{ij}$, for any $v\in V^{"}_{ij}$, $p^i _{ij} (v)=0$, so (3) follow.

 For (4), by definition of $p^i _{i,j}$, we see $p^i _i ( V^0 _i )= W^0 _i , p^i _i (V^{'}_i )=W^{'}_i  $ and $p^i _i (V^{"}_i )=0 $.
 Because $W_i = W^0 _i \oplus W^{'}_i$, therefore  $Im(p^i _i ) = W_i $.

  (5) is just by definition of $p^i _i$.\\
\end{pf}

\begin{lem} Assume notations as above. If $v \in V _i $, and
$\alpha =\sum_{k:k\neq i} \alpha_{ik}   \in V^{\# 0}_i$,where $\alpha_{ik}\in V^{\# 0}_{ik}$,
 then $[ v, \alpha  ]_i = [ p^i _i (v) , \sum_{k:k\neq i} \alpha_{ik}    ]^0 _i  $.
\end{lem}

\begin{pf} Since $V_i = V^0 _i \oplus V^{'}_i \oplus V^{"}_i $, it is enough to prove this statement in cases $v \in V^0 _i$, $v \in V^{'}_i $ and $v \in V^{"}_i $.   When $v \in V^0 _i $, by definition
$[ v, \alpha   ]_i = ( \bar{\Phi}^i _i (v) , \alpha    )^{'}_i =[ p^i _i (v) , \sum_{k:k\neq i} \alpha_{ik}  ]^0 _i $.

When $v\in V^{'}_i  $, by definition of $[-,-]_i$,
$[v, \alpha   ]_i = ( \iota_i (v) , \alpha  )^{'} _i $. On the other hand

\noindent $[ p^i _i (v) , \sum_{k:k\neq i} \alpha_{ik}    ]^0 _i = ( \sum_{l:l\neq i} p^{i}_{i,l}(v), \sum_{k:k\neq i}\alpha_{ik}  )^0 _i
= ( \sum_{l:l\neq i} \rho_{i,l}\circ \iota_i (v) , \sum_{k:k\neq i} \alpha_{ik}    )^0 _{i} = ( \iota_i (v) , \alpha  )^{'} 9_i  $.

When $v\in V^{"}_i $, then $[ v, \alpha  ]_i =0 $. On the other hand by (3) of proposition 6.2,  $v\in V_{i/j}$ for any $j\neq i$, thus $p^i _i (v)=0$, and hence  $[ p^i _i (v) , \sum_{k:k\neq i} \alpha_{ik}    ]^0 _i =0=[ v, \alpha  ]_i $.

\end{pf}

\begin{rem}  Since in a $H-$representation $( V_i , A^i _j )$, if $m_{i,j}$ is odd, then $A^i _j$ is an isomorphism. So the integers
 $\{ d_i = \dim V^{"}_i \}_{i\in I}$ should be chosen so that $\dim V_i = \dim V_j$ whenever $m_{i,j}$ is odd.

\end{rem}

\subsection{Partial realization of the jumping operators}

let $\Theta = [ ( V_{i,j} , A^{i,j}_{k,l} ) ,  \cup_{i\in I} \{ U_i ,W_i \}    ] $ be a $H^2$-representation. Assume that the second step  of construction of a $H-$representation has been completed. This yields some objects like spaces $V_i , V^{\#}_i$, pairings $[-,-]_i$, subspaces $V^0 _i ,  V_{i,j} , V_{i/j} , V_{i/}$. Based on them, in this section we will carry out the third step  of the construction , that is, to construct suitable morphisms $\phi^i _j : V^0 _i \rightarrow V_j $ and $\phi^{\# i }_j : V^{\# 0}_i \rightarrow V^{\#}_j $.
As restriction of $A^i _j$ ( $ A^{\# i}_j $  ), $\phi^i _j $ ( $ \phi^{\# i} _j $  ) should satisfy the constraints from $\Theta$ listed in the following proposition.

  \begin{prop} Suppose $(V_i , A^i _j )$ is a $H-$representation of type $[\Gamma;q]$. Assume notations $V^0 _{i}, V_{ij}, V_{i/j}, A^{ik}_{jk}$ and $p^i _{il}$ as section 6.1. Denote $A^i _j |_{V^0 _i }: V^0 _i \rightarrow V_j  $ as $\phi^i _j$, then

  \noindent(1) In cases $j=i$, $\phi^i _i (v) =v $ for any $v\in V^0 _i $.

  \noindent(2) For any $k\neq i$, $l\neq j$, and any $v\in V_{ik}$,  $p^j  _{jl}\circ \phi^i _j (v) = A^{ik}_{jl}(v) $ for any $v\in V_{ik}$.

  \noindent(3) For any $k:k\neq i,j$, and for any $v\in V_{ik}$, $\phi^i _j (v)= \phi^k _j \circ A^i _k (v)  $ .

  \noindent(4) When $m_{ij}$ is odd, $\phi^i _j$ is injective.

  \end{prop}

  \begin{pf}
  (1) is true since by definition  $A^i _i = id_{V_i }$. (2) follows by definition of the morphism $A^{ik}_{jl}$. When $m_{ij}$ is odd, $A^i _j$ is bijective, so $\phi^i _j$ is injective and (4) follows. (3) follows from the triangular condition $(T^{ik}_j )$.  \\
  \end{pf}

 Considering the short exact sequence  $V_{i/}\rightarrowtail V_i \stackrel{p^i _i }{\twoheadrightarrow} W_i $, suppose
 $\eta_i : W_i \rightarrow V_i$ is an injective morphism such that $p^i _i \circ \eta_i = id_{W_i}   $.Denote $Im (\eta_i)$ as $\bar{W}_i$.
  $\bar{W}_i$ is a complementary space of $V_{i/}$ in $V_i$ and can be regarded as a lifting of $W_i$.

  Now for any $j\neq i\in I$, let $\bar{\phi}^{ j} _i =\eta_i \circ \bar{\Phi}^j _i : V^0 _j \rightarrow V_i$. Applying the same construction in the dual side introduces objects $ \bar{W}^{\#}_i , \eta^{\#}_i $, and    produces $\bar{\phi}^{ \# j}_i : V^{\# 0} _j \rightarrow V^{\# } _i $.

  \begin{lem}
  For any $v\in V^0 _i$, $v-\bar{\phi}^{ i} _i (v) \in V_{i/}$.
  \end{lem}

  \begin{pf}
  By definition of $V_{i/}$, it is enough to show $P^{i} _{ij} (v) = P^i _{ij} (\bar{\phi}^{ i} _i (v)   )   $ for any $j: j\neq i$, for any $v\in V^0 _i$.  Let $v\in V^0 _i$, then $p^i _i (v)= \bar{\Phi}^i _i (v)$. So
   $p^i _i (v)= p^i _i ( \eta_i \circ \bar{\Phi}^i _i (v)  )  $ and $v-\bar{\phi}^{ i} _i (v) \in \ker(p^i _i) = V_{i/}$.\\
  \end{pf}

  Recall the space $\bar{V}^0 =  \oplus _{ \{i,j\} \in I^2   } V_{ij} / \sum_{i\in I} U_i $. For any $i\in I$, the space $V^0 _i$ is naturally imbedded in $\bar{V}^0$ as a subspace which is also denoted as $V^0 _i$.  Now fix some $i\in I$, set

  \[ \mathscr{B}_i = \{ \Delta_i : \bar{V}^0 \rightarrow V_{i/} |\  \Delta_{i} \text{ satisfies }  (1),(2) \text{ as  follows }   \}. \]

   \noindent(1) $\Delta_i (v)= v-\bar{\phi}^{ i} _i (v)  $ for any $v\in V^0 _i$,

   \noindent (2) $\ker(\Delta_i )\cap \ker ( \bar{\Phi}^j _i  )=0 $ for any $j: j\neq i$ and  $m_{ij}$ being odd.

   Then chose some $\Delta_i \in \mathscr{B}_i $,  for any $ j:j\neq i $, let $\phi ^j _i : V^0 _j \rightarrow V_i$ by $ \phi ^j _i (v)= \bar{\phi}^j _i (v) + \Delta_i (v) $.

  \begin{lem} Assume $\{ \phi^j _i \} $ as morphisms defined as above.

  \noindent(1) For any $i\in I$ and any $v\in V^0 _i  $, $ \phi^i _i (v) =v $.

   \noindent(2) For any $k\neq j$ and any $l\neq i$, for any $v\in V_{jk}\subset V^0 _j$,  $p^i _{il} \circ \phi^j _i (v) =  A^{jk}_{il}(v) \in V_{il}$.

   \noindent(3) For any $k, j\neq i$ and $k\neq j$, and any $v\in V_{jk}=V_{kj}$, $ \phi^k _i (v)= \phi^j _i (v) $.

  \end{lem}

  \begin{pf}
  For (1), let $v\in V^0 _i$. Then $\phi^i _i (v) = \bar{\phi}^i _i (v) + \Delta_i (v) =  \bar{\phi}^i _i (v) + v-  \bar{\phi}^i _i (v)=v.  $ The second equals sign is by (1) of the definition of $\mathscr{B}_i$.

  For (2),  because $\phi^j _i =  \bar{\phi}^j _i + \Delta_i |_{V^0 _j }=  \eta_i \circ \bar{\Phi}^j _i  + \Delta_i |_{V^0 _j } $, and $Im(\Delta_i) \subset V_{i/}$, so it is enough to prove
   $p^i _{il} ( \eta_i \circ \bar{\Phi}^j _i (v)  ) = A^{jk}_{il}(v) $, since there is
   $V_{i/}=\ker(p^i _i) =\cap_{l:l\neq i} \ker(p^i _{il} ) $. Now since
    $ p^i _i \circ \eta_i \circ \bar{\Phi}^j _i (v)= \bar{\Phi}^j _i (v) $, then by (5) of Proposition 6.2,
   $p^i _{il} (  \eta_i \circ \bar{\Phi}^j _i (v)  ) = \rho_{il}( \bar{\Phi}^j _i (v)  ) = A^{jk}_{il}(v)  $, so (2) is proved.

  For $v\in  V_{jk}$, there are $ \bar{\phi}^j _i (v)=\bar{\phi}^k _i (v)  $ and $\Delta_i |_{V^0 _j} (v) = \Delta_i |_{V^0 _k} (v)   $ so (3) follows. \\

  \end{pf}

  The following Lemma explains condition (2) in the definition of $\mathscr{B}_i$.

  \begin{lem} Assume notations $\bar{\phi}^j _i $ as above. Let $F: \bar{V}^0 \rightarrow V_{i/}$ be some linear morphism and let
  $\phi^j _i = \bar{\phi}^j _i + F|_{V^0 _j }$. Then $\phi^j _i$ is injective if and only if $\ker( \bar{\Phi}^j _i )\cap \ker(F)=0 $.
  \end{lem}

  \begin{pf} Firstly, if $\ker( \bar{\Phi}^j _i )\cap \ker(F)\neq 0 $ and contains a nonzero element $v$, then
  $\phi^j _i (v)= \bar{\phi}^j _i (v) + F|_{V^0 _j }(v) =0  $. Secondly if $v\in V^0 _j $ satisfying $\phi^j _i (v)=0 $, then
  since $\phi^j _i(v) = \bar{\phi}^j _i (v) + F|_{V^0 _j }(v)   $, $\bar{\phi}^j _i (v)\in \bar{W}_i   $, $ F|_{V^0 _j }(v) \in V_{i/}  $ and $V_i \cong \bar{W}_i \oplus V_{i/}$, so both $ \bar{\phi}^j _i (v)  $ and $F|_{V^0 _j }(v)   $ must be zero, hence
  $ v\in \ker( \bar{\Phi}^j _i )\cap \ker(F)=\{ 0\}  $ as $ \ker( \bar{\Phi}^j _i )= \ker{\bar{\phi}}^j _i $. \\

  \end{pf}

 The same constructions in the dual side produce a set of morphisms $\{ \phi^{\# j} _i : V^{\# 0}_j \rightarrow V^{\#}_i  \}   $.  By Lemma 6.4 and Lemma 6.5, the morphisms $\{ \phi^i _j \}$ and  $\{ \phi^{\# i} _j \}$ constructed above satisfy all the properties listed in Proposition 6.3.  In addition, we call the set of morphisms
  $\{ \phi^j _i \} $ and $\{  \phi^{\# j} _i \}$ as \textbf{$( \Gamma; q )-$extendable} if  there exists a morphism $A^i _j : V_i \rightarrow V_j  $ for every $  (i,j)\in (I \times I )  \setminus \{ (i,i)\}_{i\in I}$  , such that:

 \noindent (1) For any $i\neq j$, $\phi^i _j = A^i _j |_{V^0 _i }  $ and $ \phi^{\# j}_i = ( A^i _j   )^{*} | _{ V^{\# 0}_j  }  $.

 \noindent (2)  For any $i\neq j$, $( A^i _j , A^j _i  )$ is a $( m_{i,j} ;q )-$couple between $V_i$ and $V_j$.\\

The following lemma shows some kind of compatibility between $\phi^i _j$ and $\phi^{\# j}_i $.

\begin{lem} Let $v\in V^0 _i $, $\alpha \in V^{\# 0}_j $, then
 $[  v, \phi^{\# j} _i (\alpha )   ]_i = [ \phi^i _j (v) ,\alpha    ]_j   $.

\end{lem}

\begin{pf} Suppose $v=\sum_{k:k\neq i} v_{ik}$ and $\alpha=\sum_{l:l\neq j} \alpha_{jl}$, where $v_{ik}\in V_{ik}$ and $\alpha_{jl}\in V_{jl}$. Recall $ \phi^{\# j} _i (\alpha )= \bar{\phi}^{\# j}_i (\alpha )+ \Delta^{\#}|_{V^{\# 0}_j }(\alpha)  $. Since
$ \Delta^{\#}|_{V^{\# 0}_j }(\alpha)\in V^{\#}_{i/}  $, so $[ v, \Delta^{\#}|_{V^{\# 0}_j }(\alpha)  ]_i =0$ and

$[  v, \phi^{\# j} _i (\alpha )   ]_i = [  v,\bar{ \phi}^{\# j} _i (\alpha )  ]_i =[  v,\eta^{\#}_i \circ \bar{\Phi}^{\# j} _i (\alpha )  ]_i = [v, p^{\# i}_i \circ \eta^{\#}_i \circ \bar{\Phi}^{\# j} _i (\alpha )   ]^{0} _i $

$= [\sum_{k:k\neq i} v_{ik}, \bar{\Phi}^{\# j} _i ( \sum_{l:l\neq j} \alpha_{jl})    ]^0 _i =\sum_{k\neq i,l\neq j} [v_{ik} , A^{\# jl}_{ik}( \alpha_{jl} )   ]^0 _i $.

Where the third equals sign is by Lemma 6.2, since $v\in V^0 _i $.
Similarly, on the other hand

$[ \phi^i _j (v) ,\alpha    ]_j  = [ \bar{\phi}^i _j (v) ,\alpha  ]_j =
[\eta_j \circ \bar{\Phi}^i _j (v) ,\alpha   ]_j =
[ p^j _j \circ \eta_j \circ \bar{\Phi}^i _j (v) ,\alpha   ]^0 _j = [ \bar{\Phi}^i _j ( \sum_{k:k\neq i}v_{ik}  ) , \sum_{l:l\neq j} \alpha_{jl}  ]^0 _j $

$=\sum_{k\neq i,l\neq j} [ A^{ik} _{jl}(v_{ik}) , \alpha_{jl}   ]^0 _j = \sum_{k\neq i,l\neq j} [v_{ik} , A^{\# jl}_{ik}( \alpha_{jl} )   ]^0 _i .  $\\

\end{pf}

Since $V_{ij}\subset V^0 _i \subset V_i$, there exists a subspace $V^0 _{i/j}\subset V_{i/j}$ such that
$V^0 _i \cong V_{ij}\oplus V^{0}_{i/j}$. Similarly in the dual side there exists a subspace $V^{\#0}_{j/i} \subset V^{\#}_{j/i} $ such that
$V^{\#0 }_{j} \cong V^{\#}_{ji}\oplus V^{\#0}_{j/i} $.

\begin{lem} Assume above notations. Then $ \phi^i _j (v) \in V_{j/i}   $ for any $v\in V^0 _{i/j}$  , $ \phi^{\#j}_i (\alpha)\in V^{\#}_{i/j} $ for any
$\alpha \in V^{\#0}_{j/i} $, and the restriction of $\phi^j _i $ on $ V_{j,i} $ is an isomorphism to $V_{i,j}$.
\end{lem}

\begin{pf} Suppose $v=\sum_{k\neq i} v_{ik}$. Since $V_{j/i}=\ker (p^j _{ji} ) $,
it is enough to show $p^j _{ji} ( \phi^i _j (v) )=0 $, which is equivalent to
$p^j _{ji} ( \bar{\phi}^i _j (v) )= p^j _{ji}( \eta_j \circ \bar{\Phi}^i _j (v)  )= 0 $, because
$\phi^i _j (v) =\bar{\phi}^i  _j (v) + \Delta_j (v) $ and $\Delta_j (v)\in V_{j/}$.
Now by (5) of Proposition 6.2, $p^j _{ji}( \eta_j\circ \bar{\Phi}^i _j (v)  )=\rho_{ij}( \bar{\Phi}^i _j (v)   )
= \sum_{k\neq i} A^{ik}_{ji} (v_{ik})$.

On the other hand, because $v\in V_{i/j}$ so $0=p^i _{ij}(v)= \sum_{k\neq i}A^{ik}_{ij}(v_{ik})$. Since $V_{ij}=V_{ji}$ and $A^{ik}_{ji}=A^{ik}_{ij}$, so $p^j _{ji}( \eta_j \circ \bar{\Phi}^i _j (v)  )=0$ and the proof of the first statement is complete.

Now suppose $v\in V_{j,i}$, then
$\phi^j _i (v)= \bar{\phi}^j _i (v)+ \Delta_i |_{V^0 _j } (v)=\bar{\phi}^j _i (v) +v-\bar{\phi}^i _i (v)=v $ by (1) of the definition of
$\mathscr{B}_i$. So the second statement follows.\\
\end{pf}

For every pair $i,j\in I$, spaces $V_i , V^{\#}_i , V_j , V^{\#}_j $ and their subspaces $V^0 _i , V^{\# 0}_i , V^0 _j , V^{\# 0}_j   $, perfect pairings $[-,-]_i , [-,-]_j$, and morphisms $\phi^i _j , \phi^j _i , \phi^{\# i}_j ,\phi^{\# j}_i$ are constructed as above. We introduce the following conception of \textbf{"cross morphism"} to organize them. Denote the morphism
$\phi^i _j |_{ V^{0}_{i/j} }:V^{0}_{i/j}  \rightarrow V_{j/i}$ as $\psi^i _j $,  the morphism
$\phi^{\# i}_{j} |_{ V^{\# 0}_{i/j}  }: :V^{\# 0}_{i/j}  \rightarrow V^{\#} _{j/i}   $ as $ \psi^{\# i}_j $.   Denote the restricted perfect pairing $[-,-]_i : V_{i/j} \times V^{\#}_{i/j}\rightarrow \kappa$ also as $[-,-]_i$, and $[-,-]_j : V_{j/i} \times V^{\#}_{j/i}\rightarrow \kappa$ also as $[-,-]_j$.

\begin{defi}
A cross  morphism is a six-tuple $[V, V_0 ; W, W_0 ; h,g   ]$, where $V_0 \subset V , W_0 \subset W$ are linear spaces,
 and $h: V_0 \rightarrow W ; g: W_0 \rightarrow V$ are linear morphisms.
\end{defi}

\begin{defi} A \textbf{matched pair of cross morphisms} are two cross morphisms $[V, V_0 ; W, W_0 ; h,g   ]$,
$[V^{\#}, V^{\#}_0 ; W^{\#}, W^{\#}_0 ; h^{\#},g^{\#}   ]$, together with two perfect pairings $[-,-]_1 : V\times V^{\#} \rightarrow \kappa $ and $[-,-]_2 : W\times W^{\#} \rightarrow \kappa $ such that for any $v\in V_0  , \beta\in W^{\#}_0   $,  there is $[ h(v) ,\beta ]_2 =[ v, g^{\#}(\beta)  ]_1 $, and for any
$w\in W, \alpha\in V^{\#}$, there is $[w, h^{\#}(\alpha)   ]_2 =[g(w) , \alpha   ]_1  $.

Let $m\in \mathbb{Z}^{\geq 1}, q \in \kappa$, we call \textbf{this matched pair as $(m;q)$-extendable}, if there exist $\tilde{h}: V\rightarrow W$ and $\tilde{g}:W\rightarrow V$ such that

\noindent(1) $\tilde{h}|_{V_0} = h , \tilde{g}|_{W_0}=g $; $\tilde{h}^{*}|_{W^{\#}_0} = g^{\#} , \tilde{g}^{*}|_{V^{\#}_0}=h^{\#} $.

\noindent(2) $( \tilde{h} ,\tilde{g}   )$ is a specialized $(m;q)$-couple as in Definition 1.4.
\end{defi}

 $[V_{i/j} , V^{0}_{i/j}; V_{j/i}, V^0 _{j/i}; \psi^i _j , \psi^j _i   ]$ and
 $[V^{\#}_{i/j} , V^{\# 0}_{i/j}; V^{\#}_{j/i}, V^{\# 0} _{j/i}; \psi^{\# i} _j , \psi^{\# j} _i   ]$ constitute a typical example of matched pair. The needed pairings are  $[-,-]_i : V_{i/j} \times V^{\#}_{i/j}\rightarrow \kappa$ and $[-,-]_j : V_{j/i} \times V^{\#}_{j/i}\rightarrow \kappa$. By Lemma 6.6 they satisfy the conditions for a matched pair.

Let $\Theta = [ ( V_{i,j} , A^{i,j}_{k,l} ) ,  \cup_{i\in I} \{ U_i ,W_i \}    ] $ be a $H^2$-representation. Let
${ V_i , V^{\#}_i ,[-,-]_i  }_{i\in I } $ be a  space realization of $\Theta$. Assume notations $V_{i/j}, V^{\#}_{i/j}, \psi^i _j , \psi^{\# i} _j$ as above. Now assume for any $i, j \in I ,i\neq j$, the matched pair  $[V_{i/j} , V^{0}_{i/j}; V_{j/i}, V^0 _{j/i}; \psi^i _j , \psi^j _i   ]$ and
 $[V^{\#}_{i/j} , V^{\# 0}_{i/j}; V^{\#}_{j/i}, V^{\# 0} _{j/i}; \psi^{\# i} _j , \psi^{\# j} _i   ]$ is $(m_{i,j}; q)$-extendable. So there exist $B^i _j : V_{i/j} \rightarrow V_{j/i} , B^j _i :V_{j/i} \rightarrow V_{i/j} $ as in Definition 6.2.

 Now let $A^i _j : V_i \rightarrow V_j $ be defined by $A^i _j (v)= B^i _j (v) $ if $v\in V_{i/j}$ and $A^i _j (v) = v$ if $v\in V_{i,j}$ for any $i\neq j \in I$. So $A^i _j |_{V^0 _i}=\phi^i _j$ and $(A^j _i )^{*}|_{  V^{\# 0}_i }= \phi^{\# i}_j$.

 \begin{thm} Assume notations as above. Then $(  V_i , A^i _j )$ is a $H-$representation of type $(\Gamma ;q)$.

 \end{thm}

 \begin{pf}
 By choice of $B^i _j$ and definition of $A^i _j$,  for any $i\neq j\in I$, $(A^i _j , A^j _i)$ is a $(m_{i,j};q)-$couple between $V_i$ and $V_j$.

 For three different $i,j,k\in I$, for any $v\in V_{i,j}\subset V_i$, Since $A^i _k |_{V^0 _i}=\phi^i _k $ and $v\in V_{i,j}\subset V^0 _i$, so $A^i _k (v)= \phi^i _k (v) = \phi^j _k (v)= A^j _k ( A^i _j (v)  )$. The second equality is by construction of those $\phi^i _k$'s.
 On the other hand for any $\alpha \in V^{\#}_{i,j} \subset V^{\#}_i$, since $ ( A^k _i )^{*} |_{V^{\# 0}_i } = \phi^{\# i} _k  $ and $\alpha \in V^{\#}_{i,j} \subset V^{\# 0}_i$, so $( A^k _i )^{*}(\alpha )=  \phi^{\# i} _k (\alpha)=  \phi^{\# j} _k (\alpha)
 = ( A^k _j )^{*} ( ( A^j _i )^{*} (\alpha)   )  $.

 So the theorem is proved since $(  V_i , A^i _j )$  fulfills all conditions in Definition 1.6.
 \end{pf}

\subsection{Cross Morphisms}

Suppose  $[V, V_0 ; W, W_0 ; h,g   ]$ and $[V^{'}, V^{'}_0 ; W^{'}, W^{'}_0 ; h^{'},g^{'}   ]$ are two cross  morphisms. A isomorphism from the first one to the second one is a 2-tuple  $[ F_1 , F_2  ]$ where $F_1 : V\rightarrow V^{'} $
and $F_2 : W\rightarrow W^{'}$ are linear isomorphisms such that

\noindent(1) $F_1 |_{V_0}  $ is an isomorphism from $V_0$ to $V^{'}_0 $, $F_2 |_{W_0}  $ is an isomorphism from $W_0$ to $W^{'}_0 $.

\noindent(2) $F_2 \circ h = h^{'} \circ F_1 |_{V_0 } : V_0 \rightarrow W^{'}   $; $F_1 \circ g = g^{'} \circ F_2 |_{W_0 } : W_0 \rightarrow V^{'}   $.

If there exist an isomorphism from a cross morphism $[V, V_0 ; W, W_0 ; h,g   ]$ to another one
$[V^{'}, V^{'}_0 ; W^{'}, W^{'}_0 ; h^{'},g^{'}   ]$, then denote $[V, V_0 ; W, W_0 ; h,g   ]\sim [V^{'}, V^{'}_0 ; W^{'}, W^{'}_0 ; h^{'},g^{'}]   $. It is easy to see $\sim$ is an equivalence relation. A basic problem is to classify the equivalence classes of cross  morphisms.\\

\noindent \textbf{Adding a trivial factor to a matched pair.}

Suppose $[V , V_0; W ,W_0 ; h , g ]$ and $[ V^{'} , V^{'} _0 ; W^{'}, W^{'} _0 ; h^{'} , g^{'}  ]   $ are two cross morphisms matched by  nondegenerate pairings  $(-,-)_l$ and $(-,-)_r$.    Let $U_l ,U_r $ be two $\kappa-$linear spaces of dimension $k_l ,k_r $ respectively, and let $U^{*}_l , U^{*}_r$ be the dual space of $U_l ,U_r $ respectively. Let $(- ,-  )_{l,0} : U_l \times U^{*}_l \rightarrow \kappa   $ and
$ (- ,-  )_{r,0} : U_r \times U^{*}_r \rightarrow \kappa    $ be the natural pairings between a linear space and its dual.

Define two larger cross  morphisms $[ \tilde{V} , \tilde{V}^{'}; \tilde{W} ,\tilde{W}^{'} ; \tilde{h} , \tilde{g} ]$,
$[ \tilde{V}^{'} , \tilde{V}^{'}_0 ; \tilde{W}^{'} , \tilde{W}^{'} _0 ; \tilde{h}^{'} , \tilde{g}^{'}  ]   $ as follows.

\noindent $\bullet$ Let $\tilde{V} = V \oplus U_l   $, $\tilde{V}^{'} = V^{'} \oplus U^{*}_l $;  $\tilde{W} = W \oplus U_r $, $\tilde{W}^{'} = W^{'} \oplus U^{*}_r $.

 Denote the imbedding map from  $V , V^{'} , W  $ and $W^{'}$ to $\tilde{V}  , \tilde{V}^{'} , \tilde{W}  $ and $\tilde{W}^{'}$ as $\iota_{V } , \iota_{V^{'}} ,\iota_{W}   $ and $\iota_{W^{'} } $ respectively.
 In this way  $V_0  , V^{'} _0 , W_0  $ and $W^{'} _0$ can be also identified with subspaces with the same notations in $\tilde{V} , \tilde{V}^{'} , \tilde{W}  $ and $\tilde{W}^{'}$ respectively.

\noindent $\bullet$  Let $\tilde{V}_0 = V_0  , \tilde{V}^{'}_0 = V^{'}_0, \tilde{W}_0 =W_0 $ and $\tilde{W}^{'}_0 = W^{'}_0$.

\noindent $\bullet$ Let $\tilde{h} = \iota_{W }\circ h : \tilde{V}_0 = V^0 _1 \rightarrow \tilde{W }$. Similarly let $\tilde{g} =\iota_{V}\circ g  $, $\tilde{h}^{'} = \iota_{W^{'}} \circ h^{'}   $ and $ \tilde{g}^{'}= \iota_{V^{'} }\circ g^{'}  $.

\noindent $\bullet$ Let $[-,-]_l : \tilde{V} \times \tilde{V}^{'} \rightarrow \kappa$ be defined as

 $[ v_1 +w_1 , v_2 + w_2   ]_l = ( v_1 , v_2    )_l + ( w_1 , w_2  )_{l,0} $ for any $v_1 \in V , w_1 \in U_l , v_2 \in V^{'} $ and $w_2 \in U^{*}_l$.

\noindent $\bullet$ Let $[-,-]_r : \tilde{W} \times \tilde{W}^{'} \rightarrow \kappa$ be defined as

 $[ v_1 +w_1 , v_2 + w_2   ]_r = ( v_1 , v_2    )_r + ( w_1 , w_2  )_{r,0} $ for any $v_1 \in W , w_1 \in U_r , v_2 \in W^{'} $ and $w_2 \in U^{*}_r$.

 Then it is not hard to see that the combination of $[-,-]_l$ and $[-,-]_r$ is a matching  between $[ \tilde{V} , \tilde{V}_0 ; \tilde{W} ,\tilde{W}_0 ; \tilde{h} , \tilde{g}]$ and $[ \tilde{V}^{'} , \tilde{V}^{'}_0 ; \tilde{W}^{'} , \tilde{W}^{'}_0; \tilde{h}^{'} , \tilde{g}^{'}  ]   $. We call it as a a $(k_l ,k_r)$ trivial extension of the original one.\\

\noindent \textbf{Structure of cross morphisms .}Now let $[V, V_0 ; W, W_0 ; h,g   ]$ be a cross morphism. Let $v\in V_0$. If $h(v)\in W_0$, then $gh(v)$ is well defined. If $gh(v)\in V_0$ then
$hgh(v)$ is well defined. Let $[...gh]_i $ be the word of length $i$ with $h$ on the rightmost side, in which $h$ and $g$ appear alternatively. By induction we say that $[...gh]_i (v)$
is well defined if inductively $[...gh]_{i-1}(v)$ is well defined and lies in $W_0 $ or $V_0 $ depending on $i$ is even or odd. In the following, the
 notation $v\in (?)$ means $v$ is in $V$ or $W$, and  the notation  $v\in (?)_0 $ means $v$ is in $V_0$ or $W_0$.

Let $V^i _0 = \{ v\in V_0 | [...gh]_{i+1} (v) $ is well defined.$ \}$.

Let $W^i _0 = \{ w\in W_0 | [...hg]_{i+1} (w) $ is well defined.$ \}$.

It is evident that $V^{i+1}_0 \subset V^i _0$, $V^0 _0 = V_0$,$W^{i+1}_0 \subset W^i _0$, $W^0 _0 = W_0$.

Let $N_1 = \max \{\  i\  |  V^{i+1}_0 \subsetneq V^i _0   \}$ and  $N_2 = \max \{\  i\  |  W^{i+1}_0 \subsetneq W^i _0   \}$. Since $V^{N_1 +1}_0 = V^{N_1 +2}_0
=...$, we denote $ V^{N_1 +1}_0 $ as $V^{\infty}_0 $, and   denote $W^{N_2 +1}_0  $ as $ W^{\infty} _0 $ similarly.

\begin{prop} Assume notations as above. Then

\noindent(1)If $j\leq i$ then  $[...gh]_j (V^i _0 )\subset (?)^{i-j}_0 $.

\noindent(2)If a subspace  $U \subset V^{i}_0 $ satisfies $U\cap V^{i+1}_0 =0$, then for any $j\leq i+1$, $[...gh]_j$ is an isomorphism from $U$ to $ [...gh]_j (U)$. In addition $[...gh]_j (U) $ is a subspace in $(?)^{i-j} _0$ such that $[...gh]_j (U)\cap (?)^{i-j+1} _0 =0  $.

\noindent(3)$|N_1 -N_2 |\leq 1$.

\end{prop}

\begin{pf} For (1),let $j\leq i$, $v\in V^i_0$. Then $[...gh]_i(v)$ is well defined. If $j$ is even, Then $[...gh]_i(v)$ equals
$[...gh]_{i-j}( [...gh]_{j}(v)  )$, so $ [...gh]_{j}(v)\in V^{i-j}_0   $. If $j$ is odd, then $ [...gh]_i(v) $ equals
$ [...hg]_{i-j}( [...gh]_{j}(v)  )  $, so $ [...gh]_{j}(v)\in W^{i-j}_0   $.

For the first statement of  (2), if $[...gh]_j$ is not an isomorphism from $U$ to $ [...gh]_j (U)$, then there is a nonzero $v\in U$, such that $[...gh]_j (v)=0   $. Yet consequently $[...gh]_i (v) $ is well defined for any $i$ and $v\in V^{\infty}_0$, thus contradicts with the assumption $U\cap V^{i+1}_0 =0$.
For the second statement of (2), if  $[...gh]_j (U)\cap (?)^{i-j+1} _0 \neq0  $, then there is $v\in U$ such that $[...gh]_j (v) \in (?)^{i-j+1} _0  $. But then $[...gh]_{j+i-j+1+1}(v)$ is well defined, which means $ v\in V^{i+1}_0  $, thus contradicts with the choice of $U$.

For (3), without loss of generality suppose $N_1 \geq N_2$. Let $U\subset V^{N_1}_0 $ be a nonzero subspace such that
$V^{N_1} = V^{\infty}_0 \oplus U $. By (2), $h(U )$ is a nonzero subspace in $ W^{N_1 -1 }_0 $ and $h(U)\cap W^{N_1 }=0   $. So
$W^{N_1}_0 \subsetneq W^{N_1 -1}_0 $ and $N_1 -1 \leq N_2  $. \\

\end{pf}

Without loss of generality, suppose $N_1 \geq N_2$. Suppose $V^{N_1}_0 = V^{\infty}_0 \oplus X_{N_1 }$ for some nonzero subspace
$X_{N_1}\subset V^{N_1}_0$.
 Similarly suppose $W^{N_1}_0 = W^{\infty}_0 \oplus Y_{N_1}$ for some subspace $Y_{N_1}\subset W^{N_1}_0$. But now $Y_{N_1}$ could be the zero space.

We can decompose every $V^i _0$ and $W^i_0$ in this way inductively as follows.

\begin{thm}  There exist subspaces $X_i \subset V^{i}_0$, $Y_i \subset W^i_0  $ for every $i\leq N_1$ such that

For any $j\leq N_1$,  $V^j_0 =  V^{j+1}_0 \oplus X_j \oplus g(Y_{j+1})\oplus gh(X_{j+2})\oplus...\oplus [gh...]_{N_1 -j}( (?)_{N_1} )$.

For any $j\leq N_1$,  $W^j_0 =  W^{j+1}_0 \oplus Y_j \oplus h(X_{j+1})\oplus hg(Y_{j+2})\oplus...\oplus [hg...]_{N_1 -j}( (?)_{N_1} )$.

Besides for every term $[gh...]_k ( (?)_{k+j}  )$ (  $[hg...]_k ( (?)_{k+j}  )$) appearing in above identities, the morphism $[gh...]_k$
( $[hg...]_k$  ) is an isomorphism from $(?)_{k+j}$ to  $[gh...]_k ( (?)_{k+j}  )$

\noindent(  $[hg...]_k ( (?)_{k+j}  )$   )  .

\end{thm}

\begin{pf} We prove the following statement ($i$) by induction on $i$.

\textit{Statement ($i$): for any $j\geq i$, a subspace $X_j \subset V^j_0$ and a subspace $Y_j \subset W^j_0$ are chosen such that for any
$j\geq i$,$V^j_0 =  V^{j+1}_0 \oplus X_j \oplus g(Y_{j+1})\oplus gh(X_{j+2})\oplus...\oplus [gh...]_{N_1 -j}( (?)_{N_1} )$ and}

\noindent\textit{$W^j_0 =  W^{j+1}_0 \oplus Y_j \oplus h(X_{j+1})\oplus hg(Y_{j+2})\oplus...\oplus [hg...]_{N_1 -j}( (?)_{N_1} )$.}

First the statement ($N_1$) is true by the discussion before the theorem. Now assume the statement ($i+1$) has been proved.
Since in $V^{i+1}_0$ , there holds

$X_{i+1} \oplus g(Y_{i+2 }) \oplus ...\oplus [gh...]_{N_1 -i-1}( (?)_{N_1} ) \cap V^{i+2}_0 =0 ,$

by using (2) of Proposition 6.4, there holds

$h( X_{i+1} \oplus g(Y_{i+2 }) \oplus ...\oplus [gh...]_{N_1 -i-1}( (?)_{N_1} ) )=h( X_{i+1}) \oplus hg(Y_{i+2 }) \oplus ...\oplus [hg...]_{N_1 -i}( (?)_{N_1} )  ,$

 and its intersection with $ W^{i+1}_0   $ is zero.  So there is a subspace
 $Y_i \subset W^i_0  $ such that

 $W^i_0 =  W^{i+1}_0 \oplus Y_i \oplus h(X_{i+1})\oplus hg(Y_{i+2})\oplus...\oplus [hg...]_{N_1 -i}( (?)_{N_1} )$.

Similarly we can prove there is a subspace  $X_i \subset V^i _0 $ such that

 $ V^i_0 =  V^{i+1}_0 \oplus X_i \oplus g(Y_{i+1})\oplus gh(X_{i+2})\oplus...\oplus [gh...]_{N_1 -i}( (?)_{N_1} ) $. So the statement ($i$) is proved, therefore the induction is complete.

 The last statement is also proved by using (2) of Proposition 6.4.
 \\

\end{pf}

\noindent \textbf{The structures of $V^{\infty}_0$ and $W^{\infty}_0 $.}  The six-tuple $ [V^{\infty}_0 , V^{\infty}_0 ; W^{\infty} _0 , W^{\infty}_0 ; h|_{V^{\infty}_0 } , g|_{W^{\infty}_0}  ] $ constitutes a cross  morphism, since $h(V^{\infty}_0 ) \subset W^{\infty}_0 $ and $g ( W^{\infty}_0  )\subset V^{\infty}_0  $.

Let $V^{\infty}_0 [i]=\{ v\in V^{\infty}_0 |  [...gh ]_i (v)=0 \}  $, and  $W^{\infty}_0 [i]=\{ w\in W^{\infty}_0 |  [...hg ]_i (w)=0 \} .  $ Evidently $V^{\infty}_0 [i]\subset V^{\infty}_0 [i+1]   $, $W^{\infty}_0 [i]\subset W^{\infty}_0 [i+1] $,
$ h(V^{\infty}_0 [i]  ) \subset W^{\infty}_0 [i-1]  $ and $ g(W^{\infty}_0 [i]  ) \subset V^{\infty}_0 [i-1]  $.

Let $L_1 = \max \{ i| V^{\infty}_0 [i-1]\subsetneq V^{\infty}_0 [i]   \}$ and $L_2 = \max \{ i| W^{\infty}_0 [i-1]\subsetneq W^{\infty}_0 [i]
  \}$. Denote $V^{\infty}_0 [L_1 ]= V^{\infty}_0 [L_1 +1 ] =... $ as $ V^{\infty}_0 [\infty ]  $ and denote  $W^{\infty}_0 [L_2 ]= V^{\infty}_0 [L_2 +1 ] =... $ as $ W^{\infty}_0 [\infty ]  $.

\begin{lem}
Let $U\subset V^{\infty}_0 [i] $ be a subspace such that $U\cap V^{\infty}_0 [i-1] =0 $. Then for $j\leq i-1$,

 \noindent(1)The morphism from $U$ to $[...gh]_j (U) $  sending $v$ to $[...gh]_j (v)  $ is an isomorphism.

 \noindent (2)Further more, in $(?)^{\infty}_0 [i-j]  $, there is $[...gh]_j (U)\cap (?)^{\infty}_0 [i-j-1]=0   $.
\end{lem}

\begin{pf} Let $v\in U$, if $[...gh]_j (v)=0$, then $v\in V^{\infty}_0 [j] \subset V^{\infty}_0 [i-1]$, which contradicts with the definition of $U$. Thus the first statement is proved.  For the second statement, if $v\in U$ satisfies
$[...gh]_j (v)\in (?)^{\infty}_0 [i-j-1] $, then $ [...gh]_{i-1}(v)=0 $ so $v\in V^{\infty}_0 [i-1]$, so $v=0$ by the condition of $U$.\\

\end{pf}

Let $L= \max \{ L_1 ,L_2 \} $. There are subspace $S_{L} \subset V^{\infty}_0 [L ] $ such that
$V^{\infty}_0[L]\cong V^{\infty}_0[L-1] \oplus S_{L}$ , and subspace $T_{L}\subset W^{\infty}_0 [L]  $ such that
$W^{\infty}_0[L]\cong W^{\infty}_0[L-1] \oplus T_{L}$.   The following theorem equips every $V^{\infty}_0 [i] $ and $W^{\infty}_0 [i]$ with
a favorable decomposition, for $i\leq L$.

\begin{thm} There exist a sequence of subspaces $\{ S_j ,T_j  \}_{j=1,2,...,L }$ where $ S_j \subset V^{\infty}_0 [j]  $,
$ T_j \subset V^{\infty}_0 [j] $ such that for any $i\leq L$,

$V^{\infty}_0 [i]= V^{\infty}_0 [i-1]\oplus S_i \oplus g( T_{i+1} ) \oplus gh( S_{i+2} ) \oplus ...\oplus [gh...]_{L-i}(  (?)_L  )  $, and

$W^{\infty}_0 [i]= W^{\infty}_0 [i-1]\oplus T_i \oplus h( S_{i+1} ) \oplus hg( T_{i+2} ) \oplus ...\oplus [gh...]_{L-i}(  (?)_L  )  $ .

Further more, for every term $[gh...]_k ((?)_{i+k})$ ( $[hg...]_k ((?)_{i+k})$  ) appearing in the first (second) identity, $[gh...]_k $ ( $[hg...]_k $)is an isomorphism from $(?)_{i+k}$ to $[gh...]_k ((?)_{i+k})$

\noindent(  $[hg...]_k ((?)_{i+k})$   ) .
\end{thm}

\begin{pf} We prove the following statement [i] inductively.

Statement [i]:  There exist a sequence of subspaces $\{ S_j ,T_j  \}_{j=i,i+1 ,...,L }$ where $ S_j \subset V^{\infty}_0 [j]  $,
$ T_j \subset V^{\infty}_0 [j] $ such that for any $i\leq k \leq L$,

$V^{\infty}_0 [k]= V^{\infty}_0 [k-1]\oplus S_k \oplus g( T_{k+1} ) \oplus gh( S_{k+2} ) \oplus ...\oplus [gh...]_{L-k}(  (?)_L  )  $, and

$W^{\infty}_0 [k]= W^{\infty}_0 [k-1]\oplus T_k \oplus h( S_{k+1} ) \oplus hg( T_{k+2} ) \oplus ...\oplus [gh...]_{L-k}(  (?)_L  )  $.

The statement $[L]$ is evident. Suppose we have prove those statement [j] for all $i<j\leq L$. Now in $ V^{\infty}_0[i] $, by using Lemma 6.8 and the statement [i+1], $ g( T_{i+1} ) + gh( S_{i+2} ) + ...+ [gh...]_{L-i}(  (?)_L  )   $ is a direct sum
$  g( T_{i+1} ) \oplus gh( S_{i+2} ) \oplus ...\oplus [gh...]_{L-i}(  (?)_L  )  $ whose intersection with $ V^{\infty}_0[i-1] $ is zero. So There exists $ S_i \subset V^{\infty}_0 [i] $ such that

$V^{\infty}_0 [i]= V^{\infty}_0 [i-1]\oplus S_i \oplus g( T_{i+1} ) \oplus gh( S_{i+2} ) \oplus ...\oplus [gh...]_{L-i}(  (?)_L  )  $.

Similarly there exists $T_i \subset W^{\infty}_0[i] $ such that

$W^{\infty}_0 [i]= W^{\infty}_0 [i-1]\oplus T_i \oplus h( S_{i+1} ) \oplus hg( T_{i+2} ) \oplus ...\oplus [hg...]_{L-i}(  (?)_L  ) $.

So the statement [i] is proved, therefore the induction is complete. The last statement is also proved by using  Lemma 6.8. \\

\end{pf}

We call the cross partial morphism $ [V^{\infty}_0 , V^{\infty}_0 ; W^{\infty} _0 , W^{\infty}_0 ; h|_{V^{\infty}_0 } , g|_{W^{\infty}_0}  ] $ as split, if there exist $ V^{\infty '} _0 \subset V^{\infty}_0  $ and $ W^{\infty '} _0 \subset W^{\infty}_0   $ such that

\noindent(1) $V^{\infty}_0 = V^{\infty}_{0}[\infty] \oplus  V^{\infty '} _0$ and $W^{\infty}_0 = W^{\infty}_{0}[\infty] \oplus  W^{\infty '} _0$,

\noindent(2) $h (  V^{\infty '} _0   ) \subset W^{\infty '} _0  $  and $g(  W^{\infty '} _0   ) \subset V^{\infty '} _0 . \\$

The following theorem describe some features of $ [ V_{i/j} , V^{0}_{i/j} ; V_{j/i}, V^0 _{j/i} ; \psi^i _j , \psi^j _i     ]   $ as a cross morphism. For saving notations denote it as $[V, V_0 ; W, W_0 ; h,g   ]$, and denote the terms like $V^i _0 , X_i , V^{\infty}_0 ,
  V^{\infty}_0 [i], S_i ...$ associated with $ [ V_{i/j} , V^{0}_{i/j} ; V_{j/i}, V^0 _{j/i} ; \psi^i _j , \psi^j _i     ]   $ just as  $V^i _0 , X_i , V^{\infty}_0 ,  V^{\infty}_0 [i], S_i ...$ without introducing new notations.

  \begin{thm}
  (1) If $m_{i,j}$ is odd, then $h|_{V^{\infty}_0 }$ and $g|_{ W^{\infty}_0 }$ are isomorphisms. $g|_{ W^{\infty}_0 } h|_{V^{\infty}_0}$ is semisimple with eigenvalues in  $  \{ 4 cos(\frac{a\pi}{m_{i,j}})^2 \frac{1}{(v+v^{-1})^2} \} _{a=1,2,...,[\frac{m_{i,j}-1}{2} ]}  $.

  (2) If $m_{i,j}$ is even, then $( h|_{V^{\infty}_0 },g|_{ W^{\infty}_0 }   )$ is a specialized $(m,v)$-couple between $ V^{\infty}_0 $ and $W^{\infty}_0 $. Especially $ [V^{\infty}_0 , V^{\infty}_0 ; W^{\infty} _0 , W^{\infty}_0 ; h|_{V^{\infty}_0 } , g|_{W^{\infty}_0}  ] $ is split,  and $L_1 =L_2 =1$.

  \end{thm}

\subsection{Extending $\phi^i _j$ to $A^i _j$ .}

Now come back to $ [ V_{i/j} , V^{0}_{i/j} ; V_{j/i}, V^0 _{j/i} ; \psi^i _j , \psi^j _i     ]   $ and
$[ V^{\#} _{i/j} , V^{\# 0}_{i/j} ; V^{\#} _{j/i}, V^{\# 0} _{j/i} ; \psi^{\# i} _j , \psi^{\# j} _i     ] $.
 To simplify notation denote them  just as
$[V,V_0 ; W, W_0 ; h,g  ]$ and
 $[V^{\#}, V^{\#}_0 ; W^{\#}, W^{\#}_0 ; h^{\#} ,g^{\#}  ] $. Denote the bilinear form between $V$ and $V^{\#}$ as $[-,-]_l$, and the bilinear form between $W$ and $W^{\#}$ as $[-,-]_r$. Assume (notation from section 6.3) both  $[V^{\infty}_0 ,V^{\infty}_0 ; W^{\infty}_0, W^{\infty}_0 ; h|_{V^{\infty}},g|_{W^{\infty}}  ]$ and
$[V^{\# \infty}_0 , V^{\# \infty}_0 ; W^{\# \infty}_0 , W^{\# \infty}_0 , h^{\#}|_{V^{\# \infty}_0}, g^{\#}|_{W^{\# \infty}_0}   ]$ are split.

 Recall the terms  $V^i _0 , W^i _0 , N_1 , N_2 , X_i , Y_i $ in section 6.3 associated with  $[V,V_0 ; W, W_0 ; h,g  ]$. Denote the corresponding terms associated with $[V^{\#}, V^{\#}_0 ; W^{\#}, W^{\#}_0 ; h^{\#} ,g^{\#}  ] $ as
  $V^{\# i} _0 , W^{\# i} _0 , M_1
  $ $, M_2 , S_i , T_i   $ respectively. Subspaces $V_0 +g(W_0 )$ ($W_0 + h(V_0 ) $) and $V^{\#}_0 +g^{\#}(W^{\#}_0 )$
  ( $W^{\#}_0 +h^{\#}(V^{\#}_0 )$  ) and the pairing between them are important in this section. Theorem 6.2 equip both  of them with a "good" basis as follows.

\noindent $\bullet$ Choose any basis $\{ e^i _{v} \}_{v=1,..., K_i } $ of $X_i $, and any basis
$\{ f^j _{u} \}_{u=1,...,L_j}   $ of $Y_j$. On the dual side chose any basis $\{  \alpha^i _s \}_{s=1,..., K^{\#}_i } $ of $S_i$, and any basis
$\{ \beta^j _{t} \}_{t=1,...,L^{\#}_j }  $ of $T_j$.

\noindent $\bullet$ Denote $h( e^i _{v}  ) $ as $ \bar{f}^{i-1} _{v}  $, $g(  f^j _{u} ) $ as $\bar{e}^{j-1} _u $, $ h^{\#}(\alpha^i _{s})  $ as $\bar{\beta}^{i-1} _{s}$,
 and $ g^{\#} ( \beta^j _{t}  )  $ as $\bar{\alpha}^{j-1} _t $. Especially $h( e^0 _{v}  ) , g(f^0 _u ) ,$
  $ h^{\#}(\alpha^0 _s ), g^{\#}(\beta^0 _{t})  $ are denoted as $\bar{e}^{-1}_v , \bar{f}^{-1}_u , \bar{\alpha}^{-1}_s , \bar{\beta}^{-1}_t   $ respectively.

\noindent $\bullet$ Denote $[ [...gh]_{a}( e^i _v  )  , [...g^{\#} h^{\#} ]_{b}( \alpha^j _s )    ]_{*}   $ as $\Upsilon^{i,j}_{v,s}(a+b)_{0,0} $, where $0\leq a\leq i+1 $, $0\leq b\leq j+1$, $a\equiv b \mod 2 $. The symbol $*$ is $i$ if $a$ is even and is $j$ if $a$ is odd. $\Upsilon^{i,j}_{v,s}(a+b)_{0,0} $ is well defined since by Lemma 6.6, this pairing only depends on $a+b$. A series of numbers from pairings and the corresponding range of parameters are defined as follows.

\noindent$\Upsilon^{i,j}_{v,t}(a+b)_{0,1} = [ [...gh]_{a}( e^i _v  )  , [...h^{\#} g^{\#} ]_{b}( \beta^j _t )    ]_{*}   $  ,  $0\leq a\leq i+1 $, $0\leq b\leq j+1$, $2|a-b+1$.

\noindent $\Upsilon^{j,i}_{u,s}(a+b)_{1,0} =[ [...hg]_{a}( f^j _u )  , [...g^{\#} h^{\#} ]_{b}( \alpha^i _s )    ]_{*}   $ ,  $0\leq a\leq i+1 $, $0\leq b\leq j+1$, $2|a-b+1 $.

\noindent$\Upsilon^{j,i}_{u,t}(a+b)_{1,1} = [ [...hg]_{a}( f^j _u )  , [...h^{\#} g^{\#} ]_{b}( \beta^i  _t )    ]_{*}   $ ,  $0\leq a\leq j+1 $, $0\leq b\leq i+1$, $2|a-b $.\\

By the notation "$S \rhd [V_1 : V_2 ] $   " we mean $S$ is a basis of a complement of $V_2$ in $V_1$. By Theorem 6.2, there are the following facts.

\noindent  $\bigcup^{N_1} _{i=0} \{ [...gh]_{2a}e^i _v , [...gh]_{2a}\bar{e}^i _v  |  0\leq 2a\leq i+1 \}
 \bigcup \{  \bar{e}^{-1}_v  \} \rhd [V_0 +g(W_0 ) : V^{\infty}_0 ] $ .

\noindent $\bigcup^{N_2}_{i=0}\{ [...hg]_{2a}f^i _u ,[...hg]_{2a}\bar{f}^i _u  |  0\leq 2a\leq i+1 \} \bigcup \{ \bar{f}^{-1} _u \} \rhd
[W_0 +h(V_0 ): W^{\infty}_0 ]$.

\noindent $\bigcup^{M_1}_{i=0}\{ [...g^{\#}h^{\#} ]_{2a}\alpha^i _{s} , [...g^{\#}h^{\#}]_{2a}\bar{\alpha}^i _{t}  |  0\leq 2a\leq i+1 \} \bigcup \{ \bar{\alpha}^{-1}_s \} \rhd [V^{\#}_0 +g^{\#}(W^{\#}_0 ): V^{\# \infty}_0 ]$  .

\noindent $\bigcup^{M_2}_{j=0} \{ [...h^{\#}g^{\#} ]_{2b}\beta^j _{s} ,
 [...g^{\#}h^{\#}]_{2b}\bar{\beta}^j _{t}  | 0\leq 2b\leq j+1 \} \bigcup\{ \bar{\beta}^{-1}_s \} \rhd [W^{\#}_0 +g^{\#}(V^{\#}_0 ): W^{\# \infty}_0 ]$ .

 To obtain complete bases of these four subspaces, we need to chose bases of  $V^{\infty}_0 , W^{\infty}_0 ,V^{\# \infty}_0  $ and
  $ W^{\# \infty}_0$ further. To simplify notation,

\noindent If $m\in 2\mathbb{Z}^{\geq0}+1$, let $L_m =[  \frac{m-1}{2}  ]$, and $\lambda_i = 4\cos( \frac{i\pi}{m} )^2 \frac{1}{(\nu + \nu ^{-1})^2 }$ for any $1\leq i\leq L_m$.

\noindent  If  $m\in 2\mathbb{Z}^{>0}$, let $L_m =[  \frac{m-1}{2}  ]+1 $,  $\lambda_i = 4\cos( \frac{i\pi}{m} )^2 \frac{1}{(\nu + \nu ^{-1})^2 }$ for any $1\leq i\leq L_m -1 $, and  $\lambda_{L_m} =0$.

\noindent $\bullet$  By Theorem 6.4, denote the $\lambda_d -$eigenspace of $g|_{W^{\infty}_0 } h|_{V^{\infty}_0 }$ as
$V^{\infty}_{0,d} $, and the  $\lambda_d -$eigenspace of $h|_{V^{\infty}_0 } h|_{W^{\infty}_0 }$ as
$W^{\infty}_{0,d} $. Let $\{  e^{\infty} _{u,d} \}$ be a basis of $V^{\infty}_{0,d}$ for $1\leq d\leq L_{m}$.
If $m$ is odd, let $f^{\infty}_{u,d}= h( e^{\infty}_{u,d} )$, then for any $d$, $\{ f^{\infty}_{u,d} \}$ is a basis of $W^{\infty}_{0,d} $, and there hold $g (  f^{\infty}_{u,d}  )= \lambda_d e^{\infty}_{u,d}   $.

If $m$ is even, for $1\leq d\leq L_{m}-1$, let $f^{\infty}_{u,d}= h ( e^{\infty}_{u,d} )$, then  $\{ f^{\infty}_{u,d} \}$ is still a basis of $W^{\infty}_{0,d} $ and there hold $g (  f^{\infty}_{u,d}  )= \lambda_d e^{\infty}_{u,d}   $. For the special eigenvalue $0=\lambda_{L_m }$, let $\{ f^{\infty}_{u, L}  \}$ be a basis of $W^{\infty}_{0,L_m } $. There hold $ h (e^{\infty}_{u,d})=0 $ and
$g( f^{\infty}_{u,d}  )=0$.

In the dual side, similarly let $\{ \alpha^{\infty}_{u,d}\}$ be a basis of the $\lambda_d -$eigenspace of
$g^{\#}|_{W^{\# \infty}_0 } h^{\#}|_{V^{\# \infty}_0 }  $, let  $\{ \beta^{\infty}_{v,d}\}$ be a basis of the $\lambda_d -$eigenspace of
$ h^{\#}|_{V^{\# \infty}_0 } g^{\#}|_{W^{\# \infty}_0 } $ such that if $m$ is odd and $1\leq d\leq L_m$, or $m$ is even and $1\leq d\leq L_m -1$, $\beta^{\infty}_{u,d}= h^{\#}( \alpha^{\infty}_{u,d} )$, and if $m$ is even then $h^{\#}(\alpha^{\infty}_{u,L_m })=0 $,
$g^{\#}(\beta^{\infty}_{u,L_m })=0 $.




Define some constants as follows, which will be used in the proof of Theorem 6.5.

   ${}^l \Upsilon^{i,\infty}_{u,v;d}=[ e^i _u , \alpha^{\infty}_{v,d}  ]_l $,
  ${}^l \Upsilon^{\bar{i},\infty}_{u,v;d}=[ \bar{e}^i _u , \alpha^{\infty}_{v,d}  ]_l $,
   ${}^r \Upsilon^{i,\infty}_{u,v;d}=[ f^i _u , \beta^{\infty}_{v,d}  ]_r $,
   ${}^r \Upsilon^{\bar{i},\infty}_{u,v;d}=[ \bar{f}^i _u , \beta^{\infty}_{v,d}  ]_r $,

   ${}^l \Upsilon^{\infty,\infty}_{u,v;d}=[ e^{\infty}_{u,d} , \alpha^{\infty}_{v,d}  ]_l   $,
    ${}^r \Upsilon^{\infty,\infty}_{u,v;d}=[ f^{\infty}_{u,d} , \beta^{\infty}_{v,d}  ]_r  $,
     ${}^l \Upsilon^{\infty,i}_{u,v;d}=[ e^{\infty}_{u,d} , \alpha^{i}_{v,d}  ]_l   $,
     ${}^l \Upsilon^{\infty,\bar{i}}_{u,v;d}=[ e^{\infty}_{u,d} , \bar{\alpha}^{i}_{v,d}  ]_l   $,

      ${}^r \Upsilon^{\infty,i}_{u,v;d}=[ f^{\infty}_{u,d} , \beta^{i}_{v,d}  ]_r  $,
       ${}^r \Upsilon^{\infty,\bar{i}}_{u,v;d}=[ f^{\infty}_{u,d} , \bar{\beta}^{i}_{v,d}  ]_r  $.

    There hold
 ${}^l \Upsilon^{\bar{i},\infty}_{u,v;d}= {}^r \Upsilon^{i+1,\infty}_{u,v;d} $,
    ${}^r \Upsilon^{\bar{i},\infty}_{u,v;d}=\lambda_d {}^l \Upsilon^{i+1,\infty}_{u,v;d}$,
    ${}^{l}\Upsilon^{\infty,\overline{j-1}}_{u,v;d}={}^{r} \Upsilon^{\infty,j}_{u,v;d} $,
    ${}^r \Upsilon^{\infty,\bar{j}}_{u,v;d}=\lambda_d {}^l \Upsilon^{\infty,j+1}_{u,v;d}   $.


\noindent \textbf{Relevant linear equation systems. } Assume this matched pair of cross morphisms is $(m;q)$-extendable, then there
  exist $h:V\rightarrow W$ and $g:W\rightarrow V$ satisfying conditions in discussions after Definition 6.1.
  Then $gh, hg, h^{*} g^{*} , g^{*} h^{*}$ are all diagonalizable with eigenvalues in $\{ \lambda_d \}_{1\leq d\leq L_m }$. So
  $e^i _u =\sum_{d} e^i _{u,d} $ where $e^i _{u,d}$ is a $\lambda_d$-eigenvector of $g h$. Similarly there are the following decompositions.

  $\bar{e}^i _u = \sum_{d} e^i _{u,d} $,
  $f^i _u =\sum_{d} f^i _{u,d} ,  \bar{f}^i _u = \sum_{d} f^i _{u,d}  $.

  $\alpha^i _{v} =\sum_{d} \alpha^i _{v,d}$,
   $\bar{\alpha}^{i} _{v} =\sum_{d} \bar{\alpha}^{i}  _{v,d}$,
    $ \beta^{i} _{v} =\sum_{d} \beta^{i} _{v,d} $ ,
    $\bar{\beta}^i _{v} =\sum_{d} \bar{\beta}^{i} _{v,d}$.

    $e^{\infty} _u =\sum_{d} e^{\infty} _{u,d} $, $f^{\infty} _u =\sum_{d} f^{\infty} _{u,d}$.

     Since $h( e^i _v  ) = \bar{f}^{i-1}_v $,  there are
 $ h( e^i _{v,d}  )= \bar{f}^{i-1}_{v,d}  $ and $g( \bar{f}^{i-1}_{v,d}   )= \lambda_d e^i _{v,d}   $. Similarly there are
 $g( f^j _{u,d} )= \bar{e}^{j-1}_{u,d}  $, $h ( \bar{e}^{j-1}_{u,d}  )=\lambda_d f^{j}_{u,d}  $.

 \begin{prop} If $m\in 2\mathbb{Z}^{\geq0}+1$, assume above notations, then

\noindent (1) $ [ \bar{f}^{i-1}_{v,d} , \bar{\beta}^{j-1}_{u,d}    ]_r = \lambda_d [ e^i _{v,d} , \alpha^j _{u,d}   ]_l  $,
 $[\bar{e}^{i-1}_{v,d} ,\bar{\alpha}^{j-1} _{u,d} ]_l = \lambda_d [ f^i _{v,d} , \beta^j _{u,d}   ]_r   $.

 \noindent(2) $ [ f^i _{v,d} , \bar{\beta}^{j-1}_{u,d}  ]_r =[ \bar{e}^{i-1} _{v,d} , \alpha^j _{u,d}     ]_l  $,
 $[ e^i _{v,d} , \bar{\alpha}^{j-1}_{u,d}  ]_l = [ \bar{f}^{i-1}_{v,d} , \beta^j _{u,d}    ]_r $.

\noindent (3) $\sum^{L}_{d=1} (\lambda_d )^a [e^i _{v,d} ,\alpha^j _{u,d}   ]_l = \Upsilon^{i,j}_{v,u}(2a)_{00}  $,  if  $i\equiv j \pmod{2}$ and
 $0\leq 2a \leq  i+j+2 $  .

\noindent (4) $\sum^{L}_{d=1} (\lambda_d )^a [e^i _{v,d} ,\bar{\alpha}^j _{u,d}   ]_l = \Upsilon^{i, j+1}_{v,u}(2a+1)_{01}  $, if $i\equiv j \pmod{2}$ and
 $ 0\leq 2a\leq i+j+2  $ .

 \noindent(5)  $\sum^{L}_{d=1} (\lambda_d )^a [f^i _{v,d} ,\beta^j _{u,d}   ]_r = \Upsilon^{i,j}_{v,u}(2a)_{11}  $, if $i\equiv j \pmod{2}$ and
 $0\leq 2a \leq i+j+2  $ .

\noindent (6) $\sum^{L}_{d=1} (\lambda_d )^a [f^i _{v,d} ,\bar{\beta}^j _{u,d}   ]_r = \Upsilon^{i, j+1}_{v,u}(2a+1)_{10}  $, if  $i\equiv j \pmod{2}$ and
 $ 0\leq 2a\leq i+j+2   $.

 If $m\in 2\mathbb{Z}^{>0}$, if let $\bar{f}^i _{v, L_m }=0 $, $\bar{e}^i _{u, L_m }  =0 $, $\bar{\alpha}^j _{v,L_m }=0 $ and $\bar{\beta}^j _{u,L_m }=0$, then above identities (1)$\sim$(6) hold.

 \end{prop}

 \begin{pf} For the first equation of (1), $[ \bar{f}^{i-1}_{v,d} , \bar{\beta}^{j-1}_{u,d}    ]_r
 =[ h( e^i _{v,d}  ), h^{\#}( \alpha^j _{u,d}  )   ]_r = [ gh( e^i _{v,d} ), \alpha^j _{u,d}  ]_l =  \lambda_d [ e^i _{v,d} , \alpha^j _{u,d}   ]_l $. The second equation of (1) and (2) can be proved similarly.

 For (3), if $i,j$ are both odd, for any  $ 0\leq a\leq \frac{i+j+2}{2}  $, suppose $2a= 2b +2c$ where $2b\leq i+1 ,2c\leq j+1  $, then
 $[ [...gh  ]_{2b}e^i _v  , [...g^{\#} h^{\#}]_{2c} \alpha^j _u      ]_l =\Upsilon^{ij}_{v,u}(2a) _{00}  $. Notice the left side of this identity is
 $ \sum^{L}_{d=1} (\lambda_d )^a [ e^i _{v,d} ,\alpha^j _{u,d}   ]_l  $, so the identities in (3) follow.

 If $i,j$ are both even, for $0\leq a\leq \frac{i+j}{2}$, suppose
 $a= b+c$ where $ b\leq i ,c\leq j$ and both $b,c$ are even, then $ [ [...gh  ]_{b}e^i _v  , [...g^{\#} h^{\#}]_{c} \alpha^j _u      ]_l =\Upsilon^{ij}_{v,u}(a) _{00}  $, and the equation of (3) can be derived from this identity as before.

  If $i,j$ are both even and $a=\frac{i+j+2}{2}$, recall the identity $[ [...gh]_{i+1}e^i _v ,[...g^{\#}h^{\#} ]_{j+1}\alpha^j _u    ]_r =\Upsilon^{ij}_{v,u}(i+j+2)_{00}  $. Now left side of this identity equals $[ [...hg]_i ( \bar{f}^{i-1}_v ), [...h^{\#}g^{\#}]_j (\bar{\beta}^{j-1}_u)   ]_r
 = [ [...hg]_{i+j} ( \bar{f}^{i-1}_v ),  \bar{\beta }^{j-1}_u ]_r = \sum^L _{d=1} ( \lambda_d )^{\frac{i+j}{2}} [ \bar{f}^{i-1}_{v,d} , \bar{\beta }^{j-1}_{u,d}   ]_r  = \sum^L _{d=1}  ( \lambda_d )^{\frac{i+j+2}{2}} [ e^i _{v,d} , \alpha^j _{u,d}  ]_l  $, where the last equals sign is by the first identity of (1).

 (4)$\backsim $(6) can be proved similarly.\\

 \end{pf}

 We apply the same procedure to identities in (3)$\sim$(6) to obtain several linear equation systems as follows.

  $ \langle X^{i,j}_{v,u;d}   \rhd  (e^i _{v,d} , \alpha^j _{u,d}   )    \ \text{in}\ ( 3 )    \rangle  : \mathscr{E}^{i,j}_{v,u}  $, $\ \ \ $
   $ \langle  X^{i, \bar{j}}_{v,u;d}   \rhd ( e^i _{v,d} ,\bar{\alpha}^j _{u,d}   )   \ \text{in}\ ( 4 )    \rangle  :  \mathscr{E}^{i,\bar{j}}_{v,u}    $,

   $ \langle  Y^{i,j}_{v,u;d}  \rhd ( f^i _{v,d} , \beta^j _{u,d}   )   \ \text{in}\ ( 5 )   \rangle  :  \tilde{\mathscr{E}}^{i,j}_{v,u}  $, $\ \ \ $
     $ \langle  Y^{i,\bar{j}}_{v,u;d}   \rhd  (  f^i _{v,d} , \bar{\beta}^j _{u,d} )   \ \text{in}\ ( 6 )   \rangle  :  \tilde{\mathscr{E}}^{i,\bar{j}}_{v,u}   $.\\

 \noindent \textbf{Sufficient and necessary conditions for a matched couple to be $(m;q)-$extendable. } Let  $[V,V_0 ; W, W_0 ; h,g  ]$ and $[V^{\#} , V^{\#}_0 ; W^{\#} , W^{\#}_0 , h^{\#}, g^{\#}   ]$ be two cross  morphisms
  matched by a  pair of bilinear forms  $[-,- ]_l : V \times V^{\#}\rightarrow \kappa$ and $[-,-]_r : W \times W^{\#} \rightarrow \kappa $.

  \begin{thm} Assume notations $X_i ,Y_j ,S_i ,T_j$;$N_1 ,N_2 ,M_1 ,M_2$, $\Upsilon^{*,*}_{*,*}(a)_{*,*}    $,${}^{*}\Upsilon^{*,*}_{*,*;d}  $; $V^{\infty}_0 , W^{\infty}_0 $,
  $ V^{\# \infty}_0 ,  W^{\# \infty}_0 $ as before. Then there exists  a suitable trivial extension of this matched pair being
   $(m;q)-$extendable if and only if the following conditions hold.

\noindent(1) $( h|_{V^{\infty}_0 } , g|_{W^{\infty}_0 }   )$ is a specialized $(m;q)-$couple between $V^{\infty}_0$ and $W^{\infty}_0$.
     In the dual side $( h^{\#}|_{V^{\# \infty}_0 } , g^{\#}|_{W^{\# \infty}_0 }   )$ is a specialized $(m;q)-$couple between $V^{\# \infty}_0$ and $W^{\# \infty}_0$.

\noindent(2) The linear equation systems  $\mathscr{E}^{i,j}_{v,u}$ $\mathscr{E}^{i,\bar{j}}_{u,v}   $,  $ \tilde{\mathscr{E}}^{i,j}_{u,v} $, $ \tilde{\mathscr{E}}^{i,\bar{j}}_{v,u} $ are all solvable.

\noindent(3) $[\frac{N_1 +1}{2}] +1 \leq L_m$,  $[\frac{N_2 +1}{2}] +1 \leq L_m  $, $ [\frac{M_1 +1}{2}]+1 \leq L_m  $, $ [\frac{M_2 +1}{2}]+1 \leq L_m$.
   \end{thm}
   \begin{pf}  The proof proceeds as follows.

   \noindent Step 1. Two cross morphisms  $[\bar{V},\bar{V}_0 ; \bar{W}, \bar{W}_0 ; \bar{h},\bar{g}  ]$ and $[\bar{V}^{\#} , \bar{V}^{\#}_0 ; \bar{W}^{\#} , \bar{W}^{\#}_0 , \bar{h}^{\#}, \bar{g}^{\#}   ]$ are built
   from a solution of these linear equation systems. A matching of these two cross morphisms are constructed. It is called the new matched pair.

    \noindent Step 2. An isomorphism $[ \Phi ,\Theta ]$  from the new  matched pair to the old one is constructed.

    \noindent Step 3. Two morphisms $H: \bar{V} \rightarrow \bar{W}$ and $G: \bar{W} \rightarrow \bar{V}$ are constructed, which form a
$(m;q)-$extension of the new matched pair. Then the isomorphism $[\Phi, \Theta]$ converts $(H,G)$ to a $(m;q)-$ extension of the old
matched pair, thus finish the proof.



  \noindent \textbf{Step 1.}  Suppose $\{ e^i _u \}_{1\leq u\leq \varrho_i  }$, $\{ f^j _v \}_{1\leq v\leq \varrho^{'} _j    } $ , $\{ \alpha^i _u \}_{1\leq u\leq \varrho^{\#}_i   }$, $\{ \beta^j _v \}_{ 1\leq v \leq \varrho^{\# '}_j }$ are base of $X_i , Y_j , S_i , T_j$ respectively.

  By (1), denote the $\lambda_{d}-$eigenspace of $g|_{W^{\infty}_0 }\circ h|_{V^{\infty}_0 } $ as $V^{\infty}_{0,d}  $,  the $\lambda_{d}-$eigenspace of $ h|_{V^{\infty}_0 } \circ g|_{W^{\infty}_0 }$ as $W^{\infty}_{0,d}  $.
  Denote the $\lambda_{d}-$eigenspace of $g^{\#}|_{W^{\# \infty}_0 }\circ h^{\#}|_{V^{\# \infty}_0 } $ as $V^{\# \infty}_{0,d}  $,  the $\lambda_{d}-$eigenspace of $ h^{\#}|_{V^{\# \infty}_0 } \circ g^{\#}|_{W^{\# \infty}_0 }$ as $W^{\# \infty}_{0,d}  $.
   Let
   $\{  e^{\infty}_{u,d}  \}_{u=1,...,C_d }  $ be a basis of  $V^{\infty}_{0,d}  $, and  $\{  f^{\infty}_{v,d}  \}_{v=1,...,E_d }  $ be a basis of  $W^{\infty}_{0,d}$.   Let
   $\{  {\alpha}^{ \infty}_{u,d}  \}_{u=1,...,C^{\#}_d }  $ be a basis of  $V^{\# \infty}_{0,d}  $, and  $\{  {\beta}^{\infty}_{v,d}  \}_{v=1,...,E^{\#}_d }  $ be a basis of  $W^{\# \infty}_{0,d}$.

   \noindent $\bullet$ If $m$ is odd and $1\leq d\leq L_m$, or $m$ is even and $1\leq d\leq L_m -1$,  let $\bar{V}^{'}_d$ be a linear space with a specified basis
   $\{ E^i _{u,d} \}_{1\leq u\leq \rho_i } \cup \{ \bar{E}^i _{v,d} \}_{-1\leq i\leq N_1 -1; 1\leq v\leq \varrho^{'}_{i+1}   }\cup \{ E^{\infty}_{u,d}\}_{1\leq u\leq C_d} $,

   so $\bar{V}^{'} _d= \kappa< \{ E^i _{u,d} \}_{1\leq u\leq \rho_i } \cup \{ \bar{E}^i_{v,d} \}_{-1\leq i\leq N_1 -1; 1\leq v\leq \varrho^{'}_{i+1}   } \cup \{ E^{\infty}_{u,d} \}_{1\leq u\leq C_d} >$.

   Similarly we introduce linear spaces

   $ \bar{W}^{'}_d=\kappa<  \{ F^j _{u,d} \}_{1\leq u\leq \rho^{'}_j } \cup \{ \bar{F}^j _{v,d} \}_{-1\leq j\leq N_2 -1; 1\leq v\leq \varrho_{j+1}   } \cup \{ F ^{\infty}_{v,d} \}_{1\leq v\leq E_d} > $,

    $\bar{V}^{\# '} _d= \kappa< \{ A^i _{u,d} \}_{1\leq u\leq \rho^{\# }_i } \cup \{ \bar{A}^i_{v,d} \}_{-1\leq i\leq M_1 -1; 1\leq v\leq \varrho^{\# '}_{i+1}   } \cup \{ A^{\infty}_{u,d} \}_{1\leq u\leq C^{\#}_d }  >$,

     $ \bar{W}^{\# '} _d=\kappa<  \{ B^j _{u,d} \}_{1\leq u\leq \rho^{\# '}_j } \cup \{ \bar{B}^j _{v,d} \}_{-1\leq i\leq M_2 -1; 1\leq v\leq \varrho^{\# }_{i+1}   } \cup \{ B^{\infty}_{v,d} \}_{1\leq v\leq E^{\#}_d } > $.

     If $m$ is even and $d=L_m$,  define $\bar{V}^{'}_d , \bar{W}^{'}_d , \bar{V}^{\# '} _d$ and $\bar{W}^{\# '} _d $ by the identities as above but with the middle term omitted. For example
     $\bar{V}^{'} _{L_m }= \kappa< \{ E^i _{u,L_m} \}_{1\leq u\leq \rho_i } \cup \{ E^{\infty}_{u,L_m} \}_{1\leq u\leq C_{L_m}} >$.

   \noindent   Let $\bar{V}^{\infty}=\kappa<  \{ E^{\infty}_{u,d} \}  >$,  $\bar{W}^{\infty}=\kappa< \{  F^{\infty}_{v,d}  \} >$,
       $\bar{V}^{\#, \infty}=\kappa<  \{ A^{\infty}_{u,d} \}  >$,
      $\bar{W}^{\# \infty}=\kappa< \{  B^{\infty}_{v,d}  \} >$.

  \noindent Let $\bar{V}^{'}= \oplus_{d} \bar{V}^{'}_d   $,  $\bar{W}^{'}= \oplus_{d} \bar{W}^{'} _d    $,
    $\bar{V}^{\# '}= \oplus_{d} \bar{V}^{\# '}_d   $, $\bar{W}^{\# '} = \oplus_{d} \bar{W}^{\# '}_d  $.


\noindent $\bullet$ Since the linear equation systems $\mathscr{E}^{i,j}_{v,u}$, $\mathscr{E}^{i,\bar{j}}_{uv}   $,  $ \tilde{\mathscr{E}}^{i,j}_{u,v} $, $ \tilde{\mathscr{E}}^{i,\bar{j}}_{v,u} $ are solvable, suppose $X^{i,j}_{v,u;d}=\mathring{X}^{i,j}_{v,u;d} $,$X^{i,\bar{j}}_{u,v;d}=\mathring{X}^{i,\bar{j}}_{u,v;d}  $,
 $Y^{i,j}_{u,v;d}=\mathring{Y}^{i,j}_{u,v;d}$, $Y^{i,\bar{j}}_{u,v;d}=\mathring{Y}^{i,\bar{j}}_{u,v;d}$ are solutions of them.


 Let $1\leq d\leq L$. Two bilinear forms are defined by using these solutions.

 \noindent $ < >_l : \bar{V}^{'}_{d} \times \bar{V}^{\# '}_{d} \rightarrow \kappa $ be the bilinear form defined by requesting

    \noindent$< E^i _{u,d} , A^j _{v,d}   >_l =\mathring{X}^{i,j}_{u,v;d} $; $< E^i _{u,d} , \bar{A}^j _{v,d}   >_l =\mathring{X}^{i,\bar{j}}_{u,v;d} $; $< \bar{E}^i _{u,d} , A^j _{v,d}   >_l =\mathring{Y}^{i+1,\overline{j-1}}_{u,v;d} $;

    \noindent$< \bar{E}^i _{u,d} , \bar{A}^j _{v,d}   >_l =\lambda_d \mathring{Y}^{i+1,j+1}_{u,v;d} $;  $< E^i _{u,d} , A^{\infty} _{v,d}   >_l ={}^l \Upsilon^{i,\infty}_{u,v;d}$;$< \bar{E}^i _{u,d} , A^{\infty} _{v,d}   >_l ={}^l \Upsilon^{\bar{i} ,\infty}_{u,v;d} $;

 \noindent $< E^{\infty} _{u,d} , A^{i} _{v,d}   >_l ={}^{l}\Upsilon^{\infty,i}_{u,v;d}$,  $< E^{\infty} _{u,d} , \bar{A}^{i} _{v,d}   >_l ={}^{l}\Upsilon^{\infty,\bar{i}}_{u,v;d} $; $< E^{\infty}_{u,d} , A^{\infty}_{v,d}  >_l ={}^{l}\Upsilon^{\infty,\infty}_{u,v;d}. $

\noindent $ < >_r : \bar{W}^{'}_{d} \times \bar{W}^{\# '}_{d} \rightarrow \kappa    $ the bilinear form defined by requesting

     \noindent $< F^i _{u,d} , B^j _{v,d}   >_r =\mathring{Y}^{i,j}_{u,v;d} $; $< F^i _{u,d} , \bar{B}^j _{v,d}   >_r =\mathring{Y}^{i,\bar{j}}_{u,v;d} $; $< \bar{F}^i _{u,d} , B^j _{v,d}   >_r =\mathring{X}^{i+1,\overline{j-1}}_{u,v;d}$;

 \noindent $< \bar{F}^i _{u,d} , \bar{B}^j _{v,d}   >_r =\lambda_d \mathring{X}^{i+1,j+1}_{u,v;d} $;  $< F^i _{u,d} , B^{\infty} _{v,d}   >_r ={}^{r} \Upsilon^{i,\infty}_{u,v;d}  $;  $< \bar{F}^i _{u,d} , B^{\infty} _{v}   >_r ={}^{r} \Upsilon^{i ,\infty}_{u,v;d}$;

  \noindent $< F^{\infty} _{u,d} , B^{i} _{v,d}   >_r ={}^{r} \Upsilon^{\infty,i}_{u,v;d} $,  $< F^{\infty} _{u} , \bar{B}^{i} _{v,d}   >_r ={}^{r} \Upsilon^{\infty,\bar{i}}_{u,v;d} $,  $< F^{\infty}_{u,d} , B^{\infty}_{v,d}  >_r ={}^{r} \Upsilon^{\infty, \infty}_{u,v;d} .  $    \\

  \noindent $\bullet$ These bilinear forms are extended to nondegenerate ones as follows. Let $\bar{V}_{1,d} $ be some complementary space of $\ker<-, \bar{V}^{\# '}_d  >_l  $ in $\bar{V}^{'}_d$, and $\bar{V}^{\#} _{1,d}$ be some complementary space of $\ker< \bar{V}^{'}_d ,->_r  $ in $\bar{V}^{\# '}_d $. Let $l=\dim \bar{V}_{1,d} =\dim \bar{V}^{\#} _{1,d} $, and $l_1 =\dim \ker<-, \bar{V}^{\# '}_d  >_l $, $l_2 =\dim \ker< \bar{V}^{'}_d ,->_r$.
     Let $\bar{V}^{"}_d $ be a linear space of dimension $l_2$,  and let $\bar{V}^{\# "}_d $  be a linear space of dimension $l_1$.
     Let $f_{1,d} : V^{"}_d \rightarrow ( \ker< \bar{V}^{'}_d ,->_r )^{*} $ be some linear isomorphism, let $f_{2,d} : \bar{V}^{\# "}_d \rightarrow (\ker<-, \bar{V}^{\# '}_d  >_l  )^{*} $ be some linear isomorphism.

     The bilinear form $<-,->_l $ is extended to a bilinear form  $   (\bar{V}^{"}_d \oplus \bar{V}^{'}_d ) \times  (\bar{V}^{\# "}_d \oplus \bar{V}^{\# '} _d )\rightarrow \kappa $ as follows.
For any $v= v_1 + v_2 + v_3 \in \bar{V}^{"}_d \oplus \bar{V}^{'}_d $ where $v_1 \in \bar{V}^{"}_d , v_2 \in \ker<-, \bar{V}^{\# '}_d  >_l ,v_3 \in \bar{V}_{1,d}$,    and any $\alpha = \alpha_1 +\alpha_2 +\alpha_3 \in \bar{V}^{\# "}_d \oplus \bar{V}^{\# '}_d $ where $\alpha_1 \in \bar{V}^{\# "}_d , \alpha_2 \in \ker< \bar{V}^{'}_d ,->_r ,\alpha_3 \in \bar{V}^{\#} _{1,d}$,

     $< v, \alpha >_l = < v_3 , \alpha_3  >_l + f_{1,d} (v_1 )( \alpha_2  )+ f_{2,d} (\alpha_1 )( v_2 ) $.

    Evidently this bilinear form is nondegenerate,and it is still denoted as $<-,->_l $.






     Similarly we chose linear space $\bar{W}^{"}_d$ and $\bar{W}^{\# "}_d $, and extend $<-,->_r$ to a non degenerate bilinear form
     $<-,->_r : (\bar{W}^{"}_d \oplus \bar{W}^{'}_d ) \times  (\bar{W}^{\# "}_d \oplus \bar{W}^{\# '}_d )\rightarrow \kappa  $.

We can add a trivial factor of sufficiently high dimension to this matched couple such that
$\dim V \geq  \sum_{d}  \max \{ \dim \bar{V}^{"}_d +\dim \bar{V}^{'}_d  , \dim \bar{W}^{"}_d +\dim \bar{W}^{'}_d    \} $. With this condition, we can divide the number $\dim V - \sum_{d} (\dim \bar{V}^{"}_d +\dim \bar{V}^{'}_d    ) $ as a sum $\sum_d P_d   $, and  divide the number $\dim W - \sum_{d} (\dim \bar{W}^{"}_d +\dim \bar{W}^{'}_d    ) $ as a sum $\sum_d Q_d   $, such that if $m$ is odd, then
for any $d$,

$P_d + \dim \bar{V}^{"}_d +\dim \bar{V}^{'}_d  = Q_d + \dim \bar{W}^{"}_d +\dim \bar{W}^{'}_d $.

If $m$ is even, above identities are only requested to hold for $1\leq d\leq L_m -1  $.

     For any $d$, choose a linear space $\bar{V}^{'''}_d$ of dimension $P_d$ and a linear space $\bar{W}^{'''}_d$ of dimension $Q_d$,  and   let $ \bar{V}^{'''*}_d  , \bar{W}^{'''*}_d$ be their dual spaces.

     Let $\bar{V}^{'''}= \oplus_d \bar{V}^{'''}_d$, $\bar{V}^{\# '''}= \oplus_d ( \bar{V}^{ ''' *}_d ) $,$\bar{W}^{'''}= \oplus_d \bar{W}^{'''}_d$,
     $\bar{W}^{\# '''}= \oplus_d (\bar{W}^{'''}_d )^* $.

     Let $\bar{V}= \bar{V}^{'''}\oplus \bar{V}^{"} \oplus \bar{V}^{'}$,  $\bar{V}^{\#} = \bar{V}^{'''*}\oplus \bar{V}^{\# "} \oplus \bar{V}^{\# '}$,

      $\bar{W}= \bar{W}^{'''}\oplus \bar{W}^{"} \oplus \bar{W}^{'}$ and $\bar{W}^{\#} = \bar{W}^{'''*}\oplus \bar{W}^{\# "} \oplus \bar{W}^{\# '}$.

     The bilinear form $<-,->_l$ can be extended further to a non degenerate bilinear form $<-,->_l : \bar{V} \times \bar{V}^{\#} \rightarrow \kappa  $ as, for any $v= \sum_d v^{'}_d +v_2 \in \bar{V} $ where $v^{'}_d \in \bar{V}^{'''}_d ,  v_2 \in \bar{V}^{"} \oplus \bar{V}^{'} $,
     and any $\alpha =\sum_d  \alpha^{'}_d +\alpha_2 \in \bar{V}^{\#}$,
 $< v,\alpha >_l =\sum < v^{'}_d , \alpha^{'}_d  > + < v_2 ,\alpha_2  >_l  $,
where on the right hand side,   $<-,->$ is the natural pairing between a linear space and its dual space, and $<-,->_l $ is the bilinear form from $ (\bar{V}^{"} \oplus \bar{V}^{'} ) \times  (\bar{V}^{\# "} \oplus \bar{V}^{\# '} ) $  to $\kappa$.

     Similarly  $<-,->_r$ can be  extended to a non degenerate bilinear form

     $<-,->_r : \bar{W}\times \bar{W}^{\#}\rightarrow \kappa$. \\

 \noindent \textbf{ Step 2.} Define several elements as follows.

     Let $E^i _v [a] =\sum_{d} (\lambda_d)^a  E^i _{v,d}$, $\bar{E}^i _v [a] = \sum_{d} (\lambda_d )^a \bar{E}^i _{v,d} $ for $0\leq a\leq [\frac{i+1}{2}]$. Let $\bar{E}^{-1}_v = \sum_{d} \bar{E}^{-1}_{v,d}$.

     let $A^j _u [a] =\sum_{d} (\lambda_d)^a  A^j _{u,d}$, $\bar{A}^j _u [a] =\sum_{d}  (\lambda_d)^a \bar{A}^j _{u,d}$ for $0\leq a\leq [\frac{j+1}{2}]$  . Let $\bar{A}^{-1}_u =\sum_d \bar{A}^{-1}_{u,d} $.

      let  $F^i _v [a]=\sum_{d} (\lambda_d)^a F^i _{v,d}$,  $\bar{F}^i _v [a]= \sum_{d} (\lambda_d )^a \bar{F}^i _{v,d} $  for $ 0\leq a\leq [\frac{i+1}{2}]$ . Let $\bar{F}^{-1}_v =\sum_d \bar{F}^{-1}_{v,d}$.

       let $B^j _u [a] =\sum_{d} (\lambda_d )^a  B^j _{u,d}$, $\bar{B}^j _u =\sum_{d} (\lambda_d)^a  \bar{B}^j _{u,d}$  for $ 0\leq a\leq [\frac{j+1}{2}]$ . Let $\bar{B}^{-1}_u =\sum_d \bar{B}^{-1}_{u,d}$.

       Since $[\frac{N_1 +1}{2}]+1 , [\frac{N_2 +1}{2}]+1  \leq L_m$ by condition (3) of this theorem, and $\{ \lambda_d \}$ are all different,  so the set $\{ E^i _v [a], \bar{E}^i _u [a] | 0\leq i\leq N_1 ;0\leq a\leq [\frac{i+1}{2}] \}  \cup \{ \bar{E}^{-1}_u \} \cup   \{ E^{\infty}_{u,d} \} $ is linear independent.
       Similarly the sets $\{ A^i _v [a], \bar{A}^i _u [a] | 0\leq i\leq N^{\#} _1 ;0\leq a\leq [\frac{i+1}{2}] \}  \cup \{  \bar{A}^{-1}_{u}   \}  \cup \{ A^{\infty}_{u,d} \} $,  $\{ F^i _v [a], \bar{F}^i _u [a] | 0\leq i\leq N_2 ;0\leq a\leq [\frac{i+1}{2}] \}  \cup \{ \bar{F}^{-1}_{u} \} \cup \{ F^{\infty}_{u,d} \} $,   $\{ B^i _v [a], \bar{B}^i _u [a] | 0\leq i\leq N^{\#} _2 ;0\leq a\leq [\frac{i+1}{2}] \}  \cup \{ \bar{B}^{-1}_{u} \} \cup \{ B^{\infty}_{u,d} \} $ are all linear independent in the linear spaces they belong.

       Define several sets as follow.

        $\bar{S}_0 = \{ E^{i}_{v}[a], \bar{E}^{i}_{v}[a]| 0\leq 2a < i+1 \} \cup \{ E^{\infty}_{u,d} \} $,  $\bar{S}= \{ E^{i}_{v}[a], \bar{E}^{i}_{v}[a]| 0\leq 2a \leq i+1 \} \cup \{ E^{\infty}_{u,d} \} \cup \{  \bar{E}^{-1}_u \}$.

         $\bar{S}^{\#}_0 = \{ A^{i}_{v}[a], \bar{A}^{i}_{v}[a]| 0\leq a < i+1 \}   \cup \{ A^{\infty}_{u,d} \} $,  $\bar{S}^{\#} = \{ A^{i}_{v}[a], \bar{A}^{i}_{v}[a]| 0\leq a \leq i+1 \}  \cup \{ A^{\infty}_{u,d} \} \cup \{ \bar{A}^{-1}_u \}  $.

        $\bar{T}_0 = \{ F^{i}_{v}[a], \bar{F}^{i}_{v}[a]| 0\leq 2a < i+1 \} \cup \{ F^{\infty}_{u,d} \} $, $\bar{T}= \{ F^{i}_{v}[a], \bar{F}^{i}_{v}[a]| 0\leq 2a \leq i+1 \} \cup \{ F^{\infty}_{u,d} \} \cup \{ \bar{F}^{-1}_u \}   $.

         $\bar{T}^{\#}_0 =\{ B^{i}_{v}[a], \bar{B}^{i}_{v}[a]| 0\leq a < i+1 \}  \cup \{ B^{\infty}_{u,d} \} $,  $\bar{T}^{\#}= \{ B^{i}_{v}[a], \bar{B}^{i}_{v}[a]| 0\leq a \leq i+1 \}   \cup \{ B^{\infty}_{u,d} \}  \cup \{ \bar{B}^{-1}_u \}$.

      Define several subspaces as follows.

      \noindent In $\bar{V}$, $\bar{V}_0 =\kappa <\bar{S}_0   > $, $\bar{V}_1 =\kappa <  \bar{S} > $ ,
        $\bar{V}_2 = \kappa <\{ E^{i}_{v,d} , \bar{E}^{i}_{v,d}\} \cup \{ E^{\infty}_{u,d} \}  >$.So
       $\bar{V}_0 \subset \bar{V}_1 \subset \bar{V}_2$.

     \noindent   In  $  \bar{V}^{\#} $,  $\bar{V}^{\#}_0 =\kappa < \bar{S}^{\#}_0  > $,   $\bar{V}^{\#}_1 = \kappa < \bar{S}^{\#} >$,   $\bar{V}^{\#}_2 =\kappa < \{ A^{i}_{v,d} , \bar{A}^{i}_{v,d}\} \cup \{ A^{\infty}_{v,d} \}  >$. So

     \noindent $\bar{V}^{\#}_0 \subset \bar{V}^{\#}_1 \subset \bar{V}^{\#}_2$.

      \noindent  In $ \bar{W}$, $\bar{W}_0 =\kappa < \bar{T}_0 >$,  $\bar{W}_1 =\kappa < \bar{T}>$ ,
          $\bar{W}_2 =\kappa <\{ F^{i}_{v,d} , \bar{F}^{i}_{v,d}\} \cup \{ F^{\infty}_{u,d} \} >$.  So $\bar{W}_0 \subset \bar{W}_1 \subset \bar{W}_2$.

      \noindent  In $\bar{W}^{\#} $,   $\bar{W}^{\#}_0 =\kappa <\bar{T}^{\#}_0  > $,  $\bar{W}^{\#}_1 =\kappa <\bar{T}^{\#} >$,
         $\bar{W}^{\#}_1 =\kappa < \{ B^{i}_{v,d} , \bar{B}^{i}_{v,d}\} \cup \{ B^{\infty}_{u,d} \} >$. So $\bar{W}^{\#}_0 \subset \bar{W}^{\#}_1 \subset \bar{W}^{\#}_2$.

         Define several linear morphisms as follows. First if $m$ is odd,

       \noindent $\bar{h}_1 \in Hom( \bar{V}_2 , \bar{W} )$:  $\bar{h}_1 ( E^i _{v,d}  )= \bar{F}^{i-1}_{v,d}$,
        $\bar{h}_1 ( \bar{E}^i _{v,d} )= \lambda_d F^{i+1}_{v,d}$, $\bar{h}_1 ( E^{\infty}_{u,d}  )= F^{\infty}_{u,d} $.

        \noindent $\bar{g}_1 \in Hom (\bar{W}_2 , \bar{V}) $: $\bar{g}_1 ( F^i _{v,d}  )= \bar{E}^{i-1}_{v,d}$,
        $\bar{g}_1 ( \bar{F}^i _{v,d} )= \lambda_d E^{i+1}_{v,d}$, $\bar{g}_1 ( F^{\infty}_{u,d}  )= \lambda_d E^{\infty}_{u,d} $.

        \noindent $\bar{h}^{\#}_1 \in Hom( \bar{V}^{\#}_2 , \bar{W}^{\#} )$: $\bar{h}^{\#}_1 ( A^i _{v,d}  )= \bar{B}^{i-1}_{v,d}$,
        $\bar{h}^{\#}_1 ( \bar{A}^i _{v,d} )= \lambda_d B^{i+1}_{v,d}$, $\bar{h}^{\#}_1 ( A^{\infty}_{u,d}  )= B^{\infty}_{u,d} $.

        \noindent $\bar{g}^{\#}_1 \in Hom(\bar{W}^{\#}_2 , \bar{V}^{\#})$:  $\bar{g}^{\#}_1 ( B^i _{v,d}  )= \bar{A}^{i-1}_{v,d}$,
        $\bar{g}^{\#}_1 ( \bar{B}^i _{v,d} )= \lambda_d A^{i+1}_{v,d}$, $\bar{g}^{\#}_1 ( B^{\infty}_{u,d}  )= \lambda_d A^{\infty}_{u,d} $.

        When $m$ is even, those four morphisms are defined by above identities with a correction when $d=L_m$ as
        $\bar{h}_1 ( E^{\infty}_{u, L_m } )=0 $,  $\bar{g}_1 ( F^{\infty}_{u, L_m } )=0    $,  $\bar{h}^{\#}_1 ( A^{\infty}_{u, L_m } )=0    $, $\bar{}^{\#}_1 ( B^{\infty}_{u, L_m } )=0  $; and
         $\bar{h}_1 ( E^{i}_{u, L_m } )=0 $,  $\bar{g}_1 ( F^{i}_{u, L_m } )=0    $,  $\bar{h}^{\#}_1 ( A^{i}_{u, L_m } )=0    $ , $\bar{}^{\#}_1 ( B^{i}_{u, L_m } )=0  $.

        Denote  $\bar{h}_1 |_{\bar{V}_0 }$ as $\bar{h}$,  $\bar{g}_1 |_{\bar{W}_0 }$ as $\bar{g}$,$\bar{h}^{\#}_1 |_{\bar{V}^{\#}_0 }$ as $\bar{h}^{\#}$,  and  $\bar{g}^{\#}_1 |_{\bar{W}^{\#}_0 }$ as $\bar{g}^{\#}$.

        $[ \bar{V};  \bar{S}  ]$ $\&$  $[ \bar{V}^{\#};  \bar{S}^{\#}  ]$ is a couple of linear spaces with a prescribed subset. It will be
       compared with the couple   $[ V;  S  ]$ $\&$ $[ V^{\#}; S^{\#}  ]$ as follows. Where

        $S=\{[...gh]_{2a} e^{i}_{v}, [...gh]_{2a}\bar{e}^{i}_{v}| 0\leq 2a \leq i+1 \} \cup \{ \bar{e}^{-1}_u \}  \cup \{ e^{\infty}_{u,d} \}   $ and

       $S^{\#}= \{[...g^{\#}h^{\#}]_{2a} \alpha^{i}_{v}, [...g^{\#}h^{\#}]_{2a}\bar{\alpha}^{i}_{v}| 0\leq a \leq i+1 \} \cup \{ \bar{\alpha}^{-1}_u \}  \cup \{ \alpha^{\infty}_{u,d} \}  $. By the way let

       $S_0=\{[...gh]_{2a} e^{i}_{v}, [...gh]_{2a}\bar{e}^{i}_{v}| 0\leq 2a < i+1 \}   \cup \{ e^{\infty}_{u,d} \}   $ and

       $S^{\#}_0= \{[...g^{\#}h^{\#}]_{2a} \alpha^{i}_{v}, [...g^{\#}h^{\#}]_{2a}\bar{\alpha}^{i}_{v}| 0\leq a < i+1 \}  \cup \{ \alpha^{\infty}_{u,d} \}  $.

        We want to apply Lemma 6.9 to these  two couples. First there is $\dim \bar{V} =\dim V$ by construction of $V$ in step 1.

      Since $E^i _v [a]= \sum^{L_m }_{d=1} (\lambda_d )^a E^i _{v,d} $, $[\frac{N_1 +1}{2}]+1 \leq L_m$, and since $\{ E^i _{v,d}  \}_{d=1,...,L_m}$ is a linear independent subset in $\bar{V}$, so the set $\{  E^i _{v}[a]   \}_{0\leq a\leq [\frac{N_1 +1}{2}]   } $ is a linear independent subset in the subspace $\kappa<\{   E^i _{v,d}  \}_{d=1,...,L_m}>$. Similarly  $\{  \bar{E}^i _{v}[a]   \}_{0\leq a\leq [\frac{N_1 +1}{2}]   } $ is a linear independent subset in the subspace  $\kappa<\{   \bar{E}^i _{v,d}  \}_{d=1,...,L_m}>$. Now since
      $\{   E^i _{v,d}   \} \cup \{  E^{\infty}_{u,d} \} $ is a linear independent subset, so $\bar{S}$ is a linear independent subset in $\bar{V}$.

      By similar argument the set $\bar{S}^{\#}$ is a linear independent subset in $\bar{V}^{\#}$.

      Next let $\phi: \bar{S} \rightarrow
        S $ be the bijective map defined by

        $\phi( E^i _{v}[a] )=[...gh]_{2a}e^i _v , \phi(\bar{E}^i _{v}[a])=[...gh]_{2a}\bar{e}^i _v , \phi (\bar{E}^{-1}_v )= \bar{e}^{-1}_v ,
      \phi ( E^{\infty}_{u,d}  )= e^{\infty}_{u,d} $.

      and let $\psi: \bar{S}^{\#} \rightarrow
       S^{\#}$ be the bijective map defined by

       $\psi( A^i _{v}[a] )=[...g^{\#}h^{\#}]_{2a}\alpha^i _v , \psi(\bar{A}^i _{v}[a])=[...g^{\#}h^{\#}]_{2a}\bar{\alpha}^i _v , \psi (\bar{A}^{-1}_v )= \bar{\alpha}^{-1}_v , \psi ( A^{\infty}_{u,d}  )= \alpha^{\infty}_{u,d} $.

       Then by using Proposition 6.5, we can check that for any element $v \in \bar{S} $ and any element $\alpha \in  \bar{S}^{\#}$, there is
 $ (  v, \alpha  )= ( \phi(v) ,\psi(\alpha)    )  $.

       So the conditions for this pair in the following Lemma 6.9 is satisfied, and by Lemma 6.9 there exist a linear isomorphism
       $\Phi: \bar{V}\rightarrow V$, such that
       $\Phi (v) = \phi (v)  $ for any $v \in \bar{S}$, $\Phi^{*}( \alpha )= (\psi)^{-1} ( \alpha) $ for any $\alpha \in S^{\#}  $.

       Similarly consider the couple of linear spaces with a prescribed subset $[\bar{W};  \bar{T} ]$ $\&$ $[\bar{W}^{\#}; \bar{T}^{\#}]$, and compare it with another couple $[W; T ]$ $\&$ $[W^{\#}; T^{\#}]$, where

       $T= \{ [...hg]_{2a}f^{i}_{v}, [...hg]_{2a}\bar{f}^{i}_{v}| 0\leq 2a \leq i+1 \} \cup \{ \bar{f}^{-1}_u \} \cup  \{ f^{\infty}_{u,d} \} $ and

       $T^{\#}=\{ [...h^{\#} g^{\#}]_{2a} \beta^{i}_{v}, [...h^{\#} g^{\#}]_{2a}\bar{\beta}^{i}_{v}| 0\leq 2a \leq i+1 \} \cup \{  \bar{\beta}^{-1}_u \}  \cup \{ \beta^{\infty}_{u,d} \}  $. By the way let

        $T_0 = \{ [...hg]_{2a} f^{i}_{v}, [...hg]_{2a}\bar{f}^{i}_{v}| 0\leq 2a < i+1 \}  \cup  \{ f^{\infty}_{u,d} \} $ and

        $T^{\#}_0 =\{ [...h^{\#}g^{\#}]_{2a}\beta^{i}_{v},[...h^{\#}g^{\#}]_{2a} \bar{\beta}^{i}_{v}| 0\leq 2a < i+1 \}  \cup \{ \beta^{\infty}_{u,d} \}  $.

        The same argument produces a linear isomorphism $\Theta: \bar{W} \rightarrow W $ such that

       $\Theta(  F^i _v [a]  )= [...hg]_{2a} f^i _v ,  \Theta(\bar{F}^i _v [a])= [...hg]_{2a} \bar{f}^i _v   ,
        \Theta ( \bar{F}^{-1}_u  )= \bar{f}^{-1}_u  , \Theta (  F^{\infty} _{v,d} )= f^{\infty}_{v,d} $, and

       $\Theta^{*}([...hg    ]_{2a} \beta^i _v    )= B^i _v [a] , \Theta^{*}([...hg    ]_{2a} \bar{\beta}^i _v    )= \bar{B}^i _v [a] ,
        \Theta^{*} ( \bar{\beta}^{-1} _u    )= \bar{B}^{-1}_u , \Theta^{*}( \beta^{\infty}_{v,d}  )= B^{\infty}_{v,d} $.\\

 Then  we show $[ \Phi ,\Theta ]$ constitute an isomorphism between these two matched pairs. It is enough to show

        (1) $\Phi(\bar{V}_0)= V_0 , \Phi^{*}( V^{\#}_0  )=\bar{V}^{\#}_0 , \Theta(\bar{W}_0 ) = W_0 ,\Theta^{*}( W^{\#}_0 )=\bar{W}^{\#}_0 .$

        (2)$\Theta \circ \bar{h}(v )= h\circ \Phi (v) $ for any $v\in \bar{S}_0 $;
         $\Phi \circ \bar{g}(w)=g\circ \Theta(w)$ for any $w\in \bar{W}$;

         $\Theta^{*} \circ h^{\#}(\alpha )= \bar{h}^{\#}\circ \Phi^{*} (\alpha) $ for any $\alpha\in S^{\#} _0 $;
         $\Phi^{*} \circ g^{\#}(\beta)=\bar{g}^{\#} \circ \Theta^{*}(\beta)$ for any $\beta\in T^{\#}_0 $.\\

 \noindent \textbf{Step 3.} Consider the spaces $\bar{V}_d  ,  \bar{V}^{'}_d ;\bar{W}_d  ,  \bar{W}^{'}_d  $,  ($1\leq d\leq L_m $ when $m$ is odd and
 $1\leq d<L_m $ when $m$ is even)  and the morphisms

       $\bar{h}_1 :\bar{V}^{'}_d \rightarrow \bar{W}^{'} _d \subset \bar{W}_d $, $\bar{g}_1 :\bar{W}^{'}_d \rightarrow \bar{V}^{'} _d \subset \bar{W}_d   $,

       $\bar{h}^{\#}_1 :\bar{V}^{\# '}_d \rightarrow \bar{W}^{\# '} _d \subset \bar{W}^{\#}_d $, $\bar{g}^{\#}_1 :\bar{W}^{\# '}_d \rightarrow \bar{V}^{\# '} _d \subset \bar{V}^{\#}_d   $.

       By construction  there is $\dim \bar{V}_d = \dim \bar{W}_d = \dim \bar{V}^{\#}_d = \dim \bar{W}^{\#}_d$.  It is straightforward to check that

       for any $v\in \bar{V}^{'}_d$ and $\beta \in \bar{W}^{\# '}_d    $,$< \bar{h}_1 (v) ,\beta   >_r =< v , \bar{g}^{\#}_1 (\beta)    >_l $,

       and for any $w\in \bar{W}^{'}_d ,  \alpha\in \bar{V}^{\# '}_d $, $< \bar{g}_1 (w) ,\alpha  >_l =< w, \bar{h}^{\#}_1 (\alpha)  >_r $.

       So by using Lemma 6.10, there exist an isomorphism $\bar{H}_d : \bar{V}_d \rightarrow \bar{W}_d $ and an isomorphism $\bar{G}_d : \bar{W}_d \rightarrow \bar{V}_d $ such that

       (1)$\bar{H}_d |_{\bar{V}^{'}_d  }= \bar{h}_1 ; \bar{G}_d |_{\bar{W}^{'}_d }=\bar{g}_1 ;(\bar{H}_d )^* |_{\bar{W}^{\# '}_d  }= \bar{g}^{\#}_1 ; (\bar{G}_d)^* |_{\bar{V}^{\# '}_d }=\bar{h}^{\# }_1   $.

       (2)$\bar{H}_d \circ \bar{G}_d = \lambda_d id_{\bar{W}_d }$ and $ \bar{G}_d \circ \bar{H}_d =\lambda_d id_{\bar{V}_d }  $.

 When $m$ is even and $d=L_m$, just let $\bar{H}_{L_m } : \bar{V}_d \rightarrow \bar{W}_d $ and  $\bar{G}_{L_m } : \bar{W}_{L_m } \rightarrow \bar{V}_{L_m }$ be the zero maps. So $\bar{H}_d ,\bar{G}_d $ ($1\leq d\leq L_m $) are defined whenever $m$ is odd or even.

      Let $\bar{H}=\oplus_d \bar{H}_d : \bar{V}=\oplus_d \bar{V}_d  \rightarrow \bar{W}=\oplus_d \bar{W}_d $, and
      $ \bar{G}=\oplus_d \bar{G}_d : \bar{W}=\oplus_d \bar{W}_d  \rightarrow \bar{V}=\oplus_d \bar{V}_d $.

      By construction of $\bar{H}$ and $\bar{G}$, $(\bar{H} ,\bar{G})$ is a specialized $(m;q)-$couple between $\bar{V}$ and $\bar{W}$. \\

     As the final step, let $H= \Theta \bar{H} \Phi^{-1}: V\rightarrow W$ and  $G=\Phi \bar{G} \Theta^{-1}:W\rightarrow V $,

      $H^{\#}=G^* = (\Theta^*)^{-1}\bar{G}^* \Phi^* : V^{\#}\rightarrow W^{\#} $ and
      $ G^{\#}=H^* = (\Phi^*)^{-1}\bar{H}^* \Theta^* : W^{\#}\rightarrow V^{\#}  $.

      (1) Since $( \bar{H} ,\bar{G} )$ is a specialized $(m;q)-$couple between $\bar{V}$ and $\bar{W}$, so $( H ,G )$ is a specialized $(m;q)-$couple between $V$ and $W$, and $( H^{\#} ,G^{\#} )$ is a specialized $(m;q)-$couple between $V^{\#}$ and $W^{\#}$.

      (2) $H,G,H^{\#},G^{\#}$ are extensions of $h,g,h^{\#}, g^{\#}$ respectively. This is proved case by case as follows.

      For $v\in V_0$, $H(v)=\Theta \bar{H} \Phi^{-1} (v) = \Theta \bar{h} \Phi^{-1}(v) =\Theta \Theta^{-1} h(v)= h(v) $.

      For $w\in W_0$, $G(v)=\Phi \bar{G} \Theta^{-1} (v) = \Phi \bar{g} \Theta^{-1}(w) =\Phi \Phi^{-1} g(v)= g(v) $.

      For $\alpha\in V^{\#}_0$, $ H^{\#}(\alpha)=(\Theta^* )^{-1} \bar{G}^{*} \Phi^{*}(\alpha) =
       (\Theta^* )^{-1} \bar{h}^{\#} \Phi^{*}(\alpha)=  (\Theta^* )^{-1} \Theta^* h^{\#}(\alpha)= h^{\#}(\alpha)$.

       For $\beta\in W^{\#}_0$, $ G^{\#}(\beta)=(\Phi^* )^{-1} \bar{H}^{*} \Theta^{*}(\beta) =
       (\Phi^* )^{-1} \bar{g}^{\#} \Theta^{*}(\beta)=  (\Phi^* )^{-1} \Phi^* g^{\#}(\beta)= h^{\#}(\beta)$.\\

   \end{pf}

\begin{lem}
Suppose $V,W$ are two linear space of the same dimension. Suppose $\{ v_i \} _{1\leq i\leq m} \subset V ,
 \{ w_i \} _{1\leq i\leq m } \subset W  $ are two linear independent subsets. Suppose
 $\{  \alpha_j  \} _{1\leq j\leq l   } \subset V^* , \{ \beta_j \}_{1\leq j\leq l }\subset W^* $ are two linear independent subsets.
 Suppose $\alpha_j ( v_i  )  =\beta_j ( w_i  ) $ for any $1\leq i\leq m , 1\leq j\leq l$.
 Then there exists a linear isomorphism $F: V\rightarrow W$ such that $F(v_i )= w_i $  for any $1\leq i\leq m$ and
 $F^* ( \beta_j  )= \alpha_j $ for any $  1\leq j\leq l $.
\end{lem}

\begin{pf} Choose some subspace $V_0 \subset \kappa< \alpha_j >^{\perp} $ such that

$\kappa<\alpha_j >^{\perp}=( \kappa<\alpha_j>^{\perp}\cap \kappa<v_i> ) \oplus V_0    $. Then choose some subspace $V_1 \subset V$ such that
$V= ( \kappa<\alpha_j>^{\perp} + \kappa<v_i>   )\oplus V_1 $. Let $\phi: \kappa<\alpha_j>\rightarrow \kappa<\beta_j>$ be the isomorphism that $\phi(\alpha_j)=\beta_j$ for any $1\leq j\leq l$. Let $\psi: V\rightarrow \kappa<\alpha_j>^*$ be the natural morphism such that
$\psi(v)(\alpha)=\alpha(v)  $ for any $v\in V, \alpha\in \kappa<\alpha_j>$. Then $\psi|_{V_1}: V_1 \rightarrow \kappa<\alpha_j>^*$ is injective and $ \kappa<\alpha_j>^* = \psi(V_1)\oplus \psi(\kappa<v_i>) $.

Let $\bar{\psi}: W\rightarrow \kappa<\beta_j>^*$ be the natural morphism as $\psi$. Then since $\alpha_j ( v_i  )  =\beta_j ( w_i  )$, there is $\kappa<\beta_j>^*=  ( \phi^*  )^{-1}( \kappa<\alpha_j >^*  ) = ( \phi^*  )^{-1}\circ \psi (V_1)  \oplus (\phi^*)^{-1}\circ \psi (\kappa<v_i>)   $, and $(\phi^*)^{-1}\circ \psi (\kappa<v_i>) =\bar{\psi}(\kappa<w_i>)  $.  So  we have
$\kappa<\beta_j>^* = (\phi^*)^{-1}\circ\psi(V_1)\oplus \bar{\psi}(\kappa<w_i>) $. Now since $\bar{\psi}$ is surjective, there is a subspace $W_1 \subset W$ such that $\bar{\psi}|_{W_1}: W_1 \rightarrow \kappa<\beta_j>^*$ is a injective morphism whose image is $(\phi^*)^{-1}\circ\psi(V_1)$.

Let $W_0 \in \kappa<\beta_j>^{\perp}$ be a subspace such that
$\kappa<\beta_j>^{\perp}=(\kappa<\beta_j>^{\perp}\cap \kappa<w_i>)\oplus W_0$.  With these notations we have

$V= V_1 \oplus ( (\kappa<\alpha_j>^{\perp}+\kappa<v_i>)  )= V_1 \oplus V_0 \oplus \kappa<v_i> $,

$W= W_1 \oplus ( (\kappa<\beta_j>^{\perp}+\kappa<w_i>)  )= W_1 \oplus W_0 \oplus \kappa<w_i>$.

Let $\{ v^{'}_k \}$ be a basis of $V_1$ and $\{ v^{"}_d \}$ be a basis of $V_0$. Then $\{ v_i \} \cup \{ v^{'}_k \} \cup \{ v^{"}_d \} $ is a basis of $V$.  Let $\{ w^{"}_d \} $ be a basis of $W_0$. For any $k$, let $w^{'}_k = (\bar{\psi}|_{W_1})^{-1}\circ (\phi^*)^{-1}\circ \psi(v^{'}_k )$, then $\{ w^{'}_k \} $ is a basis of $W_1$, and $ \{ w_i \} \cup \{ w^{'}_k \} \cup \{ w^{"}_d \} $ is a basis of $W$.

Now let $F: V\rightarrow W$ be the isomorphism such that $F(v_i)=w_i , F(v^{'}_k )=w^{'}_k  $ and $ F(v^{"}_d )= w^{"}_d $.  The last step of the proof is to check $F^* (\beta_j )= \alpha_j $. It is done by direct computation as follows. For any $1\leq i\leq m$ and $1\leq j\leq l$,

$F^* (\beta_j )(v_i )=  \beta_j ( F(v_i)  )=  \beta_j (w_i )= \alpha_j (v_i )$.

$F^* (\beta_j )(v^{'}_k  )= \beta_j (  F( v^{'}_k ) )= \beta_j ( w^{'}_k )= \beta_j (   (\bar{\psi}|_{W_1})^{-1}\circ (\phi^*)^{-1}\circ \psi(v^{'}_k )    ) = (\phi^*)^{-1}\circ \psi (v^{'}_k)(\beta_j )$

$= \psi(v^{'}_k )( \phi^{-1}(\beta_j )   )= \psi(v^{'}_k )(\alpha_j) =\alpha_j (v^{'}_k)  $.

$F^* (\beta_j)(v^{"}_d)= \beta_j (  F(v^{"}_k) )=\beta_j (w^{"}_k)=0=\alpha_j(v^{"}_k)$.\\

\end{pf}

\begin{lem}
Suppose $V$ and $W$ are linear spaces with the same dimension. Denote the dual spaces of $V$ and $W$ as $V^{*}$ and $W^*$ respectively.
Suppose there are subspaces $V_0 \subset V  ,  W_0 \subset W , V^{\#}_0 \subset V^{*}, W^{\#}_0 \subset W^{*}$ and morphisms
$h_0 : V_0 \rightarrow W_0 , g_0 : W_0 \rightarrow V_0 , h^{\#}_0 : V^{\#}_0 \rightarrow W^{\#}_0 , g^{\#}: W^{\#}_0 \rightarrow V^{\#}_0 $ such that $g_0 \circ h_0 = \lambda id_{V_0 } ; h_0 \circ g_0 = \lambda id_{W_0 }; g^{\#}_0 \circ h^{\#}_0 =\lambda id_{V^{\#}_0 }, h^{\#}_0 \circ g^{\#}_0 =\lambda id_{W^{\#}_0 }$, and $ ( h_0 (v) , \beta   )= ( v, g^{\#}_0 (\beta)   )  $ for any $v\in V_0 , \beta \in W^{\#}_0 $, $ ( g_0 (w) , \alpha   )= ( w, h^{\#}_0 (\beta)   )  $ for any $w\in W_0 , \alpha \in V^{\#}_0 $, then there exist an isomorphism $h: V\rightarrow W$ and an isomorphism $g: W\rightarrow V  $ such that $h|_{V_0}=h_0 ,
   h^{*}|_{W^{\#}_0}=g^{\#}_0 , g|_{W_0 }= g_0 , g^{*}_{V^{\#}_0 }=h^{\#}_0 $ and $g\circ h = \lambda id_{V} , h\circ g =\lambda id_{W}$.

\end{lem}

\begin{pf} We prove it by induction on $\dim V-\dim V_0$. If $\dim V -\dim V_0 =0$, then $h=h_0 ,g=g_0$ meet the requirements of $h,g$. Suppose
the lemma is proved for cases when $\dim V-\dim V_0 \leq N$. Suppose the present case satisfies $\dim V - \dim V_0 = N+1$.
Choose any $e \in V\setminus V_0 $. Let $\eta_e \in ( W^{\#}_0 )^* $ be the element determined by $\eta_e ( \beta ) = (e,g^{\#} _0 (\beta)   )  $.  For $w\in W$, let $(w,-)\in ( W^{\#}_0 )^* $ be the element determined by $ (w,-)(\beta)= (w,\beta)  $ for any $\beta \in W^{\#}_0$.

Case (1). $\eta_e \notin \{ (w,-) \}_{w\in W_0 })$.   In these cases choose any $f\in W$ such that $(f, \beta)= \eta_e (\beta)$ for any $\beta \in W^{\#}_0 $. Since $\eta_e \notin \{ (w,-) \}_{w\in W_0 })$ then $f\notin W_0 $. Let $V_1 =V_0 \oplus \kappa<e> ,
 W_1 = W_0 \oplus \kappa<f> $. Extend $h_0 $ to $h_1 : V_1 \rightarrow W_1$ by requesting $h_1 (e)= f$, and extend $g_0 $ to $g_1 : W_1 \rightarrow V_1$ by requesting $g_1 (f)= \lambda e$. Then the combination $[ V ,V_1 ; W, W_1 ; h_1 , g_1 ; V^*  ,V^{\#}_0 ; W^{*}, W^{\#}_0 ; h^{\#}_0 , g^{\#}_0   ]$ satisfies all conditions in this lemma and $\dim V -\dim V_1 =N $. So by induction such $h:V\rightarrow W$ exist.

 Case (2).  $\eta_e \in \{ (w,-) \}_{w\in W_0 })$ and $( W^{\#}_0   )^{\perp} \nsubseteq  W_0 $. In these cases choose $f_0 \in W_0$ such that $(f_0 ,-  )= \eta_e  $, and choose any $f^{'}\in ( W^{\#}_0   )^{\perp}\setminus W_0 $, then let $f= f_0 + f^{'}$. Let $V_1 =V_0 \oplus \kappa<e> , W_1 = W_0 \oplus \kappa<f> $. Extend $h_0 $ to $h_1 : V_1 \rightarrow W_1$ by requesting $h_1 (e)= f$, and extend $g_0 $ to $g_1 : W_1 \rightarrow V_1$ by requesting $g_1 (f)= \lambda e$, and the rest of the proof is the same as Case (1).  The key is we need to find a $f$ such that  $(f_0 ,-  )= \eta_e  $ and $f\notin W_0$.

 Case (3).  $\eta_e \in \{ (w,-) \}_{w\in W_0 })$ and $( W^{\#}_0   )^{\perp} \subset  W_0 $.  We prove these cases are impossible.
 First we prove if these conditions are satisfied then $(e,- )\in \{ (v,-) \}_{v\in V_0 }  $ and $( V^{\#}_0  )^{\perp} \subset V_0  $.

 By above conditions there exists $w_0 \in W_0$ such that $ \eta_e = (w_0 ,- )$. So $ (e, g^{\#}_0 (\beta)   )=(w_0 ,\beta) $
for any $\beta \in W^{\#}_0$. Now for any $\alpha\in V^{\#}_0$,

$(e,\alpha)=\frac{1}{\lambda}(e,g^{\#}_0 h^{\#}_0 (\alpha) ) = \frac{1}{\lambda}(w_0 ,h^{\#}_0 (\alpha))=
 \frac{1}{\lambda}( g_0 (w_0 ) , \alpha ) $.

 So $v=\frac{1}{\lambda} g_0 (w_0 ) \in V_0 $ is an element such that $(e,- )= (v,- )  $ as an element in  $( V^{\#}_0 )^* $.

 Next let $w\in W_0 \cap ( W^{\#}_0 )^{\perp}$, so  $(w,\beta  )=0$  for any $\beta \in W^{\#}_0 $. So for any $\alpha \in V^{\#}_0 $ there is $( g_0 (w) ,\alpha   )= (w, h^{\#}_0 (\alpha ))=0 $. So $g_0 (w) \in V_0 \cap ( V^{\#}_0 )^{\perp}$.   Similarly for any  $v \in V_0 \cap ( V^{\#}_0 )^{\perp}$ we can prove $h_0 (v) \in \in W_0 \cap ( W^{\#}_0 )^{\perp} $. So by $h_0 \circ g_0 = \lambda id_{W_0 }$ and
 $g_0 \circ h_0 = \lambda id_{V_0 }$,  there is $\dim  W_0 \cap ( W^{\#}_0 )^{\perp} = \dim  V_0 \cap ( V^{\#}_0 )^{\perp} $.  Now since
 $( W^{\#}_0   )^{\perp} \subset  W_0 $ then $\dim  V_0 \cap ( V^{\#}_0 )^{\perp}=\dim  W_0 \cap ( W^{\#}_0 )^{\perp} =\dim( W^{\#}_0 )^{\perp} = ( V^{\#}_0 )^{\perp} $. So $( V^{\#}_0  )^{\perp} \subset V_0  $.

   Finally, $(e,- )\in \{ (v,-) \}_{v\in V_0 }  $ and $( V^{\#}_0  )^{\perp} \subset V_0  $ contradicts with the condition $e\notin V_0 $. Because by $(e,- )\in \{ (v,-) \}_{v\in V_0 }  $ there exists $v_0 \in V_0$ such that $(e,- )=(v_0 .-)$ , so  $e-v_0 \in ( V^{\#}_0  )^{\perp} $. Then by $( V^{\#}_0  )^{\perp} \subset V_0  $, $e-v_0 \in V_0 $, so $e\in V_0 $.

\end{pf}

\begin{thm}
Suppose  $\Theta = [ ( \bar{V}_{i,j} , \bar{A}^{i,j}_{k,l} ) ,  \cup_{i\in I} \{ \bar{U}_i ,\bar{W}_i \}    ] $ is a decorated $\mathbb{Q}^2 _I -$representation satisfying conditions (1)$\sim$ (5) in Theorem 5.2. Suppose $\Gamma$ is a Coxeter system such that
$m_{i,j}\geq 5$ for any $i\neq j$. Then $\Theta$ is $(\Gamma ,v)$-realizable for any $v\in \kappa$.
\end{thm}

\begin{pf} The "space realization" construction of $\Theta$ produce some spaces $V_i , V^{\#}_i$ ($i\in I$) and nondegenerate pairing $(-,-)$ between them. Notice $\dim V_i $ can be made arbitrary large by adding trivial factors to $V_i , V^{\#}_i $. Every $V_i$ possesses
some decompositions $V_i = V_{i,j}\oplus V_{i/j}$.  The proof is divided into two steps.

(1) There exist  $\{  \phi^i _j \}_{i,j \in I}$ and $\{  \phi^{\# i} _j \}_{i,j \in I}$  such that  for any $j\neq i$,  $\psi^j _i ( V^0 _{j/i} ) \cap V^0 _{i/j}= \{ 0 \} $
and  $\psi^{\# j} _i ( V^{\# 0} _{j/i} ) \cap V^{\# 0} _{i/j}= \{ 0 \} $, where $\psi^j _i = \phi ^j _i |_{ V^0 _{j/i} }$ and
$\psi^{\# j} _i = \phi ^{\# j} _i |_{ V^{\# 0} _{j/i} }$.

We show for any $i\in I$ there exist $\Delta_i : \bar{V}^0 \rightarrow V_{i/}$, such that for any $j\neq i$, the morphism
$\phi^j _i = \bar{\phi}^j _i + \Delta_i |_{V^0 _j} $ satisfy the condition in (1). Recall the restrictions on $\Delta_i$ are
(a) For any $v\in V^0 _i \subset \bar{V}^0 $, $\Delta_i (v) = v- \bar{\phi}^i _i (v) $; (b) if $m_{i,j}$ is odd, then
$\ker ( \Delta_i   ) \cap \ker ( \bar{\phi}^j _i  )=0 $.

Recall notations in section 5.3.  For the condition (b), there exist a lifting $ \bar{W}_i $ of $W_i $ such that $\bar{\phi}^i _i (v)\neq v$ for any $v\in V^0 _i$. The argument is as follows.

First let $\eta^{'}_i : W_i \rightarrow V_i $ be any injective morphism such that $p^i _i \circ \eta^{'}_i = id_{W_i }$. Then suppose $\dim V_i$ is large enough,so there is an injective morphism $\tau_i : W_i \rightarrow V_{i/}$ such that
$Im(\tau_i) \cap ( V^0 _i + Im ( \eta^{'}_i )   )= \{ 0 \}$. Let $\eta_i = \eta^{'}_i + \tau_i $, then $p^i _i \circ \eta_i =id_{W_i} $. Denote $Im (\eta_i)$ as $\bar{W}_i$. Now for any $v\in V^0 _i$,
$\bar{\phi}^i _i ( v )= \eta_i ( \bar{\Phi}^i _i (v)    )= \eta^{'}_i (\bar{\Phi}^i _i (v)   ) + \tau_i (\bar{\Phi}^i _i (v)    ) $. If this element is $v$ then it is in $V^0 _i$, contradicts with the fact that $Im(\tau_i) \cap ( V^0 _i + Im ( \eta^{'}_i )   )= \{ 0 \}$.

With this lifting $\bar{W}_i$, $v-\bar{\phi}^i _i (v)\neq 0 $ for any $v\in V^0 _i$. let $X_i = \{ v-\bar{\phi}^i _i (v) \}_{ v\in V^0 _i }   $.
Suppose $\bar{V}^0 = V^0 _i \oplus U_i $ for some subspace $U_i$.  We can choose space realizations of $\Theta$ for which $\dim V_{i/}$  is large enough such that there exist a subspace $U^{'}_i \subset V_{i/}$ satisfying  (a) $\dim U^{'}_i =\dim U_i $,  (b) $U^{'}_i \cap ( X_i + \bar{W}_i + V^0 _i ) =\{ 0 \}$.  Choose an linear isomorphism \
$\phi: U_j \rightarrow U^{'}_j$.

Let $\Delta_i: \bar{V}^0 \rightarrow V_{i/}$ be defined as ,  for any  $v+w \in V^0 _i \oplus U_i = \bar{V}^0 $ where $v\in V^0 _i$ and $w\in U_i$, $\Delta_i (v+ w) =  ( v-\bar{\phi}^i _i (v)   )+ \phi (w) $.  Let $\phi^j _i = \bar{\phi}^j _i + \Delta_i |_{V^0 _j }$ for any
$j$.

We show for any $j\neq i$, $\phi^j _i ( V^0 _{j/i}  ) \cap V^0 _i = \{ 0 \} $. Suppose $0\neq v\in V^0 _{j/i}$ and $\phi^j _i (v)\in V^0 _i$.
suppose $v=v_1 +v_2$ where $v_1 \in V^0 _i $ and $v_2 \in U_i $. Since $V^0 _i \cap V^0 _j =V_{i,j}$ in $\bar{V}^0 $ and $v\notin V_{j,i}$, there is $v_2 \neq 0$. Then there is an identity $ \phi^j _i (v) =\bar{\phi}^j _i (v) + \Delta_i (v )=\bar{\phi}^j _i (v) + \Delta_i (v_1 )  + \Delta_i (v_2)  $, so $\Delta_i (v_2) = \phi (v_2 )= \phi^j _i (v)-\bar{\phi}^j _i (v)-\Delta_i (v_1 )  $ is a nonzero element in
$U^{'}_i \cap ( X_i + \bar{W}_i + V^0 _i )$, which contradicts with the choice of $U^{'}_i $. Notice $\bar{\phi}^j _i (v)\in \bar{W}_i ,\Delta_i (v_1 )\in X_i  $.

(2) If a set of $\psi ^j _i : V^0 _{j/i} \rightarrow V_{i/j}$ satisfy  the properties in step (1), then the matched pair of crossing morphisms $ [ V_{i/j} , V^{0}_{i/j} ; V_{j/i}, V^0 _{j/i} ; \psi^i _j , \psi^j _i     ]   $ and $[ V^{\#} _{i/j} , V^{\# 0}_{i/j} ; V^{\#} _{j/i}, V^{\# 0} _{j/i} ; \psi^{\# i} _j , \psi^{\# j} _i     ] $ is $(m_{i,j},v)$-extendable.

Since the conditions in (1) is true,  for the cross morphism
$ [ V_{i/j} , V^{0}_{i/j} ; V_{j/i}, V^0 _{j/i} ; \psi^i _j , \psi^j _i     ]   $,there are $N_1 , N_2 =0$ and $X_0 = V^{0}_{i/j}  $, $Y_0 = V^0 _{j/i}   $. Suppose $\{ e^0 _u \} $ is a basis of $X_0$, $\{ f^0 _v \}$ is a basis of $Y_0$. So if let $\bar{f}^{-1}_u = \psi^i _j ( e^0 _u )$ and $ \bar{e}^{-1}_v = \psi^j _i (f^0 _v)  $, then $\{ \bar{f}^{-1}_u \}$ is a basis of $\psi^i _j ( V^{0}_{i/j}  )$ and $\{ \bar{e}^{-1}_v \}$ is a basis of $\psi^j _i ( V^0 _{j/i}  )$.

Similarly $M_1 , M_2 =0$ for the cross morphism  $ [ V^{\#}_{i/j} , V^{\# 0}_{i/j} ; V^{\#}_{j/i}, V^{\# 0} _{j/i} ; \psi^{\# i} _j , \psi^{\# j} _i     ]   $, and $S_0 = V^{\# 0}_{i/j}  $, $T_0 = V^{\# 0} _{j/i}   $. Suppose $\{ \alpha^0 _u \} $ is a basis of $S_0 $, $\{ \beta^{\# 0} _v \}$ is a basis of $T_0$. So if let $\bar{\beta}^{-1}_u = \psi^{\# i} _j ( \alpha^0 _u )$ and $ \bar{\alpha}^{-1}_v = \psi^{\# j} _i (\beta^0 _v)  $, then $\{ \bar{\beta}^{-1}_u \}$ is a basis of $\psi^{\# i} _j ( V^{\# 0}_{i/j}  )$ and $\{ \bar{\alpha}^{-1}_v \}$ is a basis of $\psi^{\# j} _i ( V^{\# 0} _{j/i}  )$.

The related classes of equations in this case is

\noindent(1) $\sum_d  \lambda_d  X^{00}_{u,d;v,d} = ( \bar{f}^{-1}_u ,\bar{\beta}^{-1}_v  )  $ and $\sum_d   X^{00}_{u,d;v,d} = ( e^0 _u ,\alpha^0 _v )  $.

\noindent (2) $\sum_d  \lambda_d  Y^{00}_{u,d;v,d} = ( \bar{e}^{-1}_u ,\bar{\alpha}^{-1}_v  )  $ and $\sum_d   Y^{00}_{u,d;v,d} = ( f^0 _u ,\beta^0 _v )  $.

\noindent(3) $\sum_d X^{0,\bar{-1}}_{u,d;v,d} =  (e^0 _u , \bar{\alpha}^{-1}_v ) $.
(4)$\sum_d X^{\bar{-1} 0}_{v,d;u,d}= ( \bar{e}^{-1}_v ,\alpha^0 _u ) $.

\noindent(5) $\sum_d X^{\bar{-1} \bar{-1} }_{v,d;u.d} = ( \bar{e}^{-1}_v , \bar{\alpha}^{-1}_u  )   $.
(6)$\sum_d Y^{\bar{-1} \bar{-1} }_{v,d;u.d} = ( \bar{f}^{-1}_v , \bar{\beta}^{-1}_u  )   $.

If $m_{i,j}\geq 5$, then $L_{m_{i,j}}\geq 2$, so above six types of equation systems  are  all solvable. By Theorem 5.5 the matched pair of cross morphisms is
$(m_{i,j} ,v)$-extendable after adding a trivial factor with large enough dimension to it. So the theorem is proved by using Theorem 6.1.

\end{pf}

\begin{rem}
In above proof, a way to construct intersecting type $H-$representations for the mentioned Coxeter type $\Gamma$ and generic $q$ is developed. Those constructed $H-$representations satisfy the conditions: $A^i _j ( V^0 _{i/j}  ) \cap V^0 _{j/i} =\{ 0\}   $ for any $i\neq j \in I$.
\end{rem}

\hspace{-0.70cm} {\sc School of Mathematics}

\nd{\sc Hefei university of technology }

\nd {\sc Hefei 230009 China}

\nd {\sc E-mail addresses}: {\sc Zhi Chen} ({\tt zzzchen@ustc.edu.cn}).


\begin{thebibliography}{99}



\bibitem[Ar]{Ar}  {\sc S.Ariki},On the decomposition numbers of the Hecke algebra of G(m,1,n) , {\it J. Math. Kyoto Univ. }{\bf 36}(1996  ),789-808.


\bibitem[Ch]{Ch}  {\sc I.Cherednik},Macdonald's evaluation conjectures and difference Fourier transform, {\it Invent. Math.}{\bf 122 }( 1995 ),191-216 .

\bibitem[EW]{EW}  {\sc B,Elias, G.Williamson}, The Hodge theory of Soergel bimodules  , {\it Ann. Math.(2) }{\bf180(3) }(2014),1089-1136.


\bibitem[GP]{GP}  {\sc M.Geck, G.Pfeiffer},Characters of finite Coxeter groups and Iwahori-Hecke algebras , {\it London Math. Soc. Mono., New Series }{\bf21 }(2000  ).

\bibitem[GJ]{GJ}  {\sc M.Geck, N.Jacon},Representations of Hecke algebras at roots of unity, {\it Algebra and Applications }{\bf 15}(2011  )



\bibitem[Hu]{Hu}  {\sc J.E.Humphreys}, Reflection groups and Coxeter groups , {\it Camb. Stud. in Adv. Math. }{\bf 29}( 1990 ) .


\bibitem[IM]{IM}  {\sc N.Iwahori, H.Matsumoto }, On some Bruhat decomposition and the strusture of the Hecke rings of the p-adic Chevalley groups  , {\it Inst. Hautes.\'Etudes Sci. Publ. Math. }{\bf 25 }( 1965 ),  5-48.

\bibitem[KL1]{KL1}  {\sc D.Kazhdan, G.Lusztig },Proof of the Deligne-Langlands conjecture for affine Hecke algebras , {\it Invent. Math.  }{\bf 87 }( 1987 ),   153-215.

\bibitem[KL2]{KL2}  {\sc D.A.Kazhdan, G.Lusztig},Representations of Coxeter groups and Hecke algebras, {\it Invent. Math. }{\bf 53}(1979  ),165-184 .

\bibitem[Lu1]{Lu1}  {\sc G.Lusztig},Affine Hecke algebras and their graded version, {\it J.Amer.Math.Soc.  }{\bf 2 }( 1989 ), 599-635.

\bibitem[Lu2]{Lu2}  {\sc G.Lusztig}, Hecke algebras with Unequal Parameters   , {\it CRM Monograph Series}{\bf 18 }(2003 ), Amer.Math. Soc.,Providence RI.



\bibitem[OS]{OS}  {\sc E.Opdam,M.Solleveld },Discrete series characters for affine Hecke algebras   , {\it Acta Math. }{\bf 205 }(2010  ),No.1, 105-187.










\bibitem[Gy]{Gy}  {\sc A.Gyoja},On the existence of a W-graph for an irreducible representation of a Coxeter group , {\it J. Algebra }{\bf 86 }(1984  ), 422-438.



\bibitem[Ra]{Ra}  {\sc A.Ram},Seminormal representations of Weyl groups and Iwahori-Hecke algebras, {\it Proc. London Math. Soc. (3) }{\bf 75}( 1997 ),99-133.

\bibitem[So]{So}  {\sc W.Soergel },The combinatorics of Harish-Chandra bimodules , {\it J. Reine Angew. Math. }{\bf 429}(1992), 49-74.


\bibitem[Xi2]{Xi2}  {\sc N. Xi },Representations of Affine Hecke algebras , {\it Lecture notes in Math. }{\bf 1587 }(1994 ), Springer-Verlag.

\bibitem[Xi1]{Xi1}  {\sc N. Xi},Representations of affine Hecke algebras and based rings of affine Weyl groups, {\it J.Amer.Math.Soc.}{\bf 20 }(2007 ), 211-217.

\end{thebibliography}
\end{document}